\documentclass[twoside,11pt,reqno]{amsart}
\usepackage{amsmath,amssymb,amscd,mathrsfs,epic,wasysym,latexsym,tikz,mathrsfs,cite,hyperref,enumerate,graphicx,amsthm}
\usepackage{pb-diagram,color}
\usepackage{tikz-cd}
\usepackage[all]{xy}
\usepackage{mathdots}
\usepackage{stmaryrd}

\usepackage[top=30mm,right=30mm,bottom=30mm,left=30mm]{geometry}

\makeatletter

\numberwithin{equation}{section}
\allowdisplaybreaks

\newtheorem{Proposition}[equation]{Proposition}
\newtheorem{Lemma}[equation]{Lemma}
\newtheorem{Theorem}[equation]{Theorem}
\newtheorem{Corollary}[equation]{Corollary}

\theoremstyle{definition}  %% makes all of the theorem environments which follow appear in \rm

\newtheorem{Remark}[equation]{Remark}

\let\<\langle
\let\>\rangle

\newcommand\Comment[2][\relax]{\space\par\medskip\noindent%
   \fbox{\begin{minipage}{\textwidth}\textbf{Comment\ifx\relax#1\else---#1\fi}\newline%
        #2\end{minipage}}\medskip
}

\newcommand{\da}{{\downarrow}}

\def\bmu{\text{\boldmath$\mu$}}
\def\bnu{\text{\boldmath$\nu$}}

\def\pmod#1{\text{ }(\text{\rm mod } #1)\,}

\newcommand{\I}{{\mathcal I}}

\newcommand{\End}{\operatorname{End}}
\newcommand{\ind}{\operatorname{ind}}
\newcommand{\im}{\operatorname{im}}

\def\sgn{\mathtt{sgn}}

\newcommand{\res}{\operatorname{res}}
\newcommand{\soc}{\operatorname{soc}\,}

\newcommand{\Z}{\mathbb{Z}}
\newcommand{\bmax}{\mathfrak{b}}

\def\eps{{\varepsilon}}
\def\phi{{\varphi}}

\newcommand{\F}{{\mathbb F}}

\newcommand{\la}{\lambda}

\newcommand{\al}{\alpha}
\newcommand{\be}{\beta}

\def\Si{\mathfrak{S}}
\newcommand{\si}{\sigma}

\newcommand{\Om}{\Omega}

\newcommand{\de}{\delta}

\newcommand{\Mull}{{\tt M}}

\newcommand{\Ker}{\operatorname{Ker}}
\newcommand{\Aut}{{\mathrm {Aut}}}
\newcommand{\Out}{{\mathrm {Out}}}

\newcommand{\Irr}{{\mathrm {Irr}}}
\newcommand{\IBR}{{\mathrm {IBr}}}

\newcommand{\Syl}{{\mathrm {Syl}}}

\newcommand{\Sym}{\operatorname{{\mathtt{Sym}}}}

\newcommand{\C}{{\mathbb C}}

\newcommand{\ZZ}{{\mathbb Z}}
\newcommand{\ZB}{{\mathbf Z}}
\newcommand{\NB}{{\mathbf N}}
\newcommand{\OB}{{\mathbf O}}

\newcommand{\EE}{{\mathcal E}}
\newcommand{\PP}{{\mathbb P}}

\newcommand{\FF}{{\mathbb F}}

\newcommand{\SSS}{{\sf S}}
\newcommand{\AAA}{{\sf A}}

\newcommand{\dar}{{\downarrow}}

\renewcommand{\mod}{\bmod \,}

\newcommand{\edit}[1]{{\color{black} #1}}

\newcommand{\IBr}{{\mathrm {IBr}}}

\def\Par{{\mathscr P}}
\def\Parinv{{\mathscr P}^\AAA}

\def\b{\mathfrak{b}}
\def\k{\Bbbk}

\def\T{{\mathtt T}}
\def\N{{\mathtt N}}
\def\NT{{\mathtt{NT}}}

\def\im{{\mathrm{im}\,}}

\def\into{{\hookrightarrow}}

\def\mod#1{#1\!\operatorname{-mod}}

\renewcommand\L{\mathscr L}

\def\col{{\tt col}}
\def\row{{\tt row}}

\newcommand{\CCC}{{\sf C}}

{\catcode`\|=\active
  \gdef\set#1{\mathinner{\lbrace\,{\mathcode`\|"8000%
  \let|\midvert #1}\,\rbrace}}
}
\def\midvert{\egroup\mid\bgroup}

\colorlet{darkgreen}{green!50!black}
\tikzset{dots/.style={very thick,loosely dotted},
         greendot/.style={fill,circle,color=darkgreen,inner sep=1.5pt,outer sep=0},
         blackdot/.style={fill,circle,color=black,inner sep=1.5pt,outer sep=0},
         graydot/.style={fill,circle,color=gray,inner sep=1.1pt,outer sep=0}
}
\def\greendot(#1,#2){\node[greendot] at(#1,#2){}}
\def\blackdot(#1,#2){\node[blackdot] at(#1,#2){}}
\def\graydot(#1,#2){\node[graydot] at(#1,#2){}}

\newenvironment{braid}{% sets defaults for the braid diagrams
  \begin{tikzpicture}[baseline=6mm,black,line width=1pt, scale=0.32,
                      draw/.append style={rounded corners},
                      every node/.append style={font=\fontsize{5}{5}\selectfont}]%
  }{\end{tikzpicture}
}

\def\Grid(#1,#2){%  draws a coordinate grid inside a braid diagram
  \draw[very thin,gray,step=2mm] (0,0)grid(#1,#2);
  \draw[very thin,darkgreen,step=10mm] (0,0)grid(#1,#2);
}

\newcommand\Tableau[2][\relax]{
  \begin{tikzpicture}[scale=0.5,draw/.append style={thick,black}]
    \ifx\relax#1\relax%
    \else % shade the boxes in #1
      \foreach\box in {#1} { \filldraw[blue!30]\box+(-.5,-.5)rectangle++(.5,.5); }
    \fi
    \newcount\row\newcount\col
    \row=0
    \foreach \Row in {#2} {
       \col=1
       \foreach\k in \Row {
          \draw(\the\col,\the\row)+(-.5,-.5)rectangle++(.5,.5);
          \draw(\the\col,\the\row)node{\k};
          \global\advance\col by 1
       }
       \global\advance\row by -1
    }
  \end{tikzpicture}
}

\newcommand\YoungDiagram[2][\relax]{
  \begin{tikzpicture}[scale=0.5,draw/.append style={thick,black}]
    \ifx\relax#1\relax%
    \else % shade the boxes in #1
    \foreach\box in {#1} {
      \filldraw[blue!30]\box rectangle ++(1,1);
    }
    \fi
    \newcount\row
    \row=0
    \foreach \col in {#2} {
       \draw(1,\the\row)grid ++(\col,1);
       \global\advance\row by -1
    }
  \end{tikzpicture}
}

\newdimen\hoogte    \hoogte=12pt    
\newdimen\breedte   \breedte=14pt  
\newdimen\dikte     \dikte=0.5pt 

\newenvironment{Young}{\begingroup
       \def\vr{\vrule height0.89\hoogte width\dikte depth 0.2\hoogte}
       \def\fbox##1{\vbox{\offinterlineskip
                    \hrule height\dikte
                    \hbox to \breedte{\vr\hfill##1\hfill\vr}
                    \hrule height\dikte}}
       \vbox\bgroup \offinterlineskip \tabskip=-\dikte \lineskip=-\dikte
            \halign\bgroup &\fbox{##\unskip}\unskip  \crcr }
       {\egroup\egroup\endgroup}
\def\Youngdiagram#1{\relax\ifmmode\vcenter{\,\begin{Young}#1\end{Young}\,}\else%
              $\vcenter{\,\begin{Young}#1\end{Young}\,}$\fi}

\newcommand{\twn}[1]{{}^#1\!}

\makeatletter
\newcommand\bigDiamond{\mathop{\mathpalette\bigDi@mond\relax}}
\newcommand\bigDi@mond[2]{%
  \vcenter{\hbox{\m@th
    \scalebox{\ifx#1\displaystyle 2\else1.2\fi}{$#1\Diamond$}%
  }}%
}
\newcommand\bigLozenge{\mathop{\mathpalette\bigL@zenge\relax}}
\newcommand\bigL@zenge[2]{%
  \vcenter{\hbox{\m@th
    \scalebox{\ifx#1\displaystyle 2\else1.2\fi}{$#1\blacklozenge$}%
  }}%
}
\makeatother

\begin{document}

\title[Irreducible restrictions in small characteristics]{{\bf Irreducible restrictions of representations of symmetric and alternating groups in small characteristics}}

\author{\sc Alexander Kleshchev}
\address{Department of Mathematics\\ University of Oregon\\Eugene\\ OR 97403, USA}
\email{klesh@uoregon.edu}

\author{\sc Lucia Morotti}
\address
{Institut f\"{u}r Algebra, Zahlentheorie und Diskrete Mathematik\\ Leibniz Universit\"{a}t Hannover\\ 30167 Hannover\\ Germany} 
\email{morotti@math.uni-hannover.de}

\author{\sc Pham Huu Tiep}
\address
{Department of Mathematics\\ Rutgers University\\ Piscataway\\ NJ 08854, USA} 
\email{tiep@math.rutgers.edu}


\thanks{The first author was supported by the NSF grant DMS-1700905 and the DFG Mercator program through the University of Stuttgart. The second author was supported by the DFG grant MO 3377/1-1 and the DFG Mercator program through the University of Stuttgart. The third author was supported by the NSF (grants DMS-1839351 and DMS-1840702), and the Joshua Barlaz Chair in Mathematics.
This work was also supported by the NSF grant DMS-1440140 and Simons Foundation while all three authors were in residence at the MSRI during the Spring 2018 semester.}

\thanks{The authors are grateful to the referee for careful reading and helpful comments on the paper.}

\begin{abstract}
%In \cite{BK} and \cite{KS} the irreducible restrictions from $\SSS_n$ and $\AAA_n$ to proper subgroups have been classified over fields of characteristic $>3$.  
Building on reduction theorems and  
dimension bounds for symmetric groups obtained in our earlier work, we classify the irreducible restrictions of representations of the symmetric and alternating groups to proper subgroups. Such a classification is known when the characteristic of the ground field is greater than $3$, but the small characteristics cases require a substantially more delicate analysis and new ideas. 
Our results fit into the Aschbacher-Scott program on maximal subgroups of finite classical groups. 
\end{abstract}

\maketitle

%\tableofcontents

\section{Introduction}
\label{SIntro}

Let $\F$ be an algebraically closed field of characteristic $p\geq 0$. In this paper we consider  the following 

\vspace{3mm}
\noindent
{\bf Problem 1.}
{\em
Let $H$ be the symmetric group $\SSS_n$ or the alternating group $\AAA_n$. Classify the pairs $(G,V)$, where $G$ is a subgroup of $H$ and $V$ is an $\F H$-module of dimension greater than $1$ such that the restriction $V\da_G$ is irreducible.
}
\vspace{3mm}

%Of course the $\F H$-module $V$ must be irreducible if $V\da_G$ is irreducible. The assumption $\dim V>1$ is natural since otherwise Problem 1 would require a classification of all finite groups. 

A major application of Problem 1 is to the Aschbacher-Scott program on maximal subgroups of finite classical groups, see \cite{Asch,Scott,Magaard,KlL,BDR} for more details on this.  We point out that for the purposes of these applications, Problem 1 needs to be solved for all almost quasi-simple groups $H$ and $G$, but we do not make any additional assumptions on $G$.

For $p=0$, Problem 1 has been solved in \cite{S}. For $p\geq 5$ and $H=\SSS_n$ (resp. $H=\AAA_n$), Problem 1 has been solved in \cite{BK} (resp. \cite{KS}). But the small characteristics cases $p=2$ and $3$ require a substantially more delicate analysis as well as new ideas, and remained open for a long time. \edit{The first major difficulty is that the submodule structure of certain permutation modules over symmetric groups gets very complicated, making the proof of reduction theorems in \cite{BK} and \cite{KS} much harder for $p=2$ or $3$. The task 
of proving new reduction theorems has now been accomplished in \cite{KMT1,KMT3}, which allows one to mostly reduce the problem to doubly transitive subgroups of $\SSS_n$. The second major difficulty is that   
the techniques employed in \cite{BK} for dealing with doubly transitive subgroups are also inefficient for small $p$. 
%The treatment of the latter subgroups is the second principal bottleneck that renders techniques employed in \cite{BK} unusable when $p=2$ or $3$. A 
So in this paper we develop a new approach, which iteratively pitches the dimension bounds against the shape of the labeling partition $\la$ of
the $\F\SSS_n$-module $D^\la$ in question, relying particularly on dimension bounds obtained recently in \cite{KMT2} and
internal structure of doubly transitive subgroups. This allows us to %overcome this second obstacle in this paper, and 
finally extend the above results to all characteristics.}

From now on we assume that $p>0$. We point out that it is the positive characteristic case that is important for the Aschbacher-Scott program, and that the characteristic $0$ case is equivalent to $p >n$. For the reader's convenience, we will formulate our main results for all characteristics, although they are only new for $p=2, 3$.

Recall that the irreducible $\F \SSS_n$-modules are labeled by the set $\Par_p(n)$ of {\em $p$-regular partitions of $n$}.  If $\la\in \Par_p(n)$, we denote by $D^\la$ the corresponding irreducible $\F\SSS_n$-module. 

The {\em Mullineux involution} 
$$\Par_p(n)\to \Par_p(n),\  \la\mapsto \la^\Mull$$ is defined from $D^{\la^\Mull}\cong D^\la\otimes \sgn$, where $\sgn$ is the $1$-dimensional sign representation. Of course the Mullineux involution is trivial when $p=2$, while for odd $p$ it has several explicit combinatorial descriptions, see \cite{Mull, KBrIII, FK, BO}.

We denote by $\Parinv_p(n)$ the set of all $p$-regular partitions of $n$ such that $D^\la\da_{\AAA_n}$ is reducible. The set of partitions $\Parinv_p(n)$ is well understood---if $p=2$ it is described explicitly in \cite{Benson} (see Lemma~\ref{LBenson} below), while for $p>2$ these are exactly the partitions which are fixed by the Mullineux involution. 

If $\la\in \Parinv_p(n)$ we have 
$$D^\la\da_{\AAA_n}\cong E^\la_+\oplus E^\la_-$$ for irreducible $\F\AAA_n$-modules $E^\la_+\not\cong E^\la_-$. If $\la\not\in\Parinv_2(n)$, we denote 
$$E^\la:=D^\la\da_{\AAA_n}.$$ 
Now, $$
\{E^\la\mid \la\in\Par_p(n)\setminus \Parinv_p(n)\}\cup\{E^\la_\pm\mid \la\in\Parinv_p(n)\}
$$
is a complete set of irreducible $\F\AAA_n$-modules, and 
the only non-trivial isomorphisms among these
are 
$E^\la\cong E^{\la^\Mull}$ for $p>2$ and $\la\in\Par_p(n)\setminus \Parinv_p(n)$. For $\la\in\Par_p(n)$, we will interpret the notation $E^\la_{(\pm)}$ as $E^\la_{\pm}$ if $\la\in\Parinv_p(n)$ and as $E^\la$ otherwise. 

%We refer the reader to \cite[\S11.1]{KBook} for the definitions of combinatorial notions of a {\em residue} of a node and of a {\em normal node}.  

We set $I:=\Z/p\Z$
identified with $\{0,1,\dots,p-1\}$. A {\em node} is an element 
$(r,s)\in\Z_{>0}^2$ (pictorially, the $x$-axis goes down and the $y$-axis goes to the right). We always identify a partition $\la = (\la_1 \geq \la_2 \geq \ldots)$ with its Young diagram $\{(r,s)\in\Z_{>0}^2\mid s\leq \la_r\}$. %Given nodes $A=(r,s)$ and $A'=(r',s')$,  we say that $A$ is {\em above} $A'$ if $r<r'$. 

Given a node $A=(r,s)$, we define its {\em residue}
$
\res A:=s-r\pmod{p}\in I.
$
Let $i \in I$ and $\la\in\Par(n)$. 
%A node $A 
%\in \la$ (resp. $B\not\in\la$) is called {\em $i$-removable} (resp. {\em $i$-addable}) for $\la$ 
%if $\res A=i$ and $\la\setminus\{A\}$ (resp. $\la\cup\{B\}$) is a Young diagram of a partition. 
%A node is called {\em removable} (resp. {\em addable}) if it is $i$-removable (resp. $i$-addable) for some $i$. 
A node $A \in \la$ (resp. $B\not\in\la$) is called {\em removable} (resp. {\em addable}) for $\la$ if 
$\la\setminus\{A\}$ (resp. $\la\cup\{B\}$) is a Young diagram of a partition. 
A removable (resp. addable) node is called {\em $i$-removable} (resp. {\em $i$-addable}) if it has residue $i$.

Labeling the $i$-addable
nodes of $\la$ by $+$ and the $i$-removable nodes of $\la$ by $-$, the {\em $i$-signature} of 
$\la$ is the sequence of pluses and minuses obtained by going along the 
rim of the Young diagram from bottom left to top right and reading off
all the signs.
The {\em reduced $i$-signature} of $\la$ is obtained 
from the $i$-signature
by successively erasing all neighbouring 
pairs of the form $-+$. 
%Note the reduced $i$-signature always looks like a sequence of $+$'s followed by $-$'s.
The nodes corresponding to  $-$'s in the reduced $i$-signature are
called {\em $i$-normal} for $\la$ (or {\em normal nodes of residue $i$}).

A partition $\la\in\Par_p(n)$ is called {\em Jantzen-Seitz} (or {\em JS}) if its top removable node is its only normal node. Equivalently, writing $\la$ in the form $\la=(l_1^{a_1},\dots,l_m^{a_m})$ with $l_1>\dots>l_m$ and $a_1,\dots,a_m>0$, $\la$ is JS of and only if $p$ divides $l_k-l_{k+1}+a_k+a_{k+1}$ for all $1\leq k<m$. 
It is known that the restriction $D^\la\da_{\SSS_{n-1}}$ is  irreducible  if and only if $\la$ is {\em JS}, see \cite{JS,k2}. 
%It is known that a $p$-regular partition $\la$ is JS if and only if it has only one normal node. 

Define the partition
\begin{equation}
\label{ESpin}
\be_n:=
\left\{
\begin{array}{ll}
(n/2+1,n/2-1) &\hbox{if $n$ is even,}\\
((n+1)/2,(n-1)/2) &\hbox{if $n$ is odd.}
\end{array}
\right.
\end{equation}

When $p=2$, the irreducible $\F\SSS_n$-module $D^{\be_n}$ is called the {\em basic spin} module, cf. \cite{Wales}. The irreducible $\F\AAA_n$-module $E^{\be_n}_{(\pm)}$ is also called {\em basic spin}. 
%For the reader's convenience, 
%\color{red}{indeed, because many readers do not know/recall this dimension formula, }\color{black}
%we recall that
Basic spin modules often play a special role, see for example \cite[Theorem 3.9]{KS2Tran} and \cite[Theorem A(vi)]{KMT1}. 
In particular, in Theorems~A,A$'$,B,B$'$ below, we exclude the basic spin case, and then consider it separately in Theorems~C and C$'$. %Similarly for Theorems~C,C$'$,D,D$'$. 

For $m\leq n$ we identify $\SSS_m$ as the subgroup of $\SSS_n$ permuting the first $m$ letters. We also have standard subgroups  
$$\SSS_{m_1,\dots,m_t}\cong \SSS_{m_1}\times\dots\times \SSS_{m_t}\leq \SSS_{m_1+\dots+m_t}\leq 
\SSS_n\quad\text{and}\quad 
\AAA_{m_1,\dots,m_t}:=\SSS_{m_1,\dots,m_t}\cap\AAA_n.
$$ 
whenever $m_1+\dots+m_t\leq n$.

Before stating the main results, in Table I we list the dimensions of the modules which give rise to special cases of irreducible restrictions and indicate when such modules split upon restriction to $\AAA_n$. This table is obtained using \cite[Table 1]{JamesDim}, Lemma \ref{LL2}, \cite[Lemma 2.2]{bkz}, \cite{Benson}, \cite[Lemma 1.21]{BK}, \cite[Theorems 24.1, 24.15, Tables]{JamesBook} and \cite[Lemma 1.8]{KS}. In the table we will always assume $n\geq 5$.

{\small
%\begin{figure}
\[\begin{array}{|c|c|c|c|}
\hline
\la&\text{ Assumptions on $p$ and $n$}&\dim D^\la&\la\in \Parinv_p(n)\\
\hline\hline
\be_n&  p=2&2^{\lfloor (n-1)/2 \rfloor}&\text{iff }n\not\equiv 2\pmod{4}\\
\hline
(n-1,1)&  p\nmid n&n-1&\text{no}\\
\hline
(n-1,1)&  p\mid n&n-2&\text{no}\\
\hline
(n-2,2)&
\begin{array}{l}
p>2,\,\,n\not\equiv 1,2\pmod{p}\\
\text{or }p=2,\,\,n\equiv 3\pmod{4}
\end{array}
&(n^2-3n)/2&\text{no}\\
%&\text{or }p=2,\,\,n\equiv 3\pmod{4}&&\\
\hline
(n-2,2)&
\begin{array}{l}
p>2,\,\,n\equiv 1\pmod{p}
\\
\text{or }p=2,\,\,n\equiv 1\pmod{4}
\end{array}
&(n^2-3n-2)/2&\text{iff }p=2\text{ and }n=5\\
%&\text{or }p=2,\,\,n\equiv 1\pmod{4}&&\\
\hline
(n-2,2)&
\begin{array}{l}
p>2,\,\,n\equiv 2\pmod{p}
\\
\text{or }p=2,\,\,n\equiv 2\pmod{4}
\end{array}
&(n^2-5n+2)/2&\text{no}\\
%&\text{or }p=2,\,\,n\equiv 2\pmod{4}&&\\
\hline
(n-2,2)& p=2,\,\,n\equiv 0\pmod{4}&(n^2-5n+4)/2&\text{no}\\
\hline
(n-2,1^2)& p>2,\,\,p\nmid n&(n^2-3n+2)/2&\text{iff }n=5\\
\hline
(n-2,1^2)& p>2,\,\,p\mid n&(n^2-5n+6)/2&\text{iff }p=3\text{ and }n=6\\
\hline
(5,3)& p=5&21&\text{no}\\
\hline
(6,3)& p=5&21&\text{no}\\
\hline
(3^2,2)& p>5&42&\text{yes}\\
\hline
(3^3)& p>5&42&\text{yes}\\
\hline
(6,5,1)& p=2&288&\text{yes}\\
\hline
(7,5,1)& p=2&288&\text{yes}\\
\hline
(21,2,1)& p\not=2,3,7,23&3520&\text{no}\\
\hline
(21,2,1)& p=7&3267&\text{no}\\
\hline
(21,2,1)& p=23&3269&\text{no}\\
\hline
(21,1^3)& p>3&1771&\text{no}\\
\hline
(22,1^3)& p=5&1771&\text{no}\\
\hline
\end{array}\]

\vspace{1mm}
\centerline{\sc Table I: Certain special modules and their dimensions}
%\end{figure}
}

\vspace{3mm}
{\it Doubly transitive} subgroups $G$ occupy a central place in the solution of Problem~1. Mortimer \cite{Mortimer} studied the problem for the heart $D^{(n-1,1)}$ of 
the natural module of $\SSS_n$ and listed the results in \cite[Table I]{Mortimer}, although leaving two 
unsettled instances. These instances can now be completely analyzed, using \cite[Satz D.2.5]{Hiss} for Ree groups
and \cite{GAP} for $Co_3$. 
We record the updated version of \cite[Table I]{Mortimer} (with $n \geq 5$) in Table II below, where 
the last column describes the conditions on $p$ (if needed) for $D^{(n-1,1)}\da_G$ to be 
irreducible. 

We point out for the purposes of Theorems~B and B$'$ that, except the third line marked with $(\dagger)$, all listed groups are almost simple. Moreover, not all subgroups $G$ satisfying $\CCC^m_r \unlhd G \leq AGL_m(r)$ are doubly transitive, but the list of such doubly transitive groups is known by Hering's Theorem, see \cite{Liebeck}. On the other hand, the subgroups from all other lines are indeed doubly transitive. 

{\small
%\begin{figure}
\[\begin{array}{|c|c|c|c|}
\hline
G&\text{Degree $n$}&\text{Transitivity}&\text{Conditions on $p$}\\
\hline\hline
\SSS_n & n & n & \\ \hline
\AAA_n & n & n-2 & \\ \hline
(\dagger)\ \ \ 
\begin{array}{c} 
\CCC^m_r \unlhd G \leq AGL_m(r),
\\
r\text{ prime }
\end{array}
 & r^m & 2\text{ or }3 & p \neq r \\ \hline
\begin{array}{c}PSL_d(q) \unlhd G \leq P\Gamma L_d(q),\\ d \geq 3 \end{array} & \dfrac{q^d-1}{q-1} & 2 & p \nmid q \\ \hline
\AAA_7 \cong G < GL_4(2) & 15 & 2 & p \neq 2 \\ \hline
Sp_{2m}(2),~m \geq 3 & 2^{m-1}(2^m \pm 1) & 2 & p \neq 2 \\ \hline
\begin{array}{c}
SL_2(q) \unlhd G \leq \Sigma L_2(q), 
\\
2|q
\end{array}
& q+1 & 3 &  \\ \hline
\begin{array}{c}
PSL_2(q) \unlhd G \leq P\Sigma L_2(q),\\
2\nmid q 
\end{array}
& q+1 & 2 & p\neq 2 \\ \hline
\begin{array}{c}PSL_2(q) \unlhd G \leq P\Gamma L_2(q),\\ G \not\leq P\Sigma L_2(q),~2\nmid q 
\end{array} & q+1& 3 &  \\ \hline
\begin{array}{c}
^2B_2(q) \unlhd G \leq \Aut(^2B_2(q)),
\\q>2
\end{array}
 & q^2+1 & 2 & p \nmid (q+1+\sqrt{2q})\\ \hline
\begin{array}{c}
PSU_3(q) \unlhd G \leq P\Gamma U_3(q),
\\q>2
\end{array}
 & q^3+1 & 2 & p \nmid (q+1)\\ \hline
^2G_2(q) \unlhd G \leq \Aut(^2G_2(q)) & q^3+1 & 2 & p \nmid (q+1)(q+1+\sqrt{3q})\\ \hline
M_{24} & 24 & 5 & p \neq 2\\ \hline
M_{23} & 23 & 4 & p \neq 2\\ \hline
M_{22} & 22 & 3 & p \neq 2\\ \hline
M_{12} & 12 & 5 & \\ \hline
M_{11} & 11 & 4 & \\ \hline
M_{11} & 12 & 3 & p \neq 3\\ \hline
PSL_2(11) & 11 & 2 & p \neq 3\\ \hline
HS & 176 & 2 & p \neq 2,3\\ \hline
Co_3 & 276 & 2 & p \neq 2,3\\ \hline
\end{array}\]
\vspace{1mm}
\centerline{\sc Table II: Irreducibility of $D^{(n-1,1)}$ over doubly transitive subgroups}
%\end{figure}
}

\begin{Remark} \label{R180319} %{\rm \cite{}}%{\bf ()}
{\rm 
In \cite[Theorem~B]{KMT1}, we have discovered a new  exceptional family of imprimitive subgroups $G$ for which $D^{(n-1,1)}\da_G$ is irreducible in characteristic $2$. Let $p=2$, $n$ be even, and 
$G \leq \SSS_{n/2} \wr \SSS_2$. Let $B := \SSS_{n/2} \times \SSS_{n/2}$ be the base subgroup of $\SSS_{n/2} \wr \SSS_2$, 
and $G_1$ (resp. $G_2$) be the projection of $G \cap B$ onto the first (resp. second) factor 
$\SSS_{n/2}$ of $B$. Then $D^{(n-1,1)}\da_G$ is irreducible if and only if $n\equiv 2\pmod{4}$, $G$ is transitive on $\{1,2 \ldots,n\}$, $G_1,G_2$ are $2$-transitive subgroups of $\SSS_{n/2}$ over which $D^{(n/2-1,1)}$ is irreducible, and $(D^{(n/2-1,1)} \boxtimes D^{(n/2)})\da_{G \cap B}\not\cong (D^{(n/2)} \boxtimes D^{(n/2-1,1)})_{G \cap B}$. We refer the reader to 
\cite[Section 7]{KMT1}, especially 
\cite[Example 7.24]{KMT1}, for more on this. 
}
\end{Remark}

\vspace{3mm}
For future reference, in Table III, we now list some additional ``non-serial'' (in the sense that they exist only in a finite 
number of degrees $n$) examples of irreducible restrictions of $\F\SSS_n$-modules $D^\la$ to subgroups 
$G < \SSS_n$. In all the cases $G$ acts (at least) $2$-transitively on $\{1,2, \ldots,n\}$ or $\{1,\dots,n-1\}$ as indicated in the table, and when $\{1,\dots,n-1\}$ is indicated we have that $G$ fixes $n$. 
The fact that the cases listed in Table III do yield irreducible restrictions $D^\la\da_G$ is part of the statements of Theorems~A and C.

{\small
%\begin{figure}
\[\begin{array}{|c|c|c|c|c|c|c|}
\hline
\text{Case} &\la \text{ or }\la^\Mull&G& n & \text{$2$-transitive on}&\text{$p$}\\
\hline\hline
\text{(S1)} &(n-2,2) & 
\begin{array}{c}
SL_3(2)\\
P\Gamma L_2(8)\\ 
M_{11}\\
M_{11}\\
M_{12}\\ 
M_{23}\\ 
M_{24}
\end{array} 
 & 
 \begin{array}{c}
 7\\
9\\ 
11\\
12\\
12\\ 
23\\ 
24
\end{array} 
&\{1,\dots,n\}
  & 
  \begin{array}{c}p=5\\
p\not=2,7\\ 
p\not=3,5\\
p=2\\
p\not=5\\ 
p\not=2,3\\ 
p\not=2 
\end{array} 
   \\ \hline

 \text{(S2)} &(n-2,2) & 
 \begin{array}{c}
 M_{11} \\ 
 M_{12}\\ 
 M_{23} \\
 M_{24} 
 \end{array} 
 & 
 \begin{array}{c}
12\\ 
13\\ 
24\\
25
\end{array} 
&
\{1,\dots,n-1\}
  & 
 \begin{array}{c}
 p=2\\ 
 p=11\\ 
 p=11\\
 p=23 
 \end{array} 
   \\ \hline
   
\text{(S3)} &(n-2,1^2) &
\begin{array}{c} 
\SSS_5\\
M_{11}\\
M_{11}\\ 
M_{12}\\
M_{22}, \Aut(M_{22})\\
%\item $(\Aut(M_{22}),22)$; 
M_{23}\\ 
M_{24}
\end{array}
 & 
\begin{array}{c} 
6\\
11\\
12\\ 
12\\
22\\
%\item $(\Aut(M_{22}),22)$; 
23\\
24
\end{array}
&\{1,\dots,n\}
  & 
 \begin{array}{c}
 p=3\\
p\not=2,11\\
p \neq 2,3\\ 
p \neq 2\\
p \neq 2\\
p \neq 2\\
p \neq 2
\end{array}
   \\ \hline

\text{(S4)} &(n-2,1^2) & 
\begin{array}{c} 
M_{11}\\ 
M_{11}\\
M_{12}\\ 
M_{22}, \Aut(M_{22}) \\
%, $(\Aut(M_{22}),23)$ for $p=23$, 
M_{23}\\ 
M_{24}
\end{array}
 & 
\begin{array}{c} 
12\\ 
13\\ 
13\\ 
23\\
24\\ 
25
\end{array} 
&
\{1,\dots,n-1\}
& \begin{array}{c} 
p=3\\ 
p=13\\ 
p=13\\
p=23\\
p=3\\ 
p=5\end{array}
   \\ \hline
   
\text{(S5)}&(14,1^2)&\CCC_2^4 \rtimes \AAA_7&16&
\{1,\dots,16\}&p\neq 2   \\ \hline

\text{(S6)} &(15,1^2) & 
\CCC_2^4 \rtimes \AAA_7&17&
\{1,\dots,16\} & p=17
   \\ \hline

 \text{(S7)} &(5,3) & 
AGL_3(2)&8
&\{1,\dots,8\} & p=5
   \\ \hline

\text{(S8)} &(6,3) & 
AGL_3(2)&9&
\{1,\dots,8\} & p=5
   \\ \hline

\text{(S9)}&(21,2,1)&M_{24}&24&
\{1,\dots,24\}&p \neq 2,3   \\ \hline

\text{(S10)}&(21,1^3)&M_{24}&24&
\{1,\dots,24\}&p \neq 2,3   \\ \hline

\text{(S11)} &(22,1^3) & 
M_{24}&25&
\{1,\dots,24\} & p=5
   \\ \hline

   \text{(S12)}&(3,2)&\CCC_5 \rtimes \CCC_4&5&
\{1,\dots,5\}&p= 2 \\ \hline

\text{(S13)}&(4,2)&\SSS_5&6&
\{1,\dots,6\}&p= 2 \\ \hline

\text{(S14)}&(6,4)&
 %\begin{array}{c}
 \SSS_6, M_{10}, \Aut(\AAA_6)% \text{ or }\\
% \Aut(\AAA_6)\end{array}
 &10&
\{1,\dots,10\}&p= 2 \\ \hline

%\text{(S15)}&(7,5)&M_{12}&12&\{1,\dots,12\}&p= 2 \\ \hline

\end{array}\]

\vspace{1 mm}
\centerline{\sc Table III: Non-serial examples of irreducible restrictions from $\SSS_n$}
}

\vspace{3mm}

Note that in the cases (S12)-(S14), we have $(\la,p)=(\be_n,2)$, i.e. these cases are concerned with restrictions of basic spin modules.

For future reference, in Table IV, we now list some ``non-serial'' examples of irreducible restrictions of $\F\AAA_n$-modules $E^\la_\pm$ with $\la\in\Parinv_p(n)$ to subgroups 
$G < \AAA_n$. In all but the case (A17), $G$ acts (at least) $2$-transitively on $\{1,2, \ldots,m\}$ as indicated in the table (and fixes $n$ if
$m=n-1$).
%each of $i \in \{m+1,\ldots,n\}$. 
The two additional  conditions in Table IV are as follows:
$$
\begin{array}{rl}
(\spadesuit) &\begin{array}{l}\text{only one of $E^{(5,4)}_{\pm}$, namely the one whose Brauer  character takes}\\ \text{value $-1$ at elements of order $9$ in 
$SL_2(8)$, is irreducible over $G$.}
\end{array}
\\
& 
\\
(\spadesuit\spadesuit)&
\begin{array}{l}
\soc(G)\ \text{acts on $\{1,2, \ldots,6\}$ and $\{7,8, \ldots,12\}$ via two inequivalent}
\\
 \text{$2$-transitive actions.} 
\end{array}
\end{array}
$$

%Finally, in Case {\bf A8}, the item with $^{(\clubsuit)}$ signifies that only one of $E^{\be_9}_{\pm}$, namely the one whose Brauer character takes value $-1$ at elements of order $9$ in $SL_2(8)$, is irreducible over $G$. The fact that the cases listed in Table IV do yield irreducible restrictions $D^\la\da_G$ is part of the statements of Theorems~A$'$ and B$'$. 

{\small
%\begin{figure}
\[\begin{array}{|c|c|c|c|c|c|c|c|}
\hline
\text{Case} &\la &G& n &\text{$2$-transitive\,\,on}&\text{ $p$}&\begin{array}{c}\text{Additional}\\ \text{conditions}
\end{array}
\\
\hline\hline
\text{(A1)} &(6,5,1) & M_{12} & 12 &  \{1,\dots,12\} & p=2 
& \\ \hline

\text{(A2)} &(7,5,1) & M_{12} & 13 &  \{1,\dots,12\} & p=2 
& \\ \hline

\text{(A3)} &(4,1^2) & \AAA_5 & 6 &  \{1,\dots,6\} & p=3 
& \\ \hline

\text{(A4)} &(3^3) & P\Gamma L_2(8) & 9 &  \{1,\dots,9\} 
& p > 5 & \\ \hline 

\text{(A5)} &(3^2,2) & AGL_3(2) & 8 &  \{1,\dots,8\} & p>5 
& \\ \hline

\text{(A6)} &(3^3) & AGL_3(2) & 9 &  \{1,\dots,8\} & p> 5 
& \\ \hline

\text{(A7)} &(3,2) & 
\CCC_5 \rtimes \CCC_2 & 5 &  \{1,\dots,5\} & p=2 
& \\ \hline

\text{(A8)} &(5,4) &
ASL_2(3),\CCC_3^2 \rtimes {\mathsf Q}_8 & 9 &  \{1,\dots,9\} & p=2 
& \\ \hline

\text{(A9)} &(5,4) & 
SL_2(8),P\Gamma L_2(8) & 9 &  \{1,\dots,9\} & p=2 
& (\spadesuit)
\\ \hline

\text{(A10)} &(6,4) & M_{10} & 10 &  \{1,\dots,10\} & p=2 & \\ \hline

\text{(A11)} &(6,5) & 
M_{11} & 11 &  \{1,\dots,11\} & p=2 
& \\ \hline

\text{(A12)} &(7,5) & 

M_{11}, 
M_{12} & 12 &  \{1,\dots,12\} & p=2 
& \\ \hline

%\text{(A8)} &(7,5) & M_{12} & 12 & & \{1,\dots,12\} & p=2 & \\ \hline

%\text{(A11)} &(5,4) & SL_2(8) \text{ or }P\Gamma L_2(8) & 9 & & \{1,\dots,9\} & p=2 & \\ \hline

\text{(A13)} &(4,3) & \AAA_5 & 7 &  \{1,\dots,6\} 
& p =2 & \\ \hline 

\text{(A14)} &(5,3) & \AAA_5,\SSS_5 & 8 &  \{1,\dots,6\} 
& p =2 & \\ \hline  

\text{(A15)} &(6,5) & M_{10} & 11 &  \{1,\dots,10\} 
& p =2 & \\ \hline  

\text{(A16)} &(7,5) & \SSS_6,M_{10},\Aut(\AAA_6) & 12 &  \{1,\dots,10\} 
& p =2 & \\ \hline 

\text{(A17)} &(7,5) & \SSS_6,M_{10},\Aut(\AAA_6) & 12 & 
% \begin{array}{c}\{1,2, \ldots,6\}\\ \{7,8, \ldots,12\}\end{array}
& p =2 & 
(\spadesuit\spadesuit)
\\ \hline 
 
\text{(A18)} &(7,5) & M_{11} & 12 &  \{1,\dots,11\} 
& p =2 & \\ \hline  

\end{array}\]

\vspace{1mm}
\centerline{\sc Table IV: Non-serial examples of irreducible restrictions from $\AAA_n$}
}

\vspace{3mm}

Note that in the cases (A7)-(A18), we have $(\la,p)=(\be_n,2)$, i.e. these cases are concerned with restrictions of basic spin modules.

We now describe the main results of the paper. In all theorems, the subgroups $G$ are listed up to $\SSS_n$-conjugation. We note that
$\SSS_n$-conjugate subgroups of $\AAA_n$ need not be $\AAA_n$-conjugate, and it may happen, as it does in case (A9) listed in 
Table IV, that one conjugate acts irreducibly while the other does not on an $\F\AAA_n$-module; such instances are specified explicitly
in our results. 
The case of the basic spin module\footnote{As pointed out by the anonymous referee, incidentally, the phenomenon of spin modules in 
characteristic $2$ giving rise to long chains of subgroups with irreducible restriction has also been observed in the context of symplectic groups over algebraically closed fields of characteristic $2$ in \cite{BMT}.},
excluded in Theorems~A and A$'$, will be considered separately in 
Theorems~C and C$'$. %This case will also be allowed in Theorems~C and C$'$. 

%\begin{MainTheorem}\label{TSn}

\vspace{3mm}
\noindent
{\bf Theorem A.} {\em 
Let $n \geq 5$, %$H=\SSS_n$, 
$G < \SSS_n$, and $\la\in\Par_p(n)$ be such that $\dim D^\la>1$. Exclude the basic spin case $(p,\la)=(2,\be_n)$. %, i.e. exclude the basic spin module. 
Then $D^\la\da_G$ is irreducible if and only if one of the following holds:
\begin{enumerate}[\rm(i)]
%\item $\la$ or $\la^\Mull$ equals $(n)$, i.e. $\dim D^\la=1$. 

\item $\la\not\in\Parinv_p(n)$ and $G=\AAA_n$.

\item $\la$ or $\la^\Mull$ equals $(n-1,1)$, $G$ is $2$-transitive,  and $(G,n,p)$ is as in Table II. 

\item $p=2$, $n\equiv 2\pmod{4}$, $\la=(n-1,1)$, and $G\leq\SSS_{n/2}\wr\SSS_2$ is as in Remark ~\ref{R180319}.

%\item $\la$ or $\la^\Mull$ equals $(n-2,2)$, $G$ is primitive, and $(G,n)$ is one of the following: 
%\begin{enumerate}[\rm(a)]
%\item $(SL_3(2),7)$ for $p=5$;
%\item $(P\Gamma L_2(8),9)$ for $p\not=2,7$; 
%\item $(M_{11},11)$ for $p\not=3,5$; 
%\item $(M_{11},12)$ for $p=2$; 
%\item $(M_{12},12)$ for $p\not=5$; 
%\item $(M_{23},23)$ for $p\not=2,3$; 
%\item $(M_{24},24)$ for $p\not=2$.
%\end{enumerate}

\item $p\not=2$, $\la$ or $\la^\Mull$ equals $(n-2,1^2)$, $n=2^m$ for some $m\geq 3$ and $G=AGL_m(2)<\SSS_n$ via its natural action on the points of $\F_2^m$. 

%$G$ is primitive, and $(G,n)$ is one of the following: 
%\begin{enumerate}[\rm(a)]
%\item $(AGL_m(2),2^m)$; 
%\item $(\SSS_5,6)$ for $p=3$;
%\item $(M_{11},11)$ for $p\not=11$;
%\item $(M_{11},12)$ with $p \neq 3$; 
%\item $(M_{12},12)$;
%\item $(\CCC_2^4 \rtimes \AAA_7,16)$; 
%\item $(M_{22}\ \text{or}\ \Aut(M_{22}),22)$;
%\item $(M_{23},23)$;
%\item $(M_{24},24)$.
%\end{enumerate}

%\item $p=5$, $n=8$, $\la$ or $\la^\Mull$ equals $(5,3)$, and $G=AGL_3(2)$.

%\item $p\not=2,3$, $n=24$, $\la$ or $\la^\Mull$ equals $(21,2,1)$ or $(21,1^3)$, and $G=M_{24}$.

\item $\la$ is JS and $G=\SSS_{n-1}$.

\item $\la$ is JS, $\la\not\in\Parinv_p(n)$, and $G=\AAA_{n-1}$.

\item $n\equiv 0\pmod{p}$, $\la$ or $\la^\Mull$ equals $(n-1,1)$, $G$ is a $2$-transitive subgroup of $\SSS_{n-1}$, and $(G,n-1,p)$ is as described in Table II.

%\item $\la$ or $\la^\Mull$ equals $(n-2,2)$, $G$ is a primitive subgroup of $\SSS_{n-1}$, and $(G,n)$ is one of the following: 
%\begin{enumerate}[\rm(a)]
%\item $(M_{11},12)$ for $p=2$; 
%\item $(M_{12},13)$ for $p=11$; 
%\item $(M_{23},24)$ for $p=11$; 
%\item $(M_{24},25)$ for $p=23$. 
%\end{enumerate}

\item $p\not=2$, $\la$ or $\la^\Mull$ equals $(n-2,1^2)$, $n=2^m+1\equiv 0\pmod{p}$ for some $m\geq 2$, and $G=AGL_m(2)<\SSS_{n-1}$ embedded 
via its natural action on the points of $\F_2^m$.

\item $(\la,G,n,p)$ is as in one of the cases (S1)-(S11) in Table III.
\end{enumerate}
}
%\end{MainTheorem}

%\vspace{1mm} For $1<k<n$, we denote $\AAA_{n-k,k}:=\AAA_n\cap (S_{n-k}\times \SSS_k).$

\vspace{3mm}
\noindent
{\bf Theorem A$'$.} 
{\em
Let $n \geq 5$, $G < \AAA_n$, and $V$ be a non-trivial irreducible $\F\AAA_n$-module. If $p=2$ assume that $V$ is not basic spin. 
Then $V\da_G$ is irreducible if and only if one of the following holds:
\begin{enumerate}[\rm(i)]
\item $V\cong E^\la$ with $\la\not\in\Parinv_p(n)$ and $(\la,G,n,p)$ is as in Theorem A.

\item $V\cong E^\la_\pm$ with  $\la\in\Parinv_p(n)$ and one of the following holds:
\begin{enumerate}[\rm(a)]
\item $G=\AAA_{n-1}$ and $\la$ is JS or it has exactly two normal nodes, both of residue different from $0$.

\item $G=\AAA_{n-2}$ or $\AAA_{n-2,2}$ and $\la$ is JS.

\item $(\la,G,n,p)$ is as in one of the cases (A1)-(A6) in Table IV. 
\end{enumerate}
\end{enumerate}
}
%\end{MainTheorem}

\vspace{3mm}
A group $G$ is called {\em almost quasisimple} if $S \,\unlhd\, G/\ZB(G) \leq \Aut(S)$ for some non-abelian simple group $S$. 
In a number of applications, irreducible restrictions to quasisimple subgroups $G$ %$G$ of $H = \AAA_n$, $\SSS_n$  that are {\it almost quasisimple},that is, when $S \lhd G/\ZB(G) \leq \Aut(S)$ for some non-abelian simple group $S$, 
are of most interest. In the next two theorems we deal just with  %consider 
this  important special case. 
%, and 
%give a complete solution of Problem 1 for the pairs $(G,H)$, where $G$ is almost quasisimple and $H = \AAA_n,\SSS_n$. 

%deal with the case irreducible restriction problem for all almost quasisimple subgroups. 

\vspace{3mm}
\noindent
{\bf Theorem B.} 
%\begin{MainTheorem}\label{TAQSSn}
{\em 
Let $n \geq 5$, $G < \SSS_n$ be an almost quasisimple subgroup, and $\la\in\Par_p(n)$ be such that $\dim D^\la >1$. 
Exclude the basic spin case $(p,\la)=(2,\be_n)$. 
Then $D^\la\da_G$ is irreducible if and only if one of the following holds:
\begin{enumerate}[\rm(i)]
%\item $\la$ or $\la^\Mull$ equals $(n)$, i.e. $\dim D^\la=1$. 

\item $\la\not\in\Parinv_p(n)$ and $G=\AAA_n$.

\item $\la$ or $\la^\Mull$ equals $(n-1,1)$, $G$ is $2$-transitive, and $(G,n,p)$ is as described in Table II, excluding the third line marked with $(\dagger)$.

%\item $p=2$, $n\equiv 2\pmod{4}$, $\la=(n-1,1)$ and $G\leq\SSS_{n/2}\wr\SSS_2$ is as in \cite[Theorem B]{KMT1}.

\item $\la$ is JS and $G=\SSS_{n-1}$.

\item $\la$ is JS, $\la\not\in\Parinv_p(n)$ with $\la\not=\be_n$ if $p=2$ and $n\equiv 2\pmod{4}$, and $G=\AAA_{n-1}$.

\item $n\equiv 0\pmod{p}$, $\la$ or $\la^\Mull$ equals $(n-1,1)$, $G$ a $2$-transitive subgroup of $\SSS_{n-1}$, and $(G,n-1,p)$ is as described in Table II, excluding the third line marked with~$(\dagger)$.

\item $(\la,G,n,p)$ is as in one of the cases (S1)-(S4) or (S9)-(S11)
 in Table III.

%\item $p=5$, $n=9$, $\la$ or $\la^\Mull$ equals $(6,3)$ and $G=AGL_3(2)<\SSS_8$.

\end{enumerate}
}
%\end{MainTheorem}

\vspace{3mm}
\noindent
{\bf Theorem B$'$.}
%\begin{MainTheorem}\label{TAQSAn}
{\em 
Let $n \geq 5$, $H=\AAA_n$, $G < H$ be almost quasisimple, and $V$ be a non-trivial irreducible $\F\AAA_n$-module. 
If $p=2$ assume that $V$ is not basic spin. 
Then $V\da_G$ is irreducible if and only if one of the following holds:
\begin{enumerate}[\rm(i)]
\item $V\cong E^\la$ with $\la\not\in\Parinv_p(n)$ and $(\la,G,n,p)$ is as in Theorem B. 

\item $V\cong E^\la_\pm$ with  $\la\in\Parinv_p(n)$ and one of the following holds:
\begin{enumerate}[\rm(a)]
\item $G=\AAA_{n-1}$ and $\la$ is JS or it has exactly two normal nodes, both of residue different from 0.

\item $G=\AAA_{n-2}$ or $\AAA_{n-2,2}$ and $\la$ is JS.

\item $(\la,G,n,p)$ is as in one of the cases (A1)-(A4)
 in Table IV.

%\item $p=2$, $n\not\equiv 2\pmod{4}$, $V \cong E^{\be_n}_{\pm}$, $G$ is primitive and $(G,n)$ is as in one of the cases (A9)-(A12) in Table IV 

\end{enumerate}
\end{enumerate}
}
%\end{MainTheorem}

\vspace{3mm}
For basic spin modules in characteristic 2 we have the following two results.

\vspace{3mm}
\noindent
{\bf Theorem C.} 
%\begin{MainTheorem}
%{\cite[Theorem C]{KMT1}}
%\label{SpinSn}
{\em 
Let $n\geq 5$, $p=2$, 
and $G< \SSS_n$ be a proper subgroup of $\SSS_n$ such that $D^{\be_n}\da_G$ is irreducible. Then one of the following happens:

\begin{enumerate}[{\rm (i)}]
\item
$G\leq \SSS_{n-k}\times\SSS_{k}$ with $n-k$ and $k$ odd. In fact, 
$$D^{\be_n}\da_{\SSS_{n-k}\times\SSS_{k}}\cong D^{\be_{n-k}}\boxtimes D^{\be_{k}}$$
is indeed irreducible.

\item
$G\leq \SSS_a\wr\SSS_b$ with $n=ab$, $a,b\in\Z_{>1}$ and $a$ is odd. Moreover if $b>2$ then $G\not\leq\SSS_{a}\times\dots\times \SSS_{a}$. In fact,   
$$D^{\be_n}\da_{\SSS_a\wr\SSS_b}\cong D^{\be_a}\wr D^{\be_b}
$$
is indeed irreducible.

\item
$G$ is primitive, in which case $D^{\be_n}\da_G$ is irreducible if and only if one of the following happens:
\begin{enumerate}[{\rm (a)}]
\item $n\equiv 2\pmod{4}$ and $G=\AAA_n$;

\item $(G,n)$ is as in one of the cases (S12)-(S14) in Table III.
\end{enumerate}
\end{enumerate}
Moreover, if $G$ is almost quasi-simple then $D^{\be_n}\da_G$ is irreducible if and only if one of the following holds:
\begin{enumerate}[{\rm (1)}]
\item $n$ is even and $G=\SSS_{n-1}$.

\item $G$ is primitive, and one of the following holds: 
\begin{enumerate}[{\rm (a)}]
\item $n\equiv 2\pmod{4}$ and $G=\AAA_n$;
%\item $n=5$, $G=\CCC_5\rtimes \CCC_4$;

\item $(G,n)$ is as in one of the cases (S13),(S14) in Table III.
\end{enumerate}
\end{enumerate}

}
%\end{MainTheorem}

\vspace{3mm}
For restrictions of basic spin modules for  $\AAA_n$ we have the following analogous result: 

\vspace{3mm}
\noindent
{\bf Theorem C$'$.}
%\begin{MainTheorem}\label{SpinAn}
{\em 
Let $n\geq 5$, $p=2$ and $G< \AAA_n$. If  $E^{\be_n}_{(\pm)}\da_{G}$ is irreducible then one of the following holds:
\begin{enumerate}[{\rm (i)}]
\item $G\leq \AAA_{n-k,k}$ for some $1\leq k<n$, and either 
$n\equiv 0\pmod{4}$ and $k$ is odd, or $n\not\equiv 2\pmod{4}$ and $k\equiv 2\pmod{4}$. Moreover, in all of these cases $E^{\be_n}_\pm\da_{\AAA_{n-k,k}}$ is indeed irreducible.

\item $G\leq (\SSS_a\wr \SSS_b)\cap \AAA_n$ for $a,b>1$ with $n=ab$,  and either $a$ is odd or $a\equiv 2\pmod{4}$ and $b=2$. Moreover, in all of these cases $E^\la_{(\pm)}\da_{(\SSS_a\wr \SSS_b)\cap \AAA_n}$ is indeed irreducible.

\item $G$ is primitive, in which case $E^{\be_n}_{(\pm)}\da_G$ is irreducible if and only if 
$(G,n)$ is as in one of the cases (A7)-(A12) in Table IV.
\end{enumerate} 
Moreover, if $G$ is almost quasi-simple then $E^{\be_n}_{(\pm)}\da_{G}$ is irreducible if and only if one of the following holds:
\begin{enumerate}[{\rm (1)}]

\item
$n\not\equiv 2\pmod{4}$ %, $G$ is not primitive,
and one of the following happens:
\begin{enumerate}
\item[\rm (a)] $4|n$ and $G = \AAA_{n-3,2,1}$ or $\AAA_{n-2,1,1}$.
\item[\rm (b)] %$n \not\equiv 2 (\bmod 4)$, 
$G = \AAA_{n-2,2}$. %Back to (a).
\item[\rm (c)] $n \equiv 0,3 \pmod{4}$ and $G = \AAA_{n-1}$.
\item[\rm (d)] $(G,n)$ is as in one of the cases (A9), (A11)-(A18) %(A13)-(A18) 
in Table IV.
\end{enumerate}
\item $(G,n)=(M_{10},10)$ (case (A10) of Table IV).
\end{enumerate}
}

\vspace{3mm}

\begin{Remark}\label{R180319_2}
We point out that \cite[Theorem C]{KMT1} contains an inaccuracy:   since $M_{12}<\AAA_{12}$ and $\be_{12}\in\Parinv_2(12)$, the restriction $D^{\be_{12}}\da_{M_{12}}$ is reducible, and so this case does not appear in Theorem C above. 
However, it does appear in Theorem~C$'$(iii) and (1)(d) as part of the case (A12). 

There is a similar inaccuracy in \cite[Main Theorem]{BK}: let $G=M_{11}\leq\SSS_{11}$ and $p=5$. Then $D^{(9,2)}\da_G$ is reducible by case (iii) of \cite[Main Theorem]{BK} and so $D^{(10,2)}\da_G$ is also reducible.
\end{Remark}

\edit{We point out that the results proved 
in \cite{KMT1,KMT3} that reduce the problem mostly to the treatment of doubly transitive groups 
do not depend on the Classification of Finite Simple Groups ({\sf {CFSG}}). However, the main results of this paper 
depend on {\sf {CFSG}} as follows: (i) our treatment of doubly transitive subgroups relies on their explicit list, see \cite{Cameron}, which is a consequence of {\sf {CFSG}}, and (ii) the treatment of ``non-generic'' situation in Section~\ref{SOD} uses the list of simple subgroups of $\SSS_n$ of large order (Proposition \ref{PSimple}) which also relies on 
{\sf {CFSG}}.}

We now describe the key ingredients of our proof and the organization of the paper. We will exploit various dimension bounds for irreducible representations of symmetric groups, especially new lower bounds obtained in \cite{KMT2}, see 
Theorems~\ref{TBound1} and \ref{TBound2}. Further dimension bounds and branching results are collected in the preliminary Section~\ref{SPrel}. 

Reduction theorems established in \cite{KMT1,KMT3} allow us to assume in many situations that the subgroup $G$ is primitive or even doubly transitive. Those subgroups tend to have a relatively large order, and we contrast order bounds with dimension bounds  in Section~\ref{SOD}, particularly to resolve the 
``non-generic'' situation where the module is either basic spin or not extendible to $\SSS_n$. 

In Sections~\ref{SSmall}--\ref{SLG3} we deal with doubly transitive subgroups $G \leq \SSS_n$. Given the well-known solution of Problem 1 in the case 
$(G,H) = (\AAA_n,\SSS_n)$, we will assume that $G \not\geq \AAA_n$.  Such subgroups $G$ are subdivided into the following four families, corresponding to the structure of the socle $\soc(G)$ and its action on $\{1,2, \ldots,n\}$: 

\begin{enumerate}
\item[\rm (A)] $\soc(G)$ is elementary abelian subgroup;

\item [\rm (B)] $\soc(G) \cong PSL_m(q)$ (is non-abelian simple)  acting on $n = (q^m-1)/(q-1)$ $1$-dimensional subspaces of $ \F_q^m$;

\item [\rm (C)] $G\cong Sp_{2m}(2)$, $m \geq 3$, acting on $n = 2^{m-1}(2^m + (-1)^\de)$ quadratic forms on $\F_2^{2m}$ of the given Witt defect $\de \in \{0,1\}$; % (we interpret $2^m+\eps$ for $\eps = \pm$ as $2^m+\eps 1$);

\item [\rm (D)] all other doubly transitive subgroups; the subgroups from this class will be called {\em small doubly transitive subgroups}. 
\end{enumerate}

The small doubly transitive subgroups of (D) are handled in Section~\ref{SSmall}, largely relying on the aforementioned results on dimension bounds, 
branching rules to Young subgroups, and available information about modular representations of $H$ and $G$.

In Section \ref{SAff}, we handle the family (A) of affine permutation subgroups. Here, the key technical result is Proposition \ref{P110417_5}
that identifies the $\SSS_n$-modules that have no (nonzero) invariants over $\soc(G) \cong \CCC_r^m$, whose proof in turn relies on representation theory of 
affine general linear group $AGL_m(r)$ and the new branching recognition result Proposition \ref{C080417}.

The families (B) and (C) are handled in Sections~\ref{SLG1} and \ref{SLG3}, respectively. 
We note that these large doubly transitive groups are the main reason why the methods of \cite{BK} and \cite{KS} break down when one tries to 
employ them in small characteristics $p=2,3$. 
%The key novel idea that allows us to overcome this principal obstacle is an iterative use of the dimension bounds and Theorem \ref{TSub} to show that the first part of the $p$-regular partition $\lambda$ labeling the module $V$ is very large, in fact at least $n-3$. 

The heart of the proof is to show that if the irreducible $\F\SSS_n$-module $D^\la$ remains irreducible over such a subgroup $G$ from 
the families (B) and (C) , then 
the longest part $\la_1 = n-\ell$ of $\la$ is very large, in fact, $\ell \leq 3$ most of the time. 
We will do this in a sequence of steps. 

First, using the obvious bound $\dim V \leq |G|^{1/2}$  for 
any irreducible $G$-module $V$ and Lemma \ref{TBound3}, we show in Propositions
\ref{PRedSL} and \ref{PRedSp} that 
$$\dim D^\la \leq n^{\frac{1}{2}\log_2n+1}.$$ Then an application of Proposition \ref{PDim1} implies that
\begin{equation}\label{eq-b1}
  \ell = O(\log n).
\end{equation}
%In what follows, all the logarithms will have base $2$.

Next, we choose some $L$ such that $2\ell \leq L < n$. Considering $G\cap\SSS_{n-L}$ and 
using Theorem \ref{TSub} and Propositions \ref{PRedSL}(iii) and \ref{PRedSp}(ii), we prove that %an upper bound
\begin{equation}\label{eq-b2}
  \dim D^\la = n^{O(k)}
\end{equation}
for some $k = O(\log \ell)$. On the other hand, Theorem \ref{TBound1} yields a lower bound
\begin{equation}\label{eq-b3}
  \dim D^\la > O(n^\ell/(\ell!)^2).
\end{equation} 
Given \eqref{eq-b1}, we can show that \eqref{eq-b3} contradicts \eqref{eq-b2}, unless $\ell$ is small. An iterative application of this argument 
will allow us to show that $\ell \leq 3$.
The remaining possibilities for $\la$ are ruled out using more precise information about $D^\la$.

%See Section \ref{SSSketch} for a sketch of this and other techniques employed in the treatment of these two families.
%A variety of techniques is used, most notably the new dimension bounds from \cite{KMT2}, reduction theorems from \cite{KMT1,KMT3}, 
%which sometimes allow us to reduce to $3$-homogeneous subgroups, branching results, structure of permutation modules, 
%\color{red} Tiep, please list the techniques used ... and perhaps add more to explain the main ideas in 3-4 sentences... 
%I feel I have not quite said enough in the previous  paragraphs---probably those need to be revised, too. \color{black} 

Finally, the main theorems are proved in Section~\ref{SMain}. First, we use the main results of \cite{KMT1,KMT3} to reduce to subgroups doubly transitive on $\{1,\dots,n\}$ or doubly transitive on $\{1,\dots,n-1\}$ and fixing $n$. The results of the previous sections then allow us to complete the proofs of Theorems A, A$'$.
% and the first part of Theorem C$'$ (the first part of Theorem C had already been proved in \cite{KMT1}). Theorems A and A$'$ are obtained by first classifying irreducible restrictions to transitive groups and then using reduction theorems and branching to study intransitive cases. 
%Tiep, please write a couple of words summarizing the additional tools used i this section. 
%Or is it just collecting everything that has already been done? Apparently there is more than that, 
%at least for the proof of Theorems E,F, and the reader needs to see the main ideas in a couple of sentences.   ...\color{black} 
The proof of Theorem B requires a delicate argument to rule out the possibility of Theorem A(iii) for almost quasisimple groups.
The proof of Theorem C$'$ combines classifications of irreducible restrictions to maximal imprimitive subgroups (from \cite{KMT3}) and to primitive subgroups (obtained in Sections~\ref{SSmall}--\ref{SLG3}). 
After that we need to handle the case when $\soc(G) \cong \AAA_m$ has only orbits of length $1$ and $m$ on
$\{1,2, \ldots,n\}$. % (and thus extending results of \cite{KST}). 

\section{Preliminary results}
\label{SPrel}
\subsection{Generalities}
Throughout the paper we work over a fixed algebraically closed ground field $\F$ of characteristic $p>0$. %We do not yet assume that $p=2$ or $3$ but will do this when necessary. 
Let $G,H$ be arbitrary finite groups, $V,V'$ be $\F G$-modules,  and $W$ be an $\F H$-module. The following notation is used throughout the paper:
$$
\begin{array}{ll}
\text{$V{\da}_H$ or $V{\da}^G_H$} & \text{the {\em restriction} of $V$ from $G$ to $H$ (if $H\leq G$);}\\
\text{$\ind^GW$ or $\ind^G_HW$} & \text{the {\em induction} of $W$ from $H$ to $G$  (if $H\leq G$);}
\\
\text{$V\boxtimes W$} & 
\begin{array}{l}
\text{\hspace{-1.5 mm}the outer tensor product of $V$ and $W$ (this is a module} 
\\
\text{\hspace{-1.5 mm}over $G\times H$);}
\end{array}
\\
\text{$V\otimes V'$} & 
\begin{array}{l}
\text{\hspace{-1.5 mm}the inner tensor product of $V$ and $V'$ (this is a module} 
\\
\text{\hspace{-1.5 mm}over $G$);}
\end{array}
\\
\text{$V^G$} & \text{the space of {\em $G$-invariant vectors} in $V$;}
\\
\text{$\F_G$} &\text{the trivial $\F G$-module;}
\\
\text{$\Irr_\F(G)$}&\text{a complete set of irreducible $\F G$-modules;}
\\
\text{$\IBr_p(G)$}&\text{the set of irreducible $p$-Brauer characters of $G$;}
\\
\text{$\bmax_p(G)$} &\text{the maximal dimension of an irreducible $\F G$-module;}
\\
\text{$\bmax(G)$} &\text{the maximal dimension of an irreducible $\C G$-module;}
\\
\text{$d(G)$} &
\begin{array}{l}
\text{\hspace{-1.5mm}the minimal degree of a non-linear irreducible complex}
\\
\text{\hspace{-1.5mm}character of $G$ (if such exists);}
\end{array}
\\
\text{$P(G)$}&\text{the smallest index of a (proper) maximal subgroup of $G$;}
\\
\text{$\Par_p(n)$}&\text{the set of $p$-regular partitions of $n$;}
\\
\text{$\Parinv_p(n)$}&\text{the set of $\la \in \Par_p(n)$ such that $D^\la\da_{\AAA_n}$ is reducible;}
\\
\text{$h(\la)$}&\text{the number of nonzero parts in the partition $\la$.}
\end{array}
$$

\edit{Let $0\leq \ell<n$. We denote
\begin{align*}
\Par^{(\ell)}(n)&:=\{\la \vdash n \mid \la_1\geq n-\ell\},
\\ 
\L^{(\ell)}(n)&:=\{\la\in \Par_p(n)\mid \la\ \text{or}\ \la^\Mull\ \text{belongs to}\  \Par^{(\ell)}(n) \cap \Par_p(n)\}.
%\Par^{(\ell)}_p(n):=\Par_p(n)\cap \Par^{(\ell)}(n).
\end{align*}
Given a partition $\mu=(\mu_1,\mu_2,\dots)$ of $\ell$ with $\mu_1\leq n-\ell$, we have a partition 
\begin{equation}\label{En-lmu}
(n-\ell,\mu):=(n-\ell,\mu_1,\mu_2,\dots)
\end{equation}
of $n$. 
Every partition $\la$ of $n$ can be written in the form $\la=(\ell,\mu)$ for a (possibly empty) partition $\mu$ of $n-\ell$. 

For $\la,\la^1,\dots,\la^s\in\Par_p(n)$, we denote 
$$\llbracket D^\la\rrbracket:=\{D^\la,D^\la\otimes \sgn\}\quad\text{and}\quad \llbracket D^{\la^1},\dots,D^{\la^s}\rrbracket:=\llbracket D^{\la^1}\rrbracket\cup\dots\cup \llbracket D^{\la^s}\rrbracket.
$$
Special roles will be played by the sets
$$
\T_n:=\llbracket D^{(n)}\rrbracket,\quad \N_n:=\llbracket D^{(n-1,1)}\rrbracket,\quad \NT_n:=\N_n\cup \T_n. 
$$
}

%We denote by $V{\da}_H$ or $V{\da}^G_H$ the {\em restriction} of $V$ from $G$ to $H$, and by $\ind^GW$ or $\ind^G_HW$ the {\em induction} of $W$ from $H$ to $G$. We denote by $V^G$ the space of {\em $G$-invariant vectors} in $V$. We denote by $\F_G$ the trivial $\F G$-module. We denote by $\IBr_p(G)$ the set of irreducible $p$-Brauer characters of $G$. 

%The maximal dimension of an irreducible $\F G$-module (resp. $\C G$-module) will be denoted by $\bmax_p(G)$ (resp. $\bmax(G)$). We denote by $d(G)$ the smallest degree of a non-linear irreducible complex character of $G$ (if such exists), and by $P(G)$ the smallest index of a (proper) maximal subgroup of $G$. 

The following simple observations turn out to be very useful:

\begin{Lemma} \label{LEasy} %{\rm \cite{}}%{\bf ()}
We have:
\begin{enumerate}
\item[{\rm (i)}] If $G\leq \SSS_n$, then $n\geq P(G)$. If $G$ is not primitive on $\{1,\dots,n\}$ then $n> P(G)$.
\item[{\rm (ii)}] If $G$ is a simple group, then $P(G)>d(G)$.
\end{enumerate}
\end{Lemma}
\begin{proof}
(i) follows by considering point stabilizers. (ii) comes on observing that $\ind_H^{G}\C_G$ contains some non-trivial irreducible components for any $H<G$.  
\end{proof}

Note that $\bmax_p(G)\leq \bmax(G)$. We will need the following bound:

\begin{Lemma}\label{TBound3} {\rm \cite[Theorem 2.2]{Se}} 
Let $G = SL_m(q)$ or $Sp_{2m}(q)$ with $m \geq 2$. If $B$ is a Borel subgroup of $G$, then $\bmax(G) \leq [G:B]$.
\end{Lemma}

\subsection{Representations of symmetric and alternating groups}
Recall the notation and the facts on representation theory of symmetric and alternating groups introduced in Section~\ref{SIntro}. In addition, we will denote by $M^\la$ the permutation module and by $S^\la$ the Specht module over the symmetric group $\SSS_n$ corresponding to a partition $\la$ of $n$, see \cite{JamesBook}. Occasionally, we will need the corresponding Specht module over $\C$, which we denote $S^\la_\C$. Thus $S^\la$ is a reduction modulo $p$ of $S^\la_\C$. 

\begin{Lemma} \label{LSCe} %{\rm \cite{}}%{\bf ()}
Suppose that $G\leq \SSS_n$. If $S^\la\da_G$ is irreducible then so is $S^\la_\C\da_G$. 
\end{Lemma}
\begin{proof}
This follows on observing that reduction modulo $p$ and restriction to a subgroup commute.
\end{proof}

\begin{Lemma}\label{l4}
Let $p=3$ and $n\equiv 0\pmod{3}$. If $G<\SSS_n$ and $D^{(n-2,2)}\da_G$ is irreducible then $n\leq 24$.
\end{Lemma}

\begin{proof}
The assumptions imply that $D^{(n-2,2)}\cong S^{(n-2,2)}$. 
By Lemma~\ref{LSCe}, if $D^{(n-2,2)}\da_G$ is irreducible then so is $S^{(n-2,2)}_\C\da_G$. The result now follows from \cite[Theorem 1]{S}.
\end{proof}
%Suppose that $G\leq \SSS_n$. Since reduction modulo $p$ and restriction to a subgroup commute, if $S^\la\da_G$ is irreducible then so is $S^\la_\C\da_G$. 

%For $0\leq \ell<n$, we denote
%\begin{align*}
%\Par^{(\ell)}(n)&:=\{\la \vdash n \mid \la_1\geq n-\ell\},
%\\ 
%\L^{(\ell)}(n)&:=\{\la\in \Par_p(n)\mid \la\ \text{or}\ \la^\Mull\ \text{belongs to}\  \Par^{(\ell)}(n) \cap \Par_p(n)\}.
%%\Par^{(\ell)}_p(n):=\Par_p(n)\cap \Par^{(\ell)}(n).
%\end{align*}

We next record some known results on dimensions of special irreducible modules for $p=2$ and $3$.

\begin{Lemma}\label{LL2}
We have: 
\begin{enumerate}[\rm(i)]
\item If $p=2$, then 
$$\dim D^{(n-2,2)}
=
\left\{
\begin{array}{ll}
(n^2-5n+4)/2 &\hbox{if $n \equiv 0 \pmod{4}$,}\\
(n^2-3n-2)/2 &\hbox{if $n \equiv 1 \pmod{4}$,}\\
(n^2-5n+2)/2 &\hbox{if $n \equiv 2 \pmod{4}$,}\\
(n^2-3n)/2&\hbox{if $n \equiv 3 \pmod{4}$.}
\end{array}
\right.
$$ 

\item If $p=3$ then 
$$\dim D^{(n-2,2)}
=
\left\{
\begin{array}{ll}
(n^2-3n)/2 &\hbox{if $n \equiv 0 \pmod{3}$,}\\
(n^2-3n-2)/2 &\hbox{if $n \equiv 1 \pmod{3}$,}\\
(n^2-5n+2)/2 &\hbox{if $n \equiv 2 \pmod{3}$.}
\end{array}
\right.
$$ 

\item If $p=3$ then  
$$\dim D^{(n-2,1^2)}=
\left\{
\begin{array}{ll}
(n^2-5n+6)/2 &\hbox{if $3\mid n$,}\\
(n^2-3n+2)/2 &\hbox{if $3\nmid n$.}
\end{array}
\right.
$$ 
\end{enumerate}
\end{Lemma}
\begin{proof}
This is well known and follows easily from \cite[24.15, 24.1]{JamesBook}.
\end{proof}

The following results will be needed to study irreducible restrictions to $M_{24}$: %, collect some information for modules in $\L^{(4)}(24)$:

\begin{Lemma}\label{LS243}
Let $n=24$, $p=3$, and $\la\in \Par^{(4)}(24)\setminus \Par^{(1)}(24)$. Then the dimension of $D^\la$ and the decomposition of $[S^\la]$ in the Grothendieck group are as follows 
\begin{enumerate}
\item[{\rm (i)}] $\dim D^{(22,2)}=252$ and $[S^{(22,2)}]=[D^{(22,2)}]$.
\item[{\rm (ii)}] $\dim D^{(22,1^2)}=231$ and  $[S^{(22,1^2)}]=[D^{(22,1^2)}]+[D^{(23,1)}]$.
\item[{\rm (iii)}] $\dim D^{(21,3)}=1726$ and $[S^{(21,3)}]=[D^{(21,3)}]+[D^{(23,1)}]$.
\item[{\rm (iv)}] $\dim D^{(21,2,1)}=1540$ and $[S^{(21,2,1)}]=[D^{(21,2,1)}]+[D^{(21,3)}]+[D^{(22,1^2)}]+[D^{(23,1)}]+[D^{(24)}]$.

\item[{\rm (v)}] $\dim D^{(20,4)}=6854$ and $[S^{(20,4)}]=[D^{(20,4)}]+[D^{(21,3)}]+[D^{(23,1)}]$.

\item[{\rm (vi)}]  $\dim D^{(20,3,1)}=26082$ and $[S^{(20,3,1)}]=[D^{(20,3,1)}]$.

\item[{\rm (vii)}]  $\dim D^{(20,2^2)}=7315$ and $[S^{(20,2^2)}]=[D^{(20,2^2)}]+[D^{(20,4)}]+[D^{(21,3)}]+[D^{(21,2,1)}]+2[D^{(23,1)}]+[D^{(24)}]$.

\item[{\rm (viii)}]  $\dim D^{(20,2,1^2)}=26334$ and $[S^{(20,2,1^2)}]=[D^{(20,2,1^2)}]$.

\end{enumerate}
\end{Lemma}
\begin{proof}
(i), (iii), (v) follow from \cite[24.15]{JamesBook}.

(ii) follows from \cite[24.1]{JamesBook}.

(iv), (vii) follow from \cite[Appendix]{JamesDim}.

(vi), (viii) follow from Carter's Criterion, see \cite{JM}.
\end{proof}

\begin{Lemma}\label{LS242}
Let $n=24$, $p=2$, and $\la\in \Par^{(4)}(24)\setminus \Par^{(1)}(24)$. Then the dimension of $D^\la$ and the decomposition of $[S^\la]$ in the Grothendieck group are as follows 
\begin{enumerate}
\item[{\rm (i)}] $\dim D^{(22,2)}=230$ and $[S^{(22,2)}]=[D^{(22,2)}]+[D^{(23,1)}]$.

\item[{\rm (ii)}] $\dim D^{(21,3)}=1496$ and $[S^{(21,3)}]=[D^{(21,3)}]+[D^{(22,2)}]+[D^{(23,1)}]$.

\item[{\rm (iii)}] $\dim D^{(21,2,1)}=3520$ and $[S^{(21,2,1)}]=[D^{(21,2,1)}]$.

\item[{\rm (iv)}] $\dim D^{(20,4)}=7084$ and $[S^{(20,4)}]=[D^{(20,4)}]+[D^{(21,3)}]+[D^{(23,1)}]$.

\item[{\rm (v)}]  $\dim D^{(20,3,1)}=17248$ and $[S^{(20,3,1)}]=[D^{(20,3,1)}]+[D^{(20,4)}]+[D^{(21,3)}]+[D^{(22,2)}]+[D^{(23,1)}]+2[D^{(24)}]$.

\end{enumerate}
\end{Lemma}
\begin{proof}
(i), (ii), (iv) follow from \cite[24.15]{JamesBook}.

(iii) follows from Carter's Criterion, see \cite{JM}.

(v) follow from \cite[Theorem 7.1]{Jamesp=2}.
\end{proof}

For partitions $\mu^1=(\mu^1_1,\dots,\mu^1_{h_1}), \dots, \mu^k=(\mu^k_1,\dots,\mu^k_{h_k})$, we define the composition 
$$
(\mu^1,\dots,\mu^k):=(\mu^1_1,\dots,\mu^1_{h_1},\dots,\mu^k_1,\dots,\mu^k_{h_k}).
$$
Recalling (\ref{ESpin}), for a partition $\la=(\la_1,\dots,\la_h)$ of $n$, we now define its {\em double}
$
\operatorname{dbl}(\la):=(\be_{\la_1},\dots,\be_{\la_h}).
$

\begin{Lemma} \label{LBenson} {\rm \cite[Theorem 1.1]{Benson}} %{\bf ()}
We have 
$$
\Parinv_2(n):=\Par_2(n)\cap \{\operatorname{dbl}(\la)\mid \la\in\Par_2(n),\ \la_r\not\equiv 2\pmod{4}\ \text{for all}\ r\}.
$$
\end{Lemma}

We record for future reference:
\begin{Lemma} \label{L4.3}  
Let $n\geq 5$ and $\la=(\la_1,\la_2,\dots)\in\Parinv_p(n)$. Then
$$
\la_1 \leq
\left\{
\begin{array}{ll}
(n+2)/2 &\hbox{if $p=2$,}\\
(n+p+1)/2 &\hbox{if $p\geq 3$.}
\end{array}
\right.
$$ 
\end{Lemma}
\begin{proof}
For $p\geq 3$ this is {\rm \cite[Proposition 4.3(i)]{KST}}, and for $p=2$ this follows from Lemma~\ref{LBenson}.
\end{proof}

%\subsection{Trivial submodules in $D^\la\da_{\SSS_{n-m}}$}
To analyze restriction to large doubly transitive subgroups, we will need to know %In this subsection we will prove some results which guarantee 
that the trivial submodule $\F_{\SSS_{n-m}}$ appears in the restriction $D^\la\da_{\SSS_{n-m}}$ for some reasonably small $m$. Recall the notation (\ref{En-lmu}). 
%The following result will play a key role in our treatment of large doubly transitive subgroups of $\SSS_n$:

\begin{Theorem}\label{TSub}
Let $ \ell,  L$ be integers satisfying $0\leq 2 \ell \leq  L < n$, and 
$\la = (n- \ell,\mu)\in\Par_p(n)$. Then $D^\la\da_{\SSS_{n- L}}$ contains a trivial submodule. 
\end{Theorem} 

\begin{proof}
We will apply branching rules from \cite{KBrII} without further reference. We use induction on $\ell=|\mu|$, the theorem clearly holding if $ \ell=0$ since in that case $D^\la=D^{(n)}=\F_{\SSS_n}$. Let $ \ell>0$.

If $\la$ has a good node below the first row then there exists $\nu\in\Par_p(\ell-1)$ such that $(n- \ell,\nu)=(n-1-( \ell-1),\nu)$ is a $p$-regular partition of $n-1$ and $D^{(n- \ell,\nu)}\subseteq D^\la\da_{\SSS_{n-1}}$. By the inductive assumption, $D^{(n- \ell,\nu)}\da_{\SSS_{n- L}}$ contains a trivial submodule. 

Assume now that $\la$ has no good node below the first row. Then $n- \ell=\la_1>\la_2=\mu_1$, $(n-1- \ell,\mu)$ is $p$-regular and $D^\la\da_{\SSS_{n-1}}\cong D^{(n-1- \ell,\mu)}$. If $A$ is the second top removable node of $\la$ then $A$ is normal in $(n-1- \ell,\mu)$. So $(n-1- \ell,\mu)$ has a good node below the first row. In particular there exists $\nu\in\Par_p(\ell-1)$ such that $(n-1- \ell,\nu)=(n-2-( \ell-1),\nu)$ is a $p$-regular partition of $n-2$ and $D^{(n-1- \ell,\nu)}\subseteq D^\la\da_{\SSS_{n-2}}$. By the inductive assumption, $D^{(n- 1-\ell,\nu)}\da_{\SSS_{n- L}}$ contains a trivial submodule. 
\end{proof}

In the following lemma we use functors $e_i:\mod{\F \SSS_n}\to\mod{\F\SSS_{n-1}}$ for which we refer the reader to \cite{KBook}. The integer $\eps_i(\la)$ is defined as $\max\{k\mid e_i^k D^\la\neq 0\}$. 

\begin{Lemma}\label{L2}
Let $\la\in\Par_p(n)$ with $\eps_i(\la)=2$. Let $A$ and $B$ be the $i$-normal nodes in $\la$ with $A$ below $B$. If $\la_B$ is $p$-regular and the socle of $(e_i D^\la)/D^{\la_A}$ is isomorphic to $D^{\la_B}$ then $e_i D^\la\cong D^{\la_A}|D^{\la_B}|D^{\la_A}$.
\end{Lemma}

\begin{proof}
This follows by self-duality of $e_i D^\la$, together with \cite[Theorem 1.4]{k4}.
\end{proof}

We will need the following strengthening of Theorem \ref{TSub} for the partition $(n-2,2)$:

\begin{Lemma}\label{LSub22}
If $n \geq 5$, then $D^{(n-2,2)}\da_{\SSS_{n-3}}$ contains a trivial submodule, provided $p=3$ and $n \equiv 0,1 \pmod{3}$, or 
$p=2$ and $n \equiv 0,1,3 \pmod{4}$.
\end{Lemma}

\begin{proof}
We will use branching rules from \cite{KBrII} without further reference. Assume first that $p=3$ and $n \equiv 0,1 \pmod{3}$, or that $p=2$ and $n \equiv 1,3 \pmod{4}$. Then $D^{(n-2,1)}\subseteq D^{(n-2,2)}\da_{\SSS_{n-3}}$ and so we can conclude using Theorem \ref{TSub}. Assume now that $p=2$ and $n \equiv 0\pmod{4}$, in which case $n\geq 6$. Then $D^{(n-2,2)}\da_{\SSS_{n-1}}\cong D^{(n-3,2)}$. By \cite{ShethBr}, we have in the Grothendieck group 
$$[D^{(n-3,2)}\da_{\SSS_{n-2}}]=2[D^{(n-2)}]+2[D^{(n-3,1)}]+[D^{(n-4,2)}]$$ 
(omitting the last summand if $n=6$). 
%any composition factor of $D^{(n-3,2)}\da_{\SSS_{n-2}}$ is isomorphic to $D^{(n-2)}$, $D^{(n-3,1)}$ or $D^{(n-4,2)}$. 
From Lemma \ref{L2} it follows that there exists $M\subseteq D^{(n-2,2)}\da_{\SSS_{n-2}}$ with $M\sim D^{(n-3,1)}|D^{(n-2)}$. Considering block structure it then follows that $D^{(n-3)}\subseteq M\da_{\SSS_{n-3}}\subseteq D^{(n-2,2)}\da_{\SSS_{n-3}}$.
\end{proof}

To conclude the subsection, we record for future reference the following recognition result for basic spin modules: %result, which follows from \cite[Theorem 8.1]{Wales}:

\begin{Lemma}\label{Lcycle}
Let $p=2$, $n \geq 5$, and let $H =\AAA_n$ or $\SSS_n$. 
Suppose that $V$ is an irreducible $\F H$-module in which a 
$3$-cycle $t$ acts with exactly two eigenvalues. Then $V$ is a basic spin module.
\end{Lemma}

\begin{proof}
In the case $H = \SSS_n$, the statement is \cite[Theorem 8.1]{Wales}. Suppose $H = \AAA_n$. If $V$ extends to $\SSS_n$, then we are done by 
the previous case. If $V$ does not extend to $\SSS_n$, then we can find an irreducible $\F\SSS_n$-module $W$ such that 
$W\da_G \cong V \oplus V^g$ for any $g \in \SSS_n \smallsetminus H$. Certainly we can choose such a $g$ to be (a $2$-cycle) centralizing $t$.
Thus $t$ has the same eigenvalues on $V^g$ as on $V$, and so $t$ acts quadratically on $W$. By the $\SSS_n$-case, $W$ is basic spin, and so
is $V$.
\end{proof}

\subsection{Branching recognition results}

We begin by recording the following well-known branching recognition result for the modules in $\NT_n$. 

\begin{Lemma} \label{LKZ}
{\rm \cite[Proposition 2.3]{KZ}}
Let $n \geq 6$ and $D$ be an irreducible $\F\SSS_n$-module. Suppose that all composition factors of the restriction $D\da_{\SSS_{n-1}}$ belong to $\NT_{n-1}$. Then $D \in\NT_n$, unless $n = 6$, $p = 3$ and $D \in \llbracket D^{(4,2)}\rrbracket$, or $n = 6$, $p = 5$ and $D \in \llbracket D^{(4,1^2)}\rrbracket$.
\end{Lemma}

Define 
$$
u:=\left\{
\begin{array}{ll}
4 &\hbox{if $p=3$,}\\
3 &\hbox{otherwise.}
\end{array}
\right.
$$

\begin{Lemma} \label{L271118}%{\rm \cite{}}%{\bf ()}
Let $n\geq 2u$ and $D$ be an irreducible $\F \SSS_n$-module. If  all composition factors of $D\da_{\SSS_{u,u}}$ are of the form $D^\mu\boxtimes D^\nu$ with $D^\mu\in \T_{u}$ or $D^\nu\in\T_{u}$, then $D^\la\in\NT_n$.
\end{Lemma}
\begin{proof}
If $D\not\in\NT_n$ then by Lemma \ref{LKZ}, the restriction $D\da_{\SSS_{2u}}$ has a composition factor not in $\NT_{2u}$. So it is enough to prove the lemma for $n=2u$, which is an easy explicit check.
\end{proof}

\begin{Proposition} \label{C080417} %{\rm \cite{}}%{\bf ()}
Let $s\geq 2$, $m_1,\dots,m_s\geq u$ and $m_1+\dots+m_s\leq n$, and $D$ be an irreducible $\F \SSS_n$-module such that all composition factors of $D\da_{\SSS_{m_1,\dots,m_s}}$ are of the form $D^{\mu^1}\boxtimes\dots\boxtimes D^{\mu^s}$ with at most one $t$ such that $D^{\mu^t}\not \in\T_{m_t}$. Then $D\in \NT_n$.
\end{Proposition}
\begin{proof}
By assumption, restricting further to the subgroups $\SSS_u\leq \SSS_{m_1}$ and $\SSS_u\leq \SSS_{m_2}$, we deduce that all composition factors of $D\da_{\SSS_{u,u}}$ are of the form $D^\mu\boxtimes D^\nu$ with $D^\mu\in \T_{u}$ or $D^\nu\in\T_{u}$. So the proposition follows from Lemma~\ref{L271118}.
\end{proof}

We need another special branching recognition result.

\begin{Lemma} \label{L271118_2}%{\rm \cite{}}%{\bf ()}
Let $p=3$, $n\geq 8$ and $D^\la$ be an irreducible $\F \SSS_n$-module. If all composition factors of $D^\la\da_{\SSS_{n-1}}$ belong to $\NT_{n-1}\cup \llbracket D^{(n-3,1^2)}\rrbracket$ then 
$D^\la\in \NT_{n}\cup \llbracket D^{(n-2,1^2)}\rrbracket$.
\end{Lemma}
\begin{proof}
Note that
\begin{align*}
\NT_{m}\cup \llbracket D^{(m-2,1^2)}\rrbracket=\{&D^{(m)},D^{(\lceil m/2\rceil,\lfloor m/2\rfloor)},D^{(m-1,1)},D^{(\lceil (m-1)/2\rceil,\lfloor (m-1)/2\rfloor,1)},\\
&D^{(m-2,1^2)},D^{(\lceil (m-2)/2\rceil,\lfloor (m-2)/2\rfloor,2)}\}
\end{align*}
for $m\geq 7$, see for example \cite[Lemma 2.2]{bkz}. Throughout the proof we will be using branching rules from \cite{KBrII} without further referring to them. 
%Let $\la\in\Par_3(n)$ with $D\cong D^\la$.

{\sf Case 1.} $h(\la)\geq 4$. Then from \cite[Lemma 4.7]{BaK} that $D^\la\da_{\SSS_6}$ contains a composition factor $D^{(2,2,1,1)}$. Hence $D^\la\da_{\SSS_{n-1}}$ has a composition factor of the form $D^\mu$ with $h(\mu)\geq 4$. In particular $D^\mu\not\in \NT_{n-1}\cup \llbracket D^{(n-3,1^2)}\rrbracket$.

{\sf Case 2.} $h(\la)=3$ and $\la_3\geq 3$. Then $n\geq 10$ and by \cite[Lemma 4.13]{BaK}, $D^\la_{\SSS_{10}}$ contains a composition factor  $D^{(4,3^2)}$. So if $n>10$ then $D^\la\da_{\SSS_{n-1}}$ contains a composition factor $D^\mu$ with $\mu_3\geq 3$, in particular $D^\mu\not\in \NT_{n-1}\cup \llbracket D^{(n-3,1^2)}\rrbracket$. If $n=10$ then $\la=(4,3^2)$ and $\la^\Mull=(7,2,1)$, so $D^{(5,2^2)}\not \in\NT_{9}\cup \llbracket D^{(7,1,1)}\rrbracket$ is a composition factor of $D^\la\da_{\SSS_9}$ since $(5,2^2)=(6,2,1)^\Mull$.

{\sf Case 3.} $h(\la)=3$, $\la_3\leq 2$ and $\la_1-\la_2\geq 3$. We may assume that $\la_2\geq 2$, since otherwise $\la=(n-2,1^2)$. But then  $D^{(\la_1-1,\la_2,\la_3)}\not\in \NT_{n-1}\cup \llbracket D^{(n-3,1^2)}\rrbracket$ is a composition factor of $D^\la\da_{\SSS_{n-1}}$.

{\sf Case 4.} $h(\la)=3$, $\la_3\leq 2$ and $\la_1-\la_2\leq 2$. We may assume that $\la_1-\la_2=2$, since otherwise $D^\la\in \NT_{n}\cup \llbracket D^{(n-2,1^2)}\rrbracket$. If $n>8$ then $\la_2\geq 3>\la_3$, so  $D^{(\la_1,\la_2-1,\la_3)}\not\in \NT_{n-1}\cup \llbracket D^{(n-3,1^2)}\rrbracket$ is a composition factor of $D^\la\da_{\SSS_{n-1}}$. If $n=8$ then $\la=(4,2,2)$ and $D^{(4,2,1)}\not \in\NT_{7}\cup \llbracket D^{(5,1,1)}\rrbracket$ is a composition factor of $D^\la\da_{\SSS_7}$.

{\sf Case 5.} $h(\la)=2$ and $\la_1-\la_2\geq 3$. We may assume that $\la_2\geq 2$, in which case $D^{(\la_1-1,\la_2)}\not\in \NT_{n-1}\cup \llbracket D^{(n-3,1^2)}\rrbracket$ is a composition factor of $D^\la\da_{\SSS_{n-1}}$.

{\sf Case 6.} $h(\la)=2$ and $\la_1-\la_2\leq 2$. We may assume that $\la_1-\la_2=2$. Since $n\geq 8$ we have that $\la_2\geq 3$ and so $D^{(\la_1,\la_2-1)}\not\in \NT_{n-1}\cup \llbracket D^{(n-3,1^2)}\rrbracket$ is a composition factor of $D^\la\da_{\SSS_{n-1}}$.
\end{proof}

\begin{Corollary} \label{P110417} %{\rm \cite{}}%{\bf ()}
Let $p=3$, $n=2^m$ for $m\geq 4$, and $D$ be an irreducible $\F \SSS_n$-module. Suppose that all composition factors $D^\mu\boxtimes D^\nu$ of the restriction $D\da_{\SSS_{n/2}\times \SSS_{n/2}}$ satisfy one of the following three conditions:
\begin{enumerate}
\item[{\rm (1)}] $D^\mu\cong D^\nu\in \N_{n/2}$, 
\item[{\rm (2)}]  $D^\mu\in\T_{n/2}, D^\nu\in \NT_{n/2}\cup \llbracket D^{(n/2-2,1,1)}\rrbracket$, 
\item[{\rm (3)}] $D^\nu\in\T_{n/2}, D^\mu\in \NT_{n/2}\cup \llbracket D^{(n/2-2,1,1)}\rrbracket$. 
\end{enumerate}
Then $D\in\NT_n\cup\llbracket D^{(n-2,1,1)}\rrbracket$. 
\end{Corollary}
\begin{proof}
By assumption, all composition factors of $D\da_{\SSS_{n/2}}$ belong to $\NT_{n/2}\cup \llbracket D^{(n/2-2,1^2)}\rrbracket$, and the result follows from Lemma \ref{L271118_2}.
\end{proof}

\subsection{Dimension bounds}
%In \cite{JamesDim}, 
Recall the notation (\ref{En-lmu}). 
We begin by recording James' lower bounds for $\dim D^{(n-\ell,\mu)}$ with $\ell\leq 4$:

\begin{Lemma} \label{LBound} {\rm \cite[Appendix]{JamesDim}} %{\bf ()}
Let $1\leq \ell\leq 4$, $\mu\in \Par_p(\ell)$, and $n$ be such that $(n-\ell,\mu)\in\Par_p(n)$ with $\mu \vdash \ell$. Then 
$$
\dim D^{(n-\ell,\mu)}\geq 
\left\{
\begin{array}{ll}
n-2 &\hbox{if $\ell=1$,}\\
(n^2-5n+2)/2 &\hbox{if $\ell=2$,}\\
(n^3 -9n^2 +14n)/6 &\hbox{if $\ell=3$.}
\\(n^4-14n^3+47n^2-34n)/24 &\hbox{if $\ell=4$.}
\end{array}
\right. 
$$
\end{Lemma}

Set
$$
\de_p:=
\left\{
\begin{array}{ll}
0 &\hbox{if $p\neq 2$,}\\
1 &\hbox{if $p=2$.}
\end{array}
\right.
$$
For integers $\ell\geq 0$ and $n$ we define the rational numbers
\begin{align*}
C_{\ell}^p(n)&:=
p^\ell\binom{n/p-\de_p}{\ell}
\\
&
=\frac{1}{\ell!}\prod_{i=0}^{\ell-1}(n-(\de_p+i)p)
\\
&
%=\left\{
%\begin{array}{ll}
%p^\ell\binom{n/p}{\ell} &\hbox{if $p>3$,}\\
%p^\ell\binom{n/p-1}{\ell} &\hbox{if $p=2,3$}
%\end{array}
%\right.
=
\left\{
\begin{array}{ll}
\frac{n(n-p)(n-2p)\cdots(n-(\ell-1)p)}{\ell!} &\hbox{if $p>2$,}\\
\frac{(n-p)(n-2p)\cdots(n-\ell p)}{\ell!}  &\hbox{if $p=2$.}
\end{array}
\right.
\end{align*}

The following result %which 
substantially develops \cite{JamesDim} (the upper bound $\dim D^\la \leq n^\ell$ is trivial, since $D^\la$ is contained in the permutation module
$M^\la$, which has dimension at most $n!/(n-\ell)! \leq n^\ell$).
% was obtained in \cite[Theorem A]{KMT2}:

\begin{Theorem}\label{TBound1} {\rm \cite[Theorem A]{KMT2}} 
Let $\ell\geq 4$, $n\geq p(\de_p+\ell-2)$, and $\la=(n-\ell,\mu)\in\Par_p(n)$ for some $\mu\in\Par_p(\ell)$. Then 
$$n^\ell \geq \dim D^\la\geq C_{\ell}^p(n).$$ 
\end{Theorem}

%In view of \cite[Theorem 1]{JamesDim}, the lower bound of  Theorem~\ref{TBound1} is asymptotically sharp. 
%Theorem~\ref{TBound1} will be crucially used in \cite{KMTTwo}. 

While Theorem~\ref{TBound1} requires that $n$ is relatively large compared to $\ell$, we also have the following universal lower bounds which strengthens \cite[Theorem 5.1]{GLT}:

\begin{Theorem}\label{TBound2} {\rm \cite[Theorems B, C]{KMT2}} 
Let $\la = (\la_1, \la_2,\dots)\in\Par_p(n)$, and $k:=\max\{\la_1,\la^\Mull_1\}$. 
\begin{enumerate}[\rm(i)]
\item If $p=2$, then $\dim D^\la \geq 2^{n-k}$.
\item If $p > 2$ let $\la^\Mull=(\la^\Mull_1,\la^\Mull_2,\dots)$, and let 
$m$ be minimal such that $D^\la\da_{\SSS_{n-m}}$ contains a $1$-dimensional submodule. Put 
$k:=\max\{\la_1,\la^\Mull_1\}$ and $t:= \max\{n-k,m\}.$
Then 
$$\dim D^\la\geq 2\cdot 3^{(t-2)/3}.$$
\end{enumerate}
In particular, for all $p$ and $n\geq 5$, we have $\dim D^\la\geq 2^{(n-k)/2}$. 
\end{Theorem}

The following technical result will be used to study irreducible restrictions to doubly transitive subgroups $G$ with $\soc(G)\cong PSL(m,q)$ and $Sp_{2m}(2)$ in Sections~\ref{SLG1} and \ref{SLG3}. 

\begin{Proposition}\label{PDim1}
Let $n \geq 324$, $p = 2$ or $3$, and define $\ell$ from 
$\max(\la_1,\la^\Mull_1)=n-\ell$. 
%$\la\in \Par_p(n)$. Let $n-\ell$ be the first part of $\la$ if $p=2$, and the largest among the first parts of $\la$ and $\la^\Mull$ if $p = 3$.  
If
$$\dim D^\la \leq n^{\frac{1}{2}\log_2 n +1}$$
then $\ell \leq 0.7\log_2 n+1.4$.
\end{Proposition}

\begin{proof}
Set $L(n) := \frac{1}{2}\log_2 n +1$. 
We need to show that $\ell \leq 1.4L(n)$. As $1.4L(324)>7$, we may assume that $\ell>7$. 
Replacing $\la$ by $\la^\Mull$ if necessary, we 
may assume that $\la = (n-\ell,\mu)$ for a partition $\mu$ of $\ell$. By Theorem~\ref{TBound2} and the assumption, we now have 
$$2^{\ell/2} \leq \dim D^\la \leq n^{L(n)},$$
and so
\begin{equation}\label{eq-PD11}
  \ell \leq 2L(n)\log_2 n = (\log_2 n+2)\log_2 n =: L_1(n).
\end{equation}
As $n\geq 324$, we certainly have that $\ell \leq L_1(n) < \frac{1}{3}n+2$, whence $n\geq p(\de_p+\ell-2)$, and Theorem \ref{TBound1} applies to give
\begin{equation}\label{250219}
\dim D^\la\geq C^p_\ell(n) > \frac{(n+3- 3\ell)^\ell}{\ell !} > \biggl( \frac{2(n+3)}{\ell}-6 \biggr)^\ell,
\end{equation}
where we have used $\ell! < (\ell/2)^\ell$ for $\ell\geq 6$ to get the last inequality.

If $\ell > cL(n)$ for some $c> 0$, we get
$$n^{L(n)} > \biggl( \frac{2(n+3)}{\ell}-6 \biggr)^{cL(n)},$$
and so 
$$\ell > f(n,c):=\frac{2(n+3)}{n^{1/c}+6}.$$
We have therefore shown that  
\begin{equation}\label{eq-PD12}
  \mbox{If }f(n,c) \geq \ell\  \mbox{ for some }c > 0, \mbox{ then }\ell \leq cL(n).
\end{equation}  
We will use this implication repeatedly to prove $\ell \leq 1.4L(n)$.

First we take $c = 16$. By the assumption on $n$ and \eqref{eq-PD11}, $f(n,16) > L_1(n) \geq \ell$. Hence,  \eqref{eq-PD12} implies that 
%\begin{equation}\label{eq-PD13}
$$\ell \leq L_2(n):=16L(n) = 8 \log_2 n +16.$$
%\end{equation}  
Next we take $c = 9$ and note that $f(n,9) > L_2(n) \geq \ell$ for $n \geq 324$ (this is the only place where smaller $n$ would not work). Applying \eqref{eq-PD12}, we deduce that 
%\begin{equation}\label{eq-PD14}
$$\ell \leq L_3(n):=9L(n) = 4.5 \log_2 n +9.$$
%\end{equation}  
Now take $c = 2.8$ and note that $f(n,2.8) > L_3(n) \geq \ell$ for $n \geq 324$. Applying \eqref{eq-PD12}, we now obtain 
%\begin{equation}\label{eq-PD15}
$$\ell \leq L_4(n):=2.8L(n) = 1.4 \log_2 n +2.8.$$
%\end{equation} 
Next we take $c = 1.6$ and note that $f(n,1.6) > L_4(n) \geq \ell$ for $n \geq 324$. Using \eqref{eq-PD12}, we deduce that 
%\begin{equation}\label{eq-PD16}
$$\ell \leq L_5(n):=1.6L(n) = 0.8 \log_2 n +1.6.$$
%\end{equation}
Finally, we take $c = 1.4$ and note that $f(n,1.4) > L_5(n) \geq \ell$ for $n \geq 324$. Again using \eqref{eq-PD12}, we conclude that 
$\ell \leq 1.4L(n)$, as stated.
\end{proof}

%\section{Bounding $D^\la$ for large doubly transitive subgroups. II}\label{SLG2}

%\subsection{Some small modules}

We now establish some dimension recognition results for modules in $\L^{(\ell)}(n)$ for small $\ell$. 

%first consequences of the cited dimension bounds. 

\begin{Lemma}\label{Lr1}
%Let $W$ be an irreducible $\F\AAA_n$-module \IBr_p(\AAA_n)$ be an irreducible summand of $D^\mu\dar_{\AAA_n}$ with $\mu \in \Parp(n)$. 
If $n \geq 17$ and $\dim E^{\mu}_{(\pm)} < (n^2-5n+2)/2$, then $\mu \in  \L^{(1)}(n)$.
\end{Lemma}

\begin{proof}
The statement follows from \cite[Lemma 6.1]{GT}. 
%The cases $n = 15$, $16$ can be checked using \cite{GAP}. 
\end{proof}

The following proposition extends \cite[Lemma 1.20]{BK}:

\begin{Proposition}\label{PBound12}
The following lower bounds hold.
\begin{enumerate}[\rm(i)]
\item Let $n\geq 13$, and assume in addition that $n\geq 23$ if $p=2$. 
%Suppose that either $p\neq 2$ and $n \geq 13$, or $p=2$ and $n \geq 23$. 
Then for $\la\in\Par_p(n)$, we have either $\lambda \in  \L^{(2)}(n)$ or $$\dim D^\lambda \geq (n^3-9n^2+14n)/6.$$
%or $p=2$ and $\dim D^\lambda > 2(n^3-9n^2+14n)/25$.
\item Suppose that $p \geq 3$ and $n \geq 17$. Then for $\la\in\Par_p(n)$, we have either $\lambda \in  \L^{(3)}(n)$ or 
$$\dim D^\lambda \geq (n^4-14n^3+47n^2-34n)/24.$$
\end{enumerate}
\end{Proposition}

\begin{proof}
By Lemma \ref{LBound}, if $\la \in \L^{(3)}(n) \smallsetminus \L^{(2)}(n)$, then 
$\dim D^\la \geq (n^3-9n^2+14n)/6$. Now assume that $\la \notin \L^{(3)}(n)$, and in addition
$D^\la$ is not basic spin if $p=2$. Then $\dim D^\la \geq (n^4-14n^3+47n^2-34n)/24$
by \cite[(6.2)]{Mu}. In the case where $D^\la$ is basic spin and $n \geq 23$, one can check directly that 
$\dim D^\lambda \geq (n^3-9n^2+14n)/6$.
\end{proof}

\begin{Remark}
The statement of Proposition \ref{PBound12}(i) does not hold for   $p=2$ and $n=22$, a counterexample given by the basic spin module $D^{\be_{22}}$. However,
a similar argument shows that for $n \geq 17$ we have either $\lambda \in  \L^{(2)}(n)$
or $\dim D^\lambda > 2(n^3-9n^2+14n)/25$.
\end{Remark}

%\begin{Proposition}\label{PBound12}
%Suppose that $p= 3$ and $n \geq 17$. Then for $\la\in\Par_p(n)$, we have either $\lambda \in  \L^{(3)}(n)$ or $\dim D^\lambda \geq n(n-3)(n-6)(n-9)/24$.
%%or $p=2$ and $\dim D^\lambda > 2(n^3-9n^2+14n)/25$.
%\end{Proposition}
%
%\begin{proof}
%Setting $f(n) = n(n-3)(n-6)(n-9)/24-1$, one can check that $2f(n) > f(n+2)$ for $n \geq 17$. Furthermore,  by Lemma~\ref{LBound}, we get
%$$\dim D^\lambda  \geq (n^4-14n^3+47n^2-34n)/24 > f(n)$$
%when $\lambda \in \L^{(4)}(n) \smallsetminus \L^{(3)}(n)$  
%(we have used that $n\geq 13$ to get the second equality). 
%Also, by Theorem \ref{TBound1} (and using $n\geq 17$ to get the second equality), we have 
%that 
%$$\dim D^\lambda  \geq n(n-3)(n-6)(n-9)(n-12)/120 > f(n)$$
%if $\lambda \in \L^{(5)}(n) \smallsetminus \L^{(4)}(n)$. %(Indeed, replacing $\la$ by $\la^\Mull$ if necessary, we may assume that $\la = (n-5,\mu)$ and then apply Theorem \ref{TBound1} to $D^\la$.)
%Using \cite{GAP} one can check that the statement holds for $n = 17$, $18$. Hence, the statement holds for all $n \geq 17$ by \cite[Lemma 4]{JamesDim}.
%\end{proof}

\section{Order bounds and dimension bounds}
\label{SOD}

\subsection{Subgroups of large order}
First we extend Propositions 6.1 and 6.2 of \cite{KlW}, following mostly the arguments given therein.

\begin{Proposition}\label{PSimple}
Let $S < \AAA_n$ be a non-abelian simple subgroup such that 
\begin{equation}\label{bound1}
  |\Aut(S)| \geq 2^{n/2-4}.
\end{equation}  
Then one of the following happens:
\begin{enumerate}[\rm(i)]
\item $S \cong \AAA_m$ with $m < n$. Moreover, if $m \geq 12$, then $S$ is intransitive, and each of its orbits
on $\{1,2, \ldots, n\}$ has length $1$ or $m$.
\item $S \cong PSL_m(q)$ with $(m,q) = (2, \leq 37)$, $(3, \leq 5)$, $(4,3)$, $(5,2)$, or $(6,2)$.
\item $S \cong SU_3(3)$ or $SU_4(2) \cong PSp_4(3)$.
\item $S \cong Sp_6(2)$.
\item $S \cong M_{11}$, $M_{12}$, $M_{22}$, $M_{23}$, or $M_{24}$.
\end{enumerate}
\end{Proposition}

\begin{proof}
(a) First we consider the case $S \cong \AAA_m$. Note that $m < n$ as $S \neq \AAA_n$.  Assume furthermore that
$m \geq 12$ and $S$ has an orbit of length $k \neq 1,m$ on $\{1,2, \ldots, n\}$. As in \cite{KlW}, it follows that 
$n \geq k \geq m(m-1)/2$. Now one can check that  $2^{m(m-1)/4-4} > m!$
for $m \geq 12$, a contradiction.

For the remaining cases, recall from Lemma~\ref{LEasy} that %$P(S)$ denotes the smallest index of a maximal subgroup of $S$, and $d(S)$ denotes the smallest degree of a nontrivial complex irreducible representation of $S$. It is easy to check that 
$$%\begin{equation}\label{for-bound1}
  n \geq P(S) > d(S). 
$$%\end{equation}  

Now, if $S$ is one of the $26$ sporadic simple groups and not listed in (v), then using 
the exact value of $P(S)$ given in \cite{Atlas} (or of $d(S)$, if $P(S)$ was not listed therein) one can check that
\eqref{bound1} cannot hold.

\smallskip
(b) Assume now that $S$ is a classical group. Then $|\Aut(S)|$ and $P(S)$ are listed in Tables 5.1.A and  5.2.A of 
\cite{KlL}.

First suppose that $S = PSL_m(q)$. Then $|\Aut(S)| < q^{m^2}$. If $m \geq 4$ (and 
$S \not\cong \AAA_8$), then $P(S) = (q^m-1)/(q-1)$, and
one can check that \eqref{bound1} can hold only when $(m,q) = (4,3)$, $(5,2)$, or $(6,2)$. If $(m,q) = (3, \geq 7)$ or
$(2, \geq 41)$, then again $P(S) = (q^m-1)/(q-1)$ and one checks that \eqref{bound1} is violated. 

Next suppose that $S = PSU_m(q)$. Then we again have $|\Aut(S)| < q^{m^2}$. If $m \geq 5$, then $P(S) > q^{2m-3}$ and 
one checks that \eqref{bound1} cannot hold. If $(m,q) = (4, \geq 3)$ then $P(S) = (q+1)(q^3+1)$. 
If $m = 3$, then $P(S) = q^3+1$ for $q \geq 7$ or $q = 4$, and $50$ if $q=5$. In all these cases, \eqref{bound1} is violated. 

Suppose now that $S = PSp_{2m}(q)$ with $m \geq 2$, or $\Omega_{2m+1}(q)$ with $m \geq 3$. If $m \geq 3$, then 
$|\Aut(S)| < q^{m(2m+1)+1}/2$ whereas $P(S) \geq q^{m-1}(q^m-1)$, and so \eqref{bound1} can possibly hold
only when $(m,q) = (3,2)$. Similarly, if $(m,q) = (2, \geq 4)$, then $P(S) = (q^4-1)/(q-1)$, and \eqref{bound1} cannot hold.

Suppose $S = P\Omega^\pm_{2m}(q)$ with $m \geq 4$. If $m >4$ or if $S \not\cong P\Omega^+_8(q)$, then 
$|\Aut(S)| < q^{m(2m-1)+1}$ whereas $P(S) > q^{2m-2}$, and so \eqref{bound1} is impossible. Similarly, 
\eqref{bound1} rules out the remaining case $S = P\Omega^+_8(q)$. 

\smallskip
(c) Finally, assume that $S$ is an exceptional group of Lie type. The cases $S = F_4(2)$, $\twn2 F_4(2)'$, $\twn3 D_4(2)$, 
$G_2(3)$, $G_2(4)$, or $\twn2 B_2(8)$ can be ruled out directly using \cite{Atlas}. In all other cases, we can use the 
Landazuri-Seitz-Zalesskii lower bound on $d(S)$ as recorded in \cite[Table 5.3.A]{KlL} to check that 
\eqref{bound1} cannot hold.
\end{proof}

The following known lemma follows from the O'Nan-Scott theorem, see e.g. \cite{LPS}:

\begin{Lemma} \label{LAbSoc} %{\rm \cite{}}%{\bf ()}
Suppose $G < \SSS_n$ is a primitive subgroup with an abelian minimal normal subgroup $S$. 
Then $n = r^m$ is a power of some prime $r$, and $G$ is a subgroup of the affine group $AGL(V) = AGL_m(r)$ in its action on the points of $V = \F_r^m$. 
\end{Lemma}

\begin{Proposition}\label{PPrim}
Let $G < \SSS_n$ be a primitive subgroup, not containing $\AAA_n$ and such that 
\begin{equation}\label{bound2}
  |G| \geq 2^{n/2-4}.
\end{equation}  
Then one of the following happens:
\begin{enumerate}[\rm(i)]
\item $\soc(G)$ is elementary abelian of order $n = r^k$, with $(r,k) = (2,\, {\leq{6}})$, $(3,\,{\leq{3}})$, or $(5,2)$.
\item $S \leq G \leq \Aut(S)$ for a non-abelian simple group $S$. Furthermore, either $S \cong \AAA_m$ with
$m \leq 11$, or $S$ satisfies Proposition \ref{PSimple}(ii)--(v).

\item $\soc(G) = S \times S$ for a non-abelian simple group $S \leq \AAA_a$, $5 \leq a \leq 9$, and $n = a^2$.
\end{enumerate}
\end{Proposition}

\begin{proof}
We apply the O'Nan-Scott theorem in the version given in \cite{LPS}. First suppose that $\soc(G) \cong \CCC_r^k$,
so that $n = r^k$ for a prime $r$. Then Lemma \ref{LAbSoc} shows that 
$G \leq AGL_k(r)$ and so $|G| < r^{k^2+k}$. A direct computation shows that \eqref{bound2} can hold only in the cases listed in
(i).

Next assume that $\soc(G) = S$ is non-abelian simple. Then $S \leq G \leq \Aut(S)$ and $S$ is transitive on 
$\{1,2,\ldots,n\}$. Now we can apply Proposition \ref{PSimple} to arrive at (ii). 

In the remaining cases, $\soc(G) = S^k$ for a non-abelian simple group $S$ and $k \geq 2$, and $G$ is of type 
III(a), III(b), or III(c) in the notation of \cite{LPS}. Suppose $G$ is of type III(b), so that $n = a^b$ with $a \geq 5$,
$b \geq 2$, and $G \leq \SSS_a \wr \SSS_b$. In this case, $b \leq \log_5 n$ and 
$$(a!)^b\cdot b! \leq (a^a)^bb^b \leq n^ab^b \leq n^{\sqrt{n}}\cdot (\log_5 n)^{\log_5 n}.$$
Now if $n \geq 318$, then
$$2^{n/2-4} >   n^{\sqrt{n}}\cdot (\log_5 n)^{\log_5 n},$$
violating \eqref{bound2}. The cases where $n = a^b \leq 317$ can now be checked directly to show
that $b = 2$ and $5 \leq a \leq 9$. This implies that $k = 2$, $S \leq \AAA_a$, and we arrive at (iii).

Suppose $G$ is of type III(a). Then $n = |S|^{k-1}$ and $G \leq S^k \cdot (\SSS_k \times \Out(S))$. Since $|S| \geq 60$,
we can check that
\begin{equation}\label{for-bound2}
  2^{|S|^{k-1}/2-4} > |S|^{k+1}\cdot k!.
\end{equation} 
As $|\Out(S)| < |S|$, \eqref{bound2} cannot hold. 

Finally, assume that $G$ is of type III(c). Then $n = |S|^k$ and 
$$G \leq \Aut(S^k) \cong S^k \cdot (\Out(S)^k \rtimes \SSS_k).$$
Now \eqref{for-bound2} implies that
$$2^{n/2-4} \geq 2^{|S|^k/2-4}  > (|S|^{k+1}\cdot k!)^2 > |S|^{2k} \cdot k! > |G|,$$
a contradiction.
\end{proof}

\subsection{Irreducible restrictions for some special modules and groups}

Now we prove main results of this section. 

\begin{Theorem}\label{ext-spin}
Let $p = 2$ or $3$, $H=\AAA_n$ or  $\SSS_n$ with $n \geq 5$, and $V$ be an irreducible $\F H$-module.
Let $G < H$ be a primitive subgroup not containing $\AAA_n$, with $S :=\soc(G)$, such that $V\dar_G$ is irreducible. 
Assume in addition that either $V$ is a basic spin module in characteristic $2$, or $H=\AAA_n$ and $V$ does not extend to $\SSS_n$. Then one of 
the following happens:
\begin{enumerate}[\rm(i)]
\item $S$ is elementary abelian of order $n = r^k$, with $(r,k) = (2, {\leq{6}})$, $(3,{\leq{3}})$, or $(5,2)$.
\item $S \in \{M_{11}, M_{12},M_{22}, M_{23}, M_{24}\}$, and $G$ is doubly transitive.
\item $H = \AAA_9$, $p = 2$, $SL_2(8) \unlhd G \leq SL_2(8) \rtimes \CCC_3$, $V$ is the basic spin module of dimension $8$
whose Brauer character takes value $-1$ on elements of order $9$ of $S$.
\item $n=10$, $p=2$, $S \cong \AAA_6$, $G \not\leq PGU_2(9)$, $V$ is basic spin of dimension $16$. 
\item $H = \AAA_6$, $p=3$, $G \cong \AAA_5$, $\dim V = 3$. 
\item $n=6$, $p=2$, $S \cong \AAA_5$, $V$ is basic spin of dimension $4$.
\end{enumerate}
%$G$ must satisfy one of the conclusions of Proposition \ref{PPrim}.  
Moreover, in the cases described in {\rm (iii)-(vi)} the restriction $V\da_G$ is indeed irreducible. 
\end{Theorem}

\begin{proof}
(a) Applying Lemma \ref{Lr1} and Proposition \ref{PBound12}(i) we deduce
\begin{equation}\label{eq-ext2}
  \dim V \geq \left\{ \begin{array}{ll}(n^2-5n+2)/2,& \mbox{ if }n \geq 17,\\ 
  (n^3-9n^2+14n)/12, & \mbox{ if }n \geq 23. \end{array} \right.
\end{equation}

\smallskip
(b) Let $\la$ be a $p$-regular partition of $n$ such that $V$ is an irreducible constituent of $D^\la\da_H$. Note that $\la=\la^\Mull$ if $p=3$. 
If $p=2$ and $V$ is basic spin then $\dim V \geq 2^{(n-4)/2}$. If $V$ does not extend to $\SSS_n$, then by Lemma~\ref{L4.3},
$$
\la_1 \leq
\left\{
\begin{array}{ll}
(n+2)/2 &\hbox{if $p=2$,}\\
(n+4)/2 &\hbox{if $p=3$.}
\end{array}
\right.
$$ 
By Theorems~\ref{TBound2} and \ref{TBound2}, we have in all cases
\begin{equation}\label{eq-ext1}
  \dim V \geq \left\{ \begin{array}{ll} 2^{(n-8)/4}, & \mbox{ if }p=3,\\
  2^{(n-4)/2}, & \mbox{ if }p=2. \end{array} \right.
\end{equation}
%The assumptions on $V$ also imply that $\la \notin \L^{(2)}(n)$. Hence, a
Since $V\dar_G$ is irreducible, it follows that $|G| > \dim(V)^2 \geq 2^{n/2-4}$, and so one of the conclusions of  Proposition \ref{PPrim} must hold. The case (i) of Proposition~\ref{PPrim} leads to the exception (i) of the theorem. 

\smallskip
(c) Suppose we are in the case (iii) of Proposition \ref{PPrim}. As mentioned in the proof of Proposition \ref{PPrim}, we have that 
$G \leq \Aut(S^2) \cong \Aut(S)^2 \rtimes \CCC_2$, and so 
$$\dim V \leq \bmax_p(\Aut(S)^2 \rtimes \CCC_2) \leq 2\bmax_p(\Aut(S))^2.$$
Checking $\bmax_p(\Aut(S))$ using \cite{GAP}, we see that
$$\dim V \leq \left\{ \begin{array}{ll} 2\cdot 189^2, & a=9,\\
    2\cdot 80^2, & a=8,\\
    2\cdot 20^2, & a=7,\\
    2\cdot 16^2, & a=6,\\
    2\cdot 6^2, & a=5. \end{array} \right.$$
As $n = a^2$, this contradicts \eqref{eq-ext1} when $7 \leq a \leq 9$ and \eqref{eq-ext2} when $a = 5,6$.
    
\smallskip    
(d) Finally, we consider the case (ii) of  Proposition \ref{PPrim}, so that $S \unlhd G \leq \Aut(S)$, and either $S \cong \AAA_m$ with
$m \leq 11$, or $S$ satisfies Proposition \ref{PSimple}(ii)--(v).
As $G$ is primitive, $S = \soc(G)$ is transitive on 
$\{1,2, \ldots, n\}$. Also $n \geq P(S)$ by Lemma~\ref{LEasy}.

(d1) Assume $S = Sp_6(2)$, so that $G=S$ and $n \geq P(S) = 28$ \cite{Atlas}. On the other hand, $\dim V \leq \bmax(G) = 512$, contradicting \eqref{eq-ext2}.

(d2) $S = SU_3(3)$. The argument is similar to (d1) but using 
$P(S) = 28$ and $\bmax(G) \leq 64$. 

(d3) $S = SU_4(2)$. The argument is similar to (d1) but using  $P(S) = 27$ and $\bmax(G) \leq 80$. 

%\smallskip
%(d2) 
(d4) Suppose $S = PSL_2(q)$ with $q = r^f\leq 37$ for a prime $r$, so that $\Aut(S) \cong PGL_2(q) \rtimes \CCC_f$. 

(d4.1) If $q \geq 16$, then $n \geq P(S) = q+1 \geq 17$, see \cite[Table 5.2.A]{KlL}, and by \cite{Lu1} we have 
$$\dim V \leq \bmax(G) \leq f\bmax(PGL_2(q)) =f(q+1) \leq q(q+1)/4 < (n^2-5n+2)/2,$$
which contradicts \eqref{eq-ext2}. 

(d4.2) If $q = 13$ (resp. $11$) then $n \geq 14$ (resp. $n \geq 11$), and
$\dim V \leq \bmax(G) = q+1$. In any of these two cases, we see that $n = 14$ (resp. $n \in \{11,12\}$). Furthermore,
since $V$ has dimension $\leq q+1$, it extends to $\SSS_n$ (see \cite{GAP}) and is not basic spin, a contradiction.

(d4.3) Suppose $S = PSL_2(9) \cong \AAA_6$. As $S \neq \AAA_n$, we have $n \in \{10,15\}$ or $n \geq 20$, and $\dim V \leq \bmax_p(\Aut(S)) \leq 16$, 
see \cite{Atlas} and \cite{GAP}. It follows from \eqref{eq-ext2} that $n \in \{10,15\}$. In this case, any irreducible $\F \AAA_n$-module  of dimension $\leq 16$ extends to 
$\SSS_n$, a contradiction. On the other hand, the basic spin modules of $H$ are of dimension $16$, and their Brauer characters take value $-2$ 
at elements $x \in G$ of order $3$ and value $1$ at elements $y \in G$ of order $5$ (see \cite{GAP}), hence they are irreducible over $G$ whenever
$G \not\leq S \cdot 2_2 \cong PGU_2(9)$, leading to the exception (iv).

(d4.4) Suppose $S = SL_2(8)$. Here, $n = 9$ or $n \geq 28$, and $\dim V \leq \bmax(G) \leq 27$. It follows from \eqref{eq-ext2} that
$n = 9$. In this case, the only irreducible $\F\AAA_9$-modules %$V \in \IBr_p(H)$ 
of dimension $\leq 28$ that do not extend to $\SSS_9$ are the two $2$-modular basic spin modules
of dimension $8$, and one can check that exactly one of them is irreducible over $G$ (namely the one whose Brauer character takes value 
$-1$ on elements of order $9$ in $S$), leading to the exception (iii).

(d4.5) Suppose $S = PSL_2(7) \cong SL_3(2)$. Here, $n \in \{7,8,14\}$ or $n \geq 21$, and $\dim V \leq \bmax(G) \leq 8$. It follows from \eqref{eq-ext2} that
$n \in \{7,8\}$. In these cases, the only irreducible $\F \AAA_n$-modules of dimension $\leq 8$ that do not extend to $\SSS_n$ are the $2$-modular basic spin modules of dimension $4$, which restrict reducibly to $G$. Likewise, the basic spin modules of $\SSS_n$ are reducible over $G$ (as can be seen by checking the 
value of the Brauer characters at elements of order $3$ in $G$, see \cite{GAP}).

(d4.6) Suppose $S = PSL_2(5) \cong \AAA_5$. As $S \neq \AAA_n$, we have $n \in \{6,10\}$ or $n \geq 15$, and $\dim V \leq \bmax(G) \leq 6$, 
see \cite{Atlas}. It follows that $n = 6$. In this case, the irreducible $\F\AAA_6$-modules of dimension $\leq 6$ that do not extend to 
$\SSS_6$ are the $3$-modular modules of dimension $3$, and they restrict irreducibly to $G$, yielding the exception (v). Next, the basic spin modules of $H=\AAA_6$ or $\SSS_6$ are of dimension $4$, and their Brauer character takes value $1$ at elements of order $3$ in $G$, whence they are irreducible over $G$, leading to the exception (vi).

%\smallskip
%(d3) 
(d5) Suppose $S = SL_3(q)$ with $q\leq 5$. The case $q=2$ is treated in (d4.5).

(d5.1) If $q=5$, then $n \geq P(S) = 31$ and 
$$\dim V \leq \bmax(G) \leq 310 < (n^2-5n+2)/2,$$ 
contradicting \eqref{eq-ext2}.

(d5.2) If $q=4$, then $n \geq P(S) = 21$ and 
$$\dim V \leq \bmax(G) \leq \bmax(PGL_3(4) \cdot \CCC_2^2) \leq 420 < (n^3-9n^2+14n)/12.$$
This contradicts \eqref{eq-ext2} unless $n \leq 22$. As $n$ divides $|S|$, we conclude that $n = 21$, whence
$G \leq PGL_3(4) \rtimes \CCC_2$ and so $\bmax(G) \leq 128 < (n^2-5n+2)/2$, again contradicting \eqref{eq-ext2}.

(d5.3) If $q=3$, then $n \geq P(S) = 13$, and 
$$\dim V \leq \bmax(G) \leq 52 < (n^2-5n+2)/2.$$
This contradicts \eqref{eq-ext2} unless $n \leq 16$. Inspecting the list of maximal subgroups of $S$ \cite{Atlas}, we conclude that $n = 13$, whence
$G = S$ and $\dim V \leq \bmax_p(S) \leq 27$. But then $V$ extends to $\SSS_{13}$ and is not basic spin.

(d6) Suppose $S = PSL_4(3)$. Then $n \geq P(S) = 40$, and  
$$\dim V \leq \bmax(\Aut(S)) \leq 4\bmax(S) = 4160 <  (n^3-9n^2+14n)/12,$$
violating \eqref{eq-ext2}.

(d7) Suppose $S = SL_5(2)$. Then $n \geq P(S) = 31$, and  
$$\dim V \leq \bmax_p(\Aut(S)) \leq 1024 <  (n^3-9n^2+14n)/12,$$
violating \eqref{eq-ext2}.

(d8) Suppose $S = SL_6(2)$. Then $P(S) = 63$ \cite[Table 5.2.A]{KlL} and $\bmax(G) \leq 66960$ \cite{GAP}. Hence $n \leq 72$ by \eqref{eq-ext1}.
Note that if $K$ is a proper subgroup of $S$, then either $[S:K] \geq 651$, or $K$ is contained in a maximal subgroup $M \cong \CCC_2^5 \rtimes SL_5(2)$ of index 
$63$ in $S$, see \cite{GAP}. Hence, if we take $K = \operatorname{Stab}_S(1)$ in the action of $S$ on $\{1,2, \ldots,n\}$ (so that $[S:K] = n$), then we must have that
$n = [S:M] = 63$, and in fact $S$ acts on $n$ points $\{1,2, \ldots,n\}$ via its action on $63$ lines or $63$ hyperplanes of 
$\F_2^6$. None of these actions can be extended to $\Aut(S)$, so we have that $G = S$. If $p=2$,
%Then $\la_1 \leq (n+2)/2$ by the proof of \cite[Proposition 4.1]{KST}, and so 
%$\la_1 \leq 32$. This in turn implies by Theorem \ref{TBound2} that 
then \eqref{eq-ext1} implies that $\dim V > 2^{29} > \bmax(G)$, a contradiction. Suppose 
$p=3$. Then $\la_1 \leq (n+4)/2$ by Lemma~\ref{L4.3}, whence $\la_1 \leq 33$. 
On the other hand, Proposition \ref{PBound12}(ii) implies that $\la \in \L^{(3)}(n)$, and so $\la_1 \geq 60$, a 
contradiction.

\smallskip
(e) Suppose that $S\cong \AAA_m$ with $m\leq 11$. The case $m=5$ and $6$ are considered in (d4.6) and (d4.3), respectively. 

\smallskip
(e1) Let $m=11$. As $S \neq \AAA_n$, we have $n \geq 55$ and 
$\dim V \leq \bmax(\SSS_{11}) = 2310$, see \cite{Atlas}. This contradicts \eqref{eq-ext2}.

(e2) Let  $m=10$. This case is treated similarly to (e1) observing that $n \geq 45$ and 
$\dim V \leq \bmax(\SSS_{10}) = 768$.

(e3) Let  $m=9$. This case is treated similarly to (e1) observing that $n \geq 36$ and 
$\dim V \leq \bmax(\SSS_9) = 216$.

(e4) Let $m=8$. As $S \neq \AAA_n$, we have $n = 15$ or $n \geq 28$, and 
$\dim V \leq \bmax(\SSS_8) = 90$, see \cite{Atlas}. It follows from \eqref{eq-ext2} that $n = 15$ and $G=S$. Now, 
the only irreducible $\F\AAA_{15}$-modules of dimension $\leq 90$ that does not extend to $\SSS_{15}$ are the two basic spin modules of dimension $64$. If $\varphi$ is the Brauer character of any of these two modules and $g \in S$ is a $3$-cycle, then $g$ becomes a
disjoint product of five $3$-cycles in the doubly transitive embedding $S \hookrightarrow \AAA_{15}$, and so $\varphi(g) = -2$. However,
the unique irreducible $2$-Brauer character of $S$ of degree $64$ takes value $4$ at $g$, and so $\varphi\da_G$ is reducible. 

(e4) Let $m=7$. As $S \neq \AAA_n$, we have $n = 15$ or $n \geq 21$, and 
$\dim V \leq \bmax(\SSS_7) = 35$, see \cite{Atlas}. It follows from \eqref{eq-ext2} that $n = 15$. Now, all 
irreducible $\F\AAA_{15}$-modules of dimension $\leq 35$ extend to $\SSS_{15}$ and are not basic spin. 

\smallskip
(f) Let $S$ be a Mathieu group. Suppose that the conclusion (ii) of the current theorem does {\it not} hold.
Then, if $S = M_{24}$, we have by \cite{Atlas} that $n \geq 276$ and $\dim V \leq \bmax(G) = 10395$, contradicting
\eqref{eq-ext2}. If $S = M_{23}$, we have by \cite{Atlas} that $n \geq 253$ and $\dim V \leq \bmax(G) = 2024$, again contradicting
\eqref{eq-ext2}. The same argument applies to $S = M_{22}$, where we have  $n \geq 77$ and 
$\dim V \leq \bmax(\Aut(S)) = 560$, to $S = M_{12}$, for which we have $n \geq 66$ and 
$\dim V \leq \bmax(\Aut(S)) = 176$, and to $S = M_{11}$, for which we have $n \geq 55$ and 
$\dim V \leq \bmax(G) = 55$.
\end{proof}

Note that cases (i) and (ii) of Theorem \ref{ext-spin} will be settled in Theorem \ref{TAGL} and Theorem \ref{T2small}.

\begin{Proposition}\label{PSmall}
Let 
$p=2$ or $3$, $\la\in\Par_p(n)$, $H = \AAA_n$ or $\SSS_n$, and $V$ be an irreducible constituent of $D^\la\da_H$, with $\dim V>1$. 
Suppose that $G < H$ is a doubly transitive subgroup such that  $S := \soc(G) < \AAA_n$ is one of the following simple groups:
\begin{enumerate}[\rm(a)]
\item $\AAA_m$ with $5 \leq m \leq 7$;
\item $PSL_2(q)$ with $7 \leq q \leq 9$;
\item $PSL_3(q)$ with $2 \leq q \leq 4$;
\item $SL_4(2) \cong \AAA_8$.
\end{enumerate}
Then $V\da_G$ is irreducible if and only if one of the following happens:
\begin{enumerate}[\rm(i)]
\item $\la \in \L^{(1)}(n)$; furthermore, $(G,n,p)$ fulfills the conditions set in Table II.
\item $S \cong \AAA_5 \cong SL_2(4) \cong PSL_2(5)$, $n = 6$, $p=3$, $\la = \la^\Mull = (4,1^2)$, and $(G,H,\dim V) = (\AAA_5,\AAA_6,3)$ or $(\SSS_5,\SSS_6,6)$.
\item $S \cong \AAA_6 \cong PSL_2(9)$, $n = 10$, $p=2$, $V = D^{(6,4)}\da_H$ of dimension $16$, and $G \leq \Aut(S)$ but $G \not\leq PSU_2(9) = S \cdot 2_2$.
\item $S = SL_2(8)$, $G = S$ or $G=\Aut(S) \cong P\Gamma L_2(8)$, $H = \AAA_9$, $p=2$, $\dim V = 8$, and $V$ is the
only one of $E^{\be_9}_{\pm}$ whose Brauer character takes value $-1$ at elements of order $9$ in $SL_2(8)$.
\item $S = SL_2(8)$, $G = \Aut(S) \cong P\Gamma L_2(8)$, $n=9$, $p=3$, $\dim V = 27$, and $\la$ or $\la^\Mull$ equals $(7,2)$.
\end{enumerate}
%Conversely, all the listed cases give rise to examples.
\end{Proposition}

\begin{proof}
The `if-part' in (i) follows \cite{Mortimer}, and in (ii)--(v) from \cite{GAP}. Conversely, suppose $V\da_G$ is irreducible. We may assume that $\la\not\in\L^{(1)}(n)$ again by \cite{Mortimer}. 

If $S = \AAA_5 \cong PSL_2(5)$ then $n=6$, and 
%As $V\da_G$ is irreducible, we see by using 
by \cite{GAP}, 
$\dim V \leq 4$ if $p=2$ and $\dim V \leq 6$ if $p=3$. As $\la \notin \L^{(1)}(n)$, we deduce that $p=3$, $\dim D^\la = 6$, and arrive at (ii).

If $S = \AAA_6 \cong PSL_2(9)$ then $n=10$, and by  \cite{GAP}, 
$\dim V \leq 16$ if $p=2$ and $\dim V \leq 9$ if $p=3$. As $\la \notin \L^{(1)}(n)$, we deduce  
that $p=2$, $\dim V = 16$, and arrive at (iii).

If $S = \AAA_7$, then $n=15$ and $G=S$, and by \cite{GAP}, 
$\dim D^\la \leq 2(\dim V) \leq 40$, whence $\la \in \L^{(1)}(n)$, a contradiction. 

If $S = \AAA_8 \cong SL_4(2)$ then $n=15$, $G=S$, and by  \cite{GAP}, 
$\dim V \leq 64$, hence $p=2$ and $\dim V = 64$. Note that $\IBr_2(G)$ contains a unique character of degree $64$, which 
takes value $4$ at an element of class $3A$ in $\AAA_8$, which belongs to class $3D$ in $\AAA_{15}$, in the notation of \cite{GAP}. However, any 
character in $\IBr_2(\AAA_{15})$ of degree $64$ takes value $-2$ at class $3D$, and $\IBr_2(\SSS_{15})$ contains no character of degree $64$,
a contradiction.

If $S = SL_3(2) \cong PSL_2(7)$ then $n=7$ or $8$, and by  \cite{GAP}, 
$\dim D^\la = \dim V \leq 8$. If $p=3$, then we conclude using \cite{GAP} that $\la \in \L^{(1)}(n)$, a contradiction. Let  $p=2$.  Since $\dim V \leq 8$ and $\la \notin \L^{(1)}(n)$, we deduce that $(H,\dim V)$ is either $(\AAA_n,4)$ or $(\SSS_n,8)$. Checking the 
degrees of characters in $\IBr_2(G)$ we see that $\dim V = 8$.
%and so $H = \SSS_n$, whence $G = S \rtimes \CCC_2 \cong PGL_2(7)$.
%Now, if $n = 7$ then $G=S$, a contradiction. If $n = 8$, then indeed $G$ is not contained in $\AAA_8$ \cite{Atlas}.
Now, any irreducible $2$-Brauer character of degree $8$ of $H$ takes value $-4$ or $2$ at elements of order $3$, whereas
any irreducible $2$-Brauer character of degree $8$ of $G$ takes value $-1$ at elements of order $3$, a contradiction. 

If $S = SL_2(8)$ then $n=9$ or $28$, $S \unlhd G \leq \Aut(S) \cong P\Gamma L_2(8) \cong S \rtimes \CCC_3$, and 
by \cite{GAP}, 
$\dim V \leq 12$ if $p=2$ and $\dim V \leq 27$ if $p=3$. In particular, if $n = 28$, then $\la \in \L^{(1)}(n)$ a contradiction. 
%$p=3$, and $G = \Aut(S)$.
If $n = 9$ then, by \cite{GAP}, remembering that $\la \notin \L^{(1)}(n)$, we can check that (iv) occurs if $p=2$ and (v) occurs if $p=3$.

If $S = SL_3(3)$ then $n=13$,  $G=S$, and by \cite{GAP}, 
$\dim V \leq 27$, whence $\la \in \L^{(1)}(n)$, a contradiction. 

If $S = PSL_3(4)$ then $n=21$, $G\leq PGL_3(4) \rtimes \CCC_2$, and by \cite{GAP}, we have 
$\dim V \leq 2 \cdot 64 = 128 < (n^2-5n+2)/2,$
whence $\la \in \L^{(1)}(n)$ by Lemma \ref{Lr1}, a contradiction. 
\end{proof}

We will also need the following extension of Theorem \ref{ext-spin} to some subgroups of $\SSS_n$ that are not
primitive:

\begin{Proposition}\label{PSmall2}
Let $p=2$, $H =\AAA_n$ or $\SSS_n$, and let $G < H$ be an almost simple subgroup with $S = \soc(G)$ that is not 
primitive in $H$. 
%Assume that $V\da_G$ is irreducible for a basic spin $\F H$-module $V$. 
Assume in addition that $(S,n)$ satisfies one of the following conditions: 
\begin{enumerate}[\rm(i)]
\item $S \cong \AAA_m$ and $(m,n)$ is $(5,\leq 10)$, $(6, \leq 14)$, $(7, \leq 16)$ or $(8, \leq 19)$. 
%$m = 5$, $6$, $7$ or $8$, and $n \leq 10$, $14$, $16$, or $19$, respectively. 
Moreover, if $m = 7$ or $8$, then some orbit of $S$ on $\Omega := \{1,2, \ldots,n\}$ has 
length $> m$. 

\item $S \cong PSL_{2}(q)$ and $(q,n)$ is $(7,\leq 12)$, $(8, \leq 14)$, $(11, \leq 14)$, $(13, \leq 15)$ or $(16, \leq 17)$. 

%$q = 7$, $8$, $11$, $13$, or $16$, and $n \leq 12$, $14$, $14$, $15$, or $17$, respectively.

\item $(S,n)$ is $(SL_{3}(3),\leq 17)$ or $(PSL_{3}(4),21)$.
%$S \cong SL_{3}(3)$ or $PSL_{3}(4)$, and $n \leq 17$ or $21$, respectively. 

\item $S \cong M_{t}$ with $t = 11$, $12$, $22$, $23$, or $24$. 
%and $n \leq 16$, $21$, $23$, $27$, or $31$, respectively.
\end{enumerate}
Then $V\da_G$ is irreducible for a basic spin $\F H$-module $V$ if and only if one of the following holds:
\begin{enumerate}[\rm(a)]
\item $H = \SSS_6$ and $G = \SSS_5$ fixes one letter.
\item $\SSS_5 \cong G = \AAA_{5,2} < H = \AAA_7$ or $\SSS_5 \cong G = \AAA_{5,2,1} < H = \AAA_8$.
\item $H=\AAA_7$ or $\AAA_8$, and $S = \AAA_6$. 
\item $H = \AAA_7$ or $\AAA_8$, and $S = \AAA_5$ acts $2$-transitively on $\{1,2, \ldots,6\}$.
\item $H = \AAA_{11}$, and $G = M_{10}< \AAA_{10}$ acts $2$-transitively on $\{1,2, \ldots,10\}$.
\item $H = \AAA_{12}$ and $G \in \{\SSS_6,M_{10},\Aut(\AAA_6)\}$ acts $2$-transitively on $\{1,2, \ldots,10\}$.
\item $H = \AAA_{12}$, $G \in \{\SSS_6,M_{10},\Aut(\AAA_6)\}$, and $S$ acts on $\{1,2, \ldots,6\}$ and $\{7,8, \ldots,12\}$ via two
inequivalent $2$-transitive actions.
\item $H = \AAA_{12}$ and $G=M_{11} < \AAA_{11}$.
\end{enumerate}
%Conversely, all the listed possibilities give rise to examples.
\end{Proposition}

\begin{proof}
We will prove the `only-if-part'. The `if-part' is then an easy explicit check. 
Set $\Omega := \{1,2, \ldots,n\}$. 
Let $U$ be an irreducible summand of $V\da_S$. If $U$ is trivial, then $S$ acts trivially on $V$ by Clifford's theorem, contradicting the faithfulness
of the $\F H$-module $V$. Thus we have 
\begin{equation}\label{eq-d}
  \dim U = 2^a,~\dim V \geq 2^{(n-4)/2}
\end{equation}
for some $a \in \ZZ_{\geq 1}$. 
%, and let $P(S)$ denote the smallest index of proper subgroups of $S$. 
Since $G$ is not primitive, we have by Lemma~\ref{LEasy}(i):
\begin{equation}\label{eq-n}
  n \geq P(S)+1.
\end{equation}   

First, we consider the case (iv). Then \eqref{eq-d} implies by \cite{GAP} that $\dim U = 16$ and $t = 11$ or $12$. If $t=11$, 
then $G = S$, $V=U$, $\dim V = 16$, and $n \geq 12$ by \eqref{eq-n}. It follows that $H = \AAA_{12}$ and $G < \AAA_{11}$, leading to (h).
If $t = 12$, then $\dim V \leq 32$ since $G/S \hookrightarrow \Out(S) \leq \CCC_2$, whereas $n \geq 13$ by \eqref{eq-n}. It follows 
that $H=\AAA_{13}$, $\dim V = 32$, $G = \Aut(M_{12}) \leq \AAA_{13} \cap \SSS_{12} = \AAA_{12}$. The latter implies 
that $V = E^{\be_{13}}_{\pm}$ is irreducible over $\AAA_{12}$, a contradiction.

\smallskip
Next suppose we are in the case (iii). If $S = PSL_3(4)$, then $P(S) = 21$, violating \eqref{eq-n}. If $S = SL_3(3)$, then 
$n \geq 14$ by \eqref{eq-n}, whereas $\dim U = 16$ by \cite{GAP}, and so $\dim V \leq 32$, a contradiction.  

\smallskip
Consider now the case (ii). Then $q \neq 16$ because of \eqref{eq-n}, and $q \neq 11,13$ by \eqref{eq-d} and \cite{GAP}. If $q=8$,
then $S \leq G \leq \Aut(S)$ and so $\dim V \in \{2,4,8\}$ by \cite{GAP}. On the other hand, $n \geq 10$ by \eqref{eq-n}, and this is impossible
since $\dim V \geq \dim E^{\be_{10}} = 16$. Thus $q=7$, in which case $\dim V = 8$ by \cite{GAP} and $n \geq 8$ by \eqref{eq-n}. The condition
$\dim V = 8$ implies that $H \in \{\SSS_8,\AAA_9\}$. If $H = \SSS_8$, then the Brauer character of $V$ can take values $-4$ or $2$ at elements of
order $3$, whereas any $\psi \in \IBr_2(G)$ of degree $8$ takes value $-1$ at elements of order $3$, a contradiction. Thus $H = \AAA_9$. Note 
that if we embed $S = PSL_2(7)$ in $H$ via a transitive embedding $S < \AAA_7$ (so fixing two more points) or a transitive embedding 
$S < \AAA_8$, then any element $g \in S$ of order $3$ will fix $3$ points, and so the Brauer character of $V$ takes value $2$ at $g$, again a contradiction.

\smallskip
Finally, suppose we are in the case (i). 

Assume first that $S = \AAA_8$. 
Then $\dim V = 4$, $8$, or $64$ by \cite{GAP}.
%, and $n \geq 9$ by \eqref{eq-n}. 
%The latter rules out the possibility $\dim V = 4$. If $\dim V = 8$, then $H = \AAA_9$, leading to (*). Suppose $\dim V = 64$.
Since $S$ is not primitive on $\Omega$ and has an orbit of length $> 8$ on $\Omega$, we have by \cite{Atlas} 
that $n \geq 16$, whence $H = \AAA_{16}$, 
$\dim V = 64$, and $G=S$ is a $2$-transitive subgroup of $\AAA_{15}$.  In this embedding, a $3$-cycle $g \in S$ will have $1$ fixed point on $\Omega$,
so the Brauer character of $V$  takes value $-2$ at $g$. But any $\psi \in \IBr_2(S)$ of degree $64$ takes value $4$ at $g$, a contradiction.

Next let $S = \AAA_7$. Then $\dim V = 4$ or $8$ by \cite{GAP}.
As $S$ is not primitive on $\Omega$ and has an orbit of length $> 7$ on $\Omega$, we have by \cite{Atlas} that $n \geq 16$, contradicting 
\eqref{eq-d}.

\smallskip
Now let $S = \AAA_6$. Then $\dim V = 4$, $8$, or $16$ by \cite{GAP}, and $n \geq 7$ by \eqref{eq-n}. It follows that 
$H = \AAA_n$ with $7 \leq n \leq 12$, or $H = \SSS_n$ with $7 \leq n \leq 10$. 

Assume first that some $S$-orbit on $\Omega$ has length
$l>6$. As $S$ is not primitive, we must then have that $l = 10$, $H = \AAA_n$ with $n = 11$ or $12$, and $\dim V = 16$. In either case, we may assume 
that $S$ acts $2$-transitively on $\{1,2, \ldots,10\}$ and fixes $11$, and also $12$ if $n=12$. In the $2$-transitive embedding 
$\AAA_6 \hookrightarrow \SSS_{10}$, elements of $S$ of order $3$ acts with one fixed point and elements of order $5$ act fixed-point-freely;
furthermore a point stabilizer in $S$ is just $\NB_S(Q)$ for $Q \in \Syl_3(S)$. 
The embedding extends to $\Aut(\AAA_6)$, with the image of $M_{10} = S \cdot 2_3$ (in the notation of \cite{Atlas} and \cite{GAP}) contained in
$\AAA_{10}$. 
Using the class fusion information, it is easy to check in \cite{GAP} that $E^{\be_{10}}$ is irreducible over $\SSS_6 \cong S \cdot 2_1$, $M_{10}$,
and $\Aut(S)$, but splits into a direct sum of two simple summands over $S$ and $S \cdot 2_2 \cong PGU_2(9)$. Also, 
$$E^{\be_{12}}_{\pm}\da_{\AAA_{11}} \cong E^{\be_{11}}_{\pm},~E^{\be_{11}}_{\pm}\da_{\AAA_{10}} \cong E^{\be_{10}}.$$ 
Hence, if $n=11$, then $G$ fixes $11$ and is contained in the natural $\AAA_{10}$, and 
so $G \cong M_{10}$, leading to (e). Note that we can extend the embedding $\Aut(S) \hookrightarrow \AAA_{10}$ to 
$\Aut(S) \hookrightarrow \AAA_{12}$ uniquely by demanding the involution $(1,2) \in \SSS_6$ to interchange $11$ and $12$. This leads to
(f) when $n=12$.  

Now we consider the case where all orbits of $S = \AAA_6$ on $\Omega$ have length $1$ or $6$, and there is more than one orbit of length $6$. Then $n=12$ and $H = \AAA_{12}$. Let $\pi_1, \pi_2: S \to \AAA_6$ be induced by the action of $S$ on its two orbits on $\Omega$, and let
$\psi_i$ denote the Brauer characters afforded by $E^{\be_6}\da_{\pi_i(S)}$. Then $\psi_i \in \IBr_2(S)$ and has degree $4$; also,
$\psi_i^2$ contains $\F_S$ (with multiplicity $4$). As $G$ is irreducible on $V$ and $V\da_{\AAA_6 \times \AAA_6} \cong E^{\be_6} \boxtimes E^{\be_6}$, 
we see that $\psi_1 \neq \psi_2$. In this case, $\psi_1\psi_2 = \nu_1+\nu_2$, with $\nu_i \in \IBr_2(S)$ of degree $8$, and 
$\operatorname{Stab}_{\Aut(S)}(\nu_1) = S \cdot 2_2$. So as long as $\pi_1$ and $\pi_2$ are inequivalent and $G \not\leq S\cdot 2_2$, $V\da_G$
is irreducible, leading to (g). (Note that such an action exists: for instance, we can embed $\SSS_6$ in $\AAA_{6,6}$, with two inequivalent actions of
$\SSS_6$ on the first and the last six letters.)

Finally, we consider the case where $S=\AAA_6$ has exactly one orbit $\{1,2, \ldots,6\}$ of length $6$ and fixes each of $7, \ldots,n$; in particular,
$G \leq \SSS_{6,n-6}$. If $H = \SSS_n$ and $n \geq 8$, then it follows that $D^{\be_n}$ is irreducible over $\SSS_{6,n-6}$, contradicting 
\cite[Proposition 2.15]{KMT1}. The case $H = \SSS_7$ is also impossible by dimension consideration. Suppose $H = \AAA_n$ and let
$n_2 \geq n_3 \geq \ldots \geq n_h \geq 1$ denote the lengths of $G$-orbits on $\{7, \ldots,n\}$. Then $G \leq \AAA_\nu$ for $\nu := (6,n_2, \ldots,n_h)$.
Since $G/S \leq \CCC_2^2$, $n_i \in \{1,2,4\}$. As $V$ is irreducible over $\AAA_\nu$, we see by \cite[Proposition 6.3]{KMT3} that
$\nu = (6,1)$, $(6,1,1)$, or $(6,2)$, and arrive at (c).

\smallskip
Finally, let $S = \AAA_5$. Then $\dim V = 2$ or $4$ by \cite{GAP}, and $n \geq 6$ by \eqref{eq-n}. It follows that 
$H = \AAA_n$ with $6 \leq n \leq 8$, or $H = \SSS_6$. 

If $H = \SSS_6$, then $G \leq \SSS_5$ as $G$ is 
not primitive, and we arrive at (a). 

Suppose $H = \AAA_n$ but some $S$-orbit has length $l>5$. As $S$ is not primitive and 
$n \leq 8$, we have that $l=6$, $n=7$ or $8$, and $S$ has one orbit $\Omega':=\{1,2, \ldots,6\}$ and fixes the remaining letters. In this action,
elements of order $3$ in $S$ act fixed-point-freely on $\Omega'$. Using this class fusion information, we can check in \cite{GAP} that 
$V\da_S$ is irreducible (of dimension $4$), giving rise to (d).

Finally, assume $H = \AAA_n$ and $S=\AAA_5$ has only orbits of length $5$ and $1$ on $\Omega$. Then we may assume that $S$ has one orbit 
$\{1,2, \ldots,5\}$ and fixes each of $6, \ldots, n$. Again let $n_2 \geq n_3 \geq \ldots \geq n_h \geq 1$ denote the lengths of $G$-orbits on $\{6, \ldots,n\}$. Then $G \leq \AAA_\nu$ for $\nu := (5,n_2, \ldots,n_h)$.
Since $G/S \leq \CCC_2$, $n_i \in \{1,2\}$. As $V$ is irreducible over $\AAA_\nu$, we see by \cite[Proposition 6.3]{KMT3} that
$\nu = (5,2,1)$, or $(5,2)$, and arrive at (b).
\end{proof}

\section{The small doubly transitive groups}\label{SSmall}
\label{SSmallGroups}
Recall that we call a doubly transitive subgroup $G < \SSS_n$  small, if $S = \soc(G)$ is non-abelian, $S \not\cong \AAA_n$, $S \not\cong PSL_m(q)$ when $n = (q^m-1)/(q-1)$, and $S \not\cong Sp_{2m}(2)$ when $n = 2^{m-1}(2^m \pm 1)$. 

All small $2$-transitive subgroups are handled in the following theorem:

\begin{Theorem}\label{T2small}
Let $p=2$ or $3$, $H =\AAA_n$ or $\SSS_n$, and $W$ be an irreducible summand of $D^\la\dar_H$ for some $\la \in\Par_p(n)\setminus \L^{(1)}(n)$. Let $G<H$ be a small doubly transitive subgroup. Then $W\dar_G$ is irreducible if and only if one of the following cases occurs.
\begin{enumerate}[\rm(i)]
\item $G=M_{11}$, $\AAA_n \leq H \leq \SSS_n$, and one of the following happens:
\begin{enumerate}[\rm(a)]
\item $p=2$, $n=11$ or $12$, $\la = (n-2,2)$, and 
$W = D^\la\dar_H$ has dimension $44$;
\item $p=3$, $n=11$, $\la$ or $\la^\Mull$ is $(9,1^2)$, 
%$D^\la$ or $D^\la \otimes \sgn$ is isomorphic to $D^{(9,1^2)}$, 
and $W = D^\la\dar_H$ has dimension $45$.
\end{enumerate}

\item $G = M_{11}$,  $p=2$, $H = \AAA_n$, $n=11$ or $12$, and $W=E^{\be_{n}}_\pm$ has dimension $16$. 

%and $\dim W = (\dim D^\la)/2 = 16$. Furthermore, either $n = 11$ and $\la = (6,5)$, or  
%$n = 12$ and $\la = (7,5)$.

\item $G = M_{12}$, $p = 2$, $n =12$, and one of the following happens:
\begin{enumerate}[\rm(a)]
\item $\AAA_{12} \leq H \leq \SSS_{12}$, $\la = (10,2)$, and 
$W = D^\la\dar_H$ has dimension $44$;
\item $H= \AAA_{12}$, $\la = \be_{12}$, and $W = E^{\be_{12}}_\pm$ has dimension  $16$;
\item $H = \AAA_{12}$, $\la = (6,5,1)$, and $W =  E^\la_\pm$ has dimension $144$.
\end{enumerate}

\item $G = M_{12}$, $p = 3$, $n=12$, $\AAA_{12} \leq H \leq \SSS_{12}$, and one of the following happens:
\begin{enumerate}[\rm(a)]
\item $\la$ or $\la^\Mull$ is $(10,1^2)$,  
and $W = D^\la\dar_H$ has dimension $45$;
\item $\la$ or $\la^\Mull$ is $(10,2)$, 
and $W = D^\la\dar_H$ has dimension $54$.
\end{enumerate}

\item $M_{22} \leq G \leq \Aut(M_{22})$, $p=3$, $n=22$, $\AAA_{22} \leq H \leq \SSS_{22}$,  
$\la$ or $\la^\Mull$ is $(20,1^2)$, and $W = D^\la\dar_H$ has dimension $210$.

\item $G = M_{23}$, $p=3$, $n = 23$, $\AAA_{23} \leq H \leq \SSS_{23}$, 
$\la$ or $\la^\Mull$ is $(21,1^2)$, and $W = D^\la\dar_H$ has dimension $231$.

\item $G = M_{24}$, $p = 3$, $n=24$, $\AAA_{24} \leq H \leq \SSS_{24}$, and one of the following happens:
\begin{enumerate}[\rm(a)]
\item $\la$ or $\la^\Mull$ is $(22,1^2)$,  
and $W = D^\la\dar_H$ has dimension $231$;
\item $\la$ or $\la^\Mull$ is $(22,2)$, 
and $W = D^\la\dar_H$ has dimension $252$.
\end{enumerate}
\end{enumerate}
\end{Theorem} 

\begin{proof}
%Suppose $G < \SSS_n$ is a small $2$-transitive subgroup and $\la \notin \L^{(1)}(n)$. 
If $S:=\soc(G)$ is a not a Mathieu group or $Co_3$, then the arguments in \cite[Section 5]{BK}, but using Lemma \ref{Lr1} instead of \cite[Lemma 1.18(i)]{BK}, show that $D^\la\dar_G$ is reducible. We now consider the remaining possibilities for $S$. 
%, and assume $V \in \IBr_p(H)$ is irreducible over $G$, with $H \in \{\AAA_n,\SSS_n\}$ and $V$ is contained in $D^\la\dar_H$, still with $\la \notin \L^{(1)}(n)$. 
Replacing $\la$ by $\la^\Mull$ if necessary, we will assume that 
$\la_1\geq \la^\Mull_1$.

\smallskip
{\sf Case 1}: {\it $G = M_{11}$ in transitive representations of degrees $n = 11$ and $12$.}
%\label{m11}

By comparing the traces in these transitive representations \cite{GAP}, one can see that the classes 
$2a$, $3a$, $4a$, $5a$, $8ab$, $11a$ of $G$ belong to classes
$2b$, $3c$, $4b$, $5b$, $8a$, $11a$ in $\AAA_{11}$, and 
$2c$, $3d$, $4d$, $5b$, $8b$, $11a$ in $\AAA_{12}$.

Let $p = 2$. According to \cite{GAP}, any $\varphi \in \IBr_2(G)$ has degree $1$, $10$, $16$, or $44$; and the degrees 
of characters in $\IBR_p(H)$ are also known. Hence we need to consider 
only the cases where $\dim D^\la = 32$ or $44$. In the latter case, $D^\la\dar_{\AAA_{11}}$ is irreducible and is obtained by 
reducing $S^{(9,2)}_\C$ modulo $2$ (and restricting to $\AAA_{11}$). 
%Since $\chi := S^{(9,2)}\dar_G$ is irreducible by \cite[Main Theorem (iii)]{BK}, and $\chi (\bmod\ 2)$ is irreducible by \cite{GAP}, 
Using the above class fusion, we see that the case $\dim D^\la = 44$ does give rise
to an example (and $\la = (n-2,2)$). If $\dim D^\la = 32$ (and so $\la = (6,5)$, respectively $(7,5)$),
then its restriction to $\AAA_n$ is a direct sum of two simple modules of
dimension $16$, both of which are irreducible over $G$, giving rise to another example with $(\dim V,H) = (16,\AAA_n)$.
 
Let $p = 3$. Using \cite{GAP} as above, when $n=11$ we see that $\dim V = 45$ and $\la = (n-2,1^2)$ (up to tensoring with $\sgn$), yielding another example. There is no example when $n=12$, since $\varphi \in \IBr_3(\AAA_{12})$ of degree $45$ takes value $3$ at 
the class $8b$ of $\AAA_{12}$, whereas  $\psi \in \IBr_3(M_{11})$ of degree $45$ takes value $-1$ at 
the class $8a$ of $M_{11}$.
%since $D^{(10,1^2)}\dar_{\SSS_{11}} =  D^{(9,1^2)}$ is obtained by reducing $S^{(9,1^2)}$ modulo $3$, 
%$\chi := S^{(9,1^2)}\dar_G$ is irreducible by 
%\cite[Main Theorem (iv)]{BK}, and $\chi (\bmod\ 3)$ is irreducible by \cite{GAP}.

\smallskip
{\sf Case 2}: {\it $G = M_{12}$ in permutation representations of degree $n = 12$.}

Let $p = 2$. Using the character degrees of $G$ and $H$ as listed in \cite{GAP}, we need to consider 
only the cases where $\dim D^\la = 32$, $44$ and $288$. We can embed $M_{11}$ into $G$ as a $2$-transitive subgroup of 
$G < \SSS_{12}$. Now the first two cases, with $(\la,H) = ((7,5),\AAA_{12})$ and $((10,2),\AAA_{12} \mbox { or }\SSS_{12}$), give rise to examples, since $V\dar_H$ is irreducible by the results of Case 1.  Next, by restricting $\psi \in \IBr_2(\AAA_{12})$ to $G$, we see
that conjugacy classes $3A$, $3B$, and $5A$ of $G$ as listed in \cite{GAP} correspond to the classes $3D$, $3C$, and $5B$ in $\AAA_{12}$.
It follows that the last case, with $(\la,\dim V,H) = ((6,5,1),144,\AAA_{12})$, gives rise to another example.

Let $p=3$. Using \cite{GAP} as above, we see that $\dim V = 45$ or $54$, and $\la = (10,1^2)$ or $(10,2)$, respectively (up to 
tensoring with $\sgn$). In both cases, $D^\la$ is obtained by reducing $S^\la_\C$ modulo $3$. Since $\chi := S^\la_\C\dar_G$ is irreducible by 
\cite[Main Theorem (iii), (iv)]{BK}, and $\chi (\bmod\ 3)$ is irreducible by \cite{GAP}, both cases give rise to examples. 

\smallskip
{\sf Case 3}: {\it $M_{22} \leq G \leq \Aut(M_{22})$ in permutation representations of degree $n = 22$.}

Let $p=2$. According to \cite{GAP}, any $\varphi \in \IBr_2(G)$ has degree $\leq 140 < (n^2-5n+2)/2$. By Lemma 
\ref{Lr1}, this contradicts the assumption
$\la \notin \L^{(1)}(n)$.
% if $H = \SSS_{22}$. Suppose that $H = \AAA_{22}$. Since the two conjugacy classes of elements of order $11$ of 
%$S$ fuse in $\Aut(S)$ but not in $\AAA_{22}$, we see that $G = S$. It follows by \cite{GAP} that $\dim V \leq 98$. Restricting 
%$V$ to a natural subgroup $H = \AAA_{19}$ of $H$ using \cite{GAP}, we then see that every composition factor of $V\dar_H$ 
% belongs to $\L^{(1)}(19)$. But this again contradicts the condition $\la \notin  \L^{(1)}(n)$ by ***.

Let $p=3$. By \cite{GAP}, any $\varphi \in \IBr_3(G)$ has degree $\leq 231 < (n^3-9n^2+14n)/12$. 
Hence $\dim D^\la < (n^3-9n^2+14n)/6$ and so $\la \in \L^{(2)}(n) \smallsetminus \L^{(1)}(n)$ by Proposition \ref{PBound12}(i). By Lemma \ref{LL2} and by checking the possible dimensions of $V$ in \cite{GAP}, we see that $\dim D^\la = 210$, and that
without any loss we may assume that $D^\la$ is obtained by reducing $S^{(n-2,1^2)}_\C$ modulo $p$. As 
$\chi := S^{(n-2,1^2)}_\C\dar_S$ is irreducible by 
\cite[Main Theorem (iv)]{BK}, and $\chi (\bmod\ 3)$ is irreducible by \cite{GAP}, both cases $G = S$ and 
$G = \Aut(S)$ give rise to examples.   
 
%\color{purple}
%Now we need  $p=2,3$ analogues of Lemmas 1.20 and 1.21 of \cite{BK}.
%\color{black} 
 
\smallskip
{\sf Case 4}: {\it $G = M_{23}$ in permutation representations of degree $n = 23$.}

First we consider the case where $V$ does not extend to $\SSS_n$. As in part (b) of the proof of Theorem \ref{ext-spin},
we have that $\la_1 \leq (n+2)/2$ if $p = 2$ and $\la_1 \leq (n+4)/2$ if $p=3$. Now, if $p=2$ then $\la_1 \leq 12$,
whence $\dim V = (\dim D^\la)/2 \geq 2^{10}$ by Theorem \ref{TBound2}(i), whereas $\bmax_2(G) = 896$ \cite{GAP},
a contradiction. If $p=3$, then $\la_1 \leq 13$ and so
$\dim V = (\dim D^\la)/2 \geq 2783$ by Proposition \ref{PBound12}(ii), contrary to $\bmax_3(G) = 1035$ \cite{GAP}.

Now we consider the case where $V = D^\la\da_H$. Then 
$$\dim D^\la \leq \bmax_p(G) \leq 1035 < (n^3-9n^2+14n)/6,$$ 
whence $\la \in \L^{(2)}(n) \smallsetminus \L^{(1)}(n)$ by Proposition \ref{PBound12}(i), and so $\la_1 \geq n-2$.

Consider the case $D^\la = D^{(n-2,2)}$. If $p=2$, then $D^\la$ is obtained 
by reducing $S^\la_\C$ modulo $p$. Furthermore, $\chi := S^\la_\C\dar_G$ is irreducible by 
\cite[Main Theorem (iii)]{BK}, but $\chi (\bmod\ 2)$ is reducible by \cite{GAP}, ruling out this case.  If $p = 3$, 
then $\dim D^\la = 208$. Since $\la \notin \L^{(1)}(n)$, $D^\la$ is irreducible over $\AAA_n$ by Lemma \ref{Lr1}. 
Since no $\varphi \in \IBr_p(G)$ has degree $208$ \cite{GAP}, there are no examples in this case either.

Suppose now that $p=3$ and $D^\la = D^{(n-2,1^2)}$. Then $D^\la$ is obtained 
by reducing $S^\la_\C$ modulo $p$. Furthermore, $\psi := S^\la_\C\dar_G$ is irreducible by 
\cite[Main Theorem (iv)]{BK}, and $\psi (\bmod\ 3)$ irreducible by \cite{GAP}, we obtain another example. 

\smallskip 
{\sf Case 5}: {\it $G = M_{24}$ in permutation representations of degree $n = 24$.}

According to \cite{GAP}, $\bmax_p(G) \leq 10395$. 

\smallskip
{\sf Case 5.1}: {\it $V = D^\la\da_H$ or $\dim V < 10395/2$. }
In this case we have that $\dim D^\la \leq 10395$. This implies 
by \cite[Lemma 1.23]{BK} that either $\la \in \L^{(4)}(n) \smallsetminus \L^{(1)}(n)$, or $p=2$ and $\la = (13,11)$. 
In the latter case, $D^\la$ is a basic spin module of dimension $2048$. Since no $\varphi \in \IBr_2(G)$ has 
degree $2048$ or $1024$, this case is ruled out. 

Let $p=2$. Then $\dim D^\mu$ for $\mu \in \L^{(4)}(n) \smallsetminus \L^{(1)}(n)$ is determined by 
Lemma \ref{LS242}, and neither $\dim D^\mu$ nor $(\dim D^\mu)/2$ matches $\varphi(1)$ for any 
$\varphi \in \IBr_2(G)$ \cite{GAP}.  In fact, $D^{(23,1)}$ is also reducible over $G$. Thus we have no example for $p=2$.

Let $p=3$. Then $\dim D^\mu$ for $\mu \in \L^{(4)}(n) \smallsetminus \L^{(1)}(n)$ is determined by 
Lemma \ref{LS243}, and using \cite{GAP} we see that $\dim D^\mu$ or $(\dim D^\mu)/2$ can match $\varphi(1)$ for some
$\varphi \in \IBr_3(G)$ only when $\mu = (22,2)$, $(22,1^2)$, or $(21,2,1)$. If $\la = (22,2)$, then $D^\la$ is obtained 
by reducing $S^\la_\C$ modulo $p$. Furthermore, $\psi := S^\la_\C\dar_G$ is irreducible by 
\cite[Main Theorem (iii)]{BK}, and $\psi (\bmod\ 3)$ irreducible by \cite{GAP}, giving rise to an example. Next,
$\al := S^{(22,1^2)}_\C\dar_G$ is irreducible by \cite[Main Theorem (iii)]{BK}, and $\al (\bmod\ 3) = \beta_{22} + \beta_{231}$,
where $\beta_i$ is an irreducible $3$-Brauer character of $G$ of degree $i \in \{22,231\}$. On the other hand,
$[S^{(22,1^2)}] = [D^{(22,1^2)}] + [D^{(23,1)}]$ by Lemma \ref{LS243}(ii), with $D^{(23,1)}\dar_G = \beta_{23}$. It follows
that $D^{(22,1^2)}\dar_G = \beta_{231}$, leading to another example. Finally, by Lemma~\ref{L4.3}, 
$D^{(21,2,1)}\dar_{\AAA_{24}}$ is irreducible (of dimension $1540$ by Lemma \ref{LS242}), ruling out the case $\la = (21,2,1)$.

\smallskip
{\sf Case 5.2}: {\it $V$ does not extend to $\SSS_{24}$ and $\dim V > 10395/2$. }

Since 
$\dim V \leq \bmax_p(G) \leq 10395$ and $\bmax_2(G) = 1792$, we must have that $p=3$, and $10395 < \dim D^\la \leq 2 \cdot 10395$.
Consider the Young subgroup $\SSS_{22,2}\cong\SSS_{22} \times \SSS_2<\SSS_{24}$. Note that the second factor $\SSS_{2}$ is generated by a transposition $t$,
which acts semisimply on $D^\la$ and has both $1$ and $-1$ as eigenvalues. The two corresponding $t$-eigenspaces are invariant under $\SSS_{22}$.
Thus the restriction of $D^\la$ to a natural subgroup $\SSS_{22}$ contains a simple submodule $D^\mu$ of dimension at most
$(\dim D^\la)/2 \leq 10395$. By \cite[Lemma 1.23]{BK} applied to $D^\mu$, we have $\mu \in \L^{(4)}(22)$. By the Frobenius reciprocity, 
$D^\la$ is a quotient of $\ind^{\SSS_{24}}_{\SSS_{22}}(D^\mu)$, and so $\la \in \L^{(6)}(24)$, i.e. $\la_1 \geq 18$. But this implies 
by Lemma~\ref{L4.3} that $D^\la$ is irreducible over $\AAA_{24}$, contrary to our assumption. 

\smallskip 
{\sf Case 6}: {\it $G = Co_3$ in permutation representations of degree $n = 276$.}

According to \cite{GAP}, any $\varphi \in \IBr_2(G)$ has degree $\leq 131584 < (n^3-9n^2+14n)/12$. 
Hence $\la \in \L^{(2)}(n) \smallsetminus \L^{(1)}(n)$ by Proposition \ref{PBound12}(i). It 
follows by Lemma \ref{LL2} that $\dim D^\la = 37400$ if $p = 2$, $37401$ or $37674$ if $p = 3$. Since no $\varphi \in \IBr_p(G)$ has such degree (or half of it, in case $D^\la\dar_{\AAA_n}$ is reducible) by \cite{GAP}, there are no examples.   
\end{proof} 

\section{Affine permutation groups}
\label{SAff}
In this section we consider restrictions to subgroups $G$ of $\SSS_n$ having regular normal elementary abelian $r$-subgroups, whose structure is explained in Lemma~\ref{LAbSoc}. Note that any irreducible $r$-modular representation of a finite group $G$ with nontrivial 
normal $r$-subgroup $R$ must be trivial on $R$. Therefore, if an $\F_r\SSS_n$-module $V$ is irreducible over such $G$ and 
$n \geq 5$, then $\dim V=1$. Henceforth we may restrict ourselves to $\SSS_n$-modules in characteristic $p \neq r$.

%\subsection{Representation theory of $AGL_m(r)$}
\subsection{Invariants in modules over wreath products}
Throughout this subsection, we assume that $r$ is a prime different from $p$. For $m\in\Z_{\geq 1}$, we denote 
$$H_m:=GL_m(r),\quad V_m:=\F_r^m, \quad G_m:=AGL_m(r)=V_m\rtimes H_m.$$ 
 % is arbitrary (not just $2$ or $3$). 
We also denote by $X_m$ the set of linear characters $V_m \to \F^\times$, and $X_m^\times\subset X_m$ be the subset of all non-trivial linear characters.  
Note that $X_m$ is an abelian group via $(\xi+\eta)(v):=\xi(v)\eta(v)$ for $\xi,\eta\in X_m$ and $v\in V_m$. In fact, $X_m$ can be identified with $\F_r^m$. 
In particular, for any $\xi\in X_m$, we have 
\begin{equation}\label{EqKills}
r\xi=0.
\end{equation}

\begin{Lemma} \label{LPChars} %{\rm \cite{}}%{\bf ()}
Let $r>2$ and assume that $m>1$ if $r=3$. There exist $\xi_1,\dots,\xi_r\in X_m^\times$ not all equal to each other and such that $\xi_1+\dots+\xi_r=0$. 
\end{Lemma}
\begin{proof}
Checked easily identifying $X_m$ with $\F_r^m$. 
\end{proof}

A key role in the study of the restriction of irreducible modules $D^\la$ from $\SSS_{r^m}$ to $G_m$ embedded via its natural action on the points of $V_m$ is played by the analysis of the invariant space $(D^\la)^{V_m}$. For this, it is convenient to embed $V_m$ into some wreath product subgroup of $\SSS_{r^m}$.

We will now assume that $m\geq 2$ and denote $n:=r^m$. We have $V_m\leq G_m\leq \SSS_n$ via the natural action of $G_m$ on $n=r^m$ points of $V_m$. Consider the corresponding embedding $\phi_m:V_m\into \SSS_n$---this comes from the regular action of $V_m$ on itself. We consider subgroups of $V_m$ as subgroups of $\SSS_n$ via the embedding $\phi_m$.

Let $e_1,\dots,e_m$ be the standard basis of $V_m=\F_r^m$, and $a\in \F_r$. We denote 
\begin{align*}
V_m(a)&:=\{b_1e_1+\dots+b_{m-1}e_{m-1}+ae_m\mid b_1,\dots,b_{m-1}\in\F_r\}\subseteq V_m,
\\
A&:=\{be_m\mid b\in\F_r\}\leq V_m,
\\
n'&:=r^{m-1}=n/r.
\end{align*}
We identify $V_{m-1}$ with $V_m(0)$ and $A$ with $(\F_r,+)$. 
Note that 
$
V_m=V_{m-1}\times A. 
$

For each $a\in A=\F_r$, let $\Si^a\cong \SSS_{n'}$ be the symmetric group on $V_m(a)$. We have a natural embedding 
$$
P:={\mathop{\scalebox{1.5}{\raisebox{-0.2ex}{$\times$}}}}_{a\in A}\,\Si^a\cong \SSS_{n'}^{\times r}\,\,\into\,\, \SSS_n 
$$
as a parabolic subgroup. 
Note that $V_{m-1}$ acts on each $V_m(a)$ regularly, so $V_{m-1}$ is embedded into $P$ diagonally via ${\mathop{\scalebox{1.5}{\raisebox{-0.2ex}{$\times$}}}}_{a\in A}\,\phi_{m-1}$. The group $A$ acts on the components $\Si^a$ of $P$ via conjugation:
$$
b\Si^ab^{-1}=\Si^{a+b}\qquad(a,b\in A). 
$$
Let 
\begin{equation}\label{EW}
W:=\langle P,A\rangle = P\rtimes A\cong \SSS_{n'}\wr A\leq \SSS_n.
\end{equation}
As $V_{m-1}$ is a subgroup of $P$, we have that 
$V_m=V_{m-1}\times A$ is a subgroup of $W=P\rtimes A$. 

We now describe irreducible $\F W$-modules. For this we consider the {\em $p$-regular $A$-multipartitions}, i.e. tuples $\bmu=(\mu^a)_{a\in A}$ such that $\mu^a\in\Par_p(n')$ for all $a\in A$. To any such $\bmu$, we associate the $\F P$-module \,
$
\boxtimes_{a\in A}D^{\mu^a}.
$ 
Elements of $\boxtimes_{a\in A}D^{\mu^a}$ are linear combinations of pure tensors of the form $\otimes_{a\in A} d_a$ with $d_ a\in D^{\mu_a}$ for all $a\in A$. Define 
$$
M(\bmu):=\ind_{P}^W (\boxtimes_{a\in A} D^{\mu^a}).
$$
Elements of $M(\bmu)$ are linear combinations of elements of the form $b\otimes (\otimes_{a\in A}d_a)$, where $b\in A$ and $d_a\in D^{\mu_a}$ for all $a\in A$. Recalling that $V_m=V_{m-1}\times A$  and considering $\F A$ as a left regular $\F A$-module, we have
\begin{equation}\label{E290317}
M(\bmu)\da_{V_{m}}=M(\bmu)\da_{V_{m-1}\times A}
\cong (\otimes_{a\in A}D^{\mu^a}\da_{V_{m-1}})\boxtimes \F A. 
\end{equation}

We say that a $p$-regular $A$-multipartition $\bmu=(\mu^a)_{a\in A}$ is {\em constant} (with value $\mu$) if there exists $\mu\in\Par_p(n')$ such that 
$\mu^a=\mu$ for all $a\in A$. Let $\bmu=(\mu)_{a\in A}$ be a  constant $A$-multipartition with value $\mu$. For any linear character $\al:A\to \F$, we define a $\F W$-module $M_\al(\mu)$ by extending the $P$-action on $\boxtimes_{a\in A} D^\mu$ to $W=P\rtimes A$-module via 
$$
b(\otimes_{a\in A} d_a)=\al(b)(\otimes_{a\in A} d_{a+b})\qquad(b\in A).
$$

The following result follows from Clifford theory:

\begin{Lemma} %\label{}%{\rm \cite{}}%{\bf ()}
If $M$ is an irreducible $\F W$-module then it is isomorphic to a module of one of the following two types:
\begin{enumerate}
\item[{\rm (1)}] $M_{\al}(\mu)$ for some $\mu\in\Par_p(n')$ and some linear character $\al:A\to \F$;
\item[{\rm (2)}] $M(\bmu)$ for some non-constant $p$-regular $A$-multipartition $\bmu$. 
\end{enumerate}
Conversely, all modules of the forms {\rm (1)} and {\rm (2)} are irreducible, and the only isomorphisms between them are  $M(\bmu)\cong M(\bnu)$ if there exists $b\in A$ such that $\nu^a=\mu^{a+b}$ for all $a\in A$. 
\end{Lemma}

In the next two lemmas, we study $V_m$-invariants in irreducible $\F W$-modules. 

\begin{Lemma} \label{LI} %{\rm \cite{}}%{\bf ()}
Let $M$ be an irreducible $\F W$-module of the form $M=M_{\al}(\mu)$ for some $\mu\in \Par_p(n')$. Then $M^{V_m}=0$ if and only if $\al\neq 0$ and one of the following conditions holds:
\begin{enumerate}
\item[{\rm (i)}] $D^\mu\in \T_{n'}$.
\item[{\rm (ii)}] $D^\mu\in \N_{n'}$ and either $r=2$, or $r=3$ and $m=2$.
\end{enumerate} 
\end{Lemma}
\begin{proof}
The `if'-part is an explicit check. 
For the `only-if'-part, if $\al=0$, pick a non-zero $d\in D^{\mu}_\xi$ for some 
$\xi\in X_{m-1}$. Then, using (\ref{EqKills}), it is easy to see that  $\otimes_{a\in A}d\in M^{V_m}\setminus\{0\}$. Now let $\al\neq 0$. Suppose we are given the following data: 
\begin{enumerate}
\item[(a)] characters $\{\xi_a\in X_{m-1}\mid a\in A\}$ with  $\sum_{a\in A}\xi_a=0$;
\item[(b)] non-zero vectors $\{d_a\in D^\mu_{\xi_a}\mid a\in A\}$, not all proportional to each other.
\end{enumerate}
Then it is easy to see that 
$$
\sum_{b\in A}\al(b)(\otimes_{a\in A} d_{a+b})\in M^{V_m}\setminus\{0\}.
$$
%is a nonzero $V_m$-invariant element of $M$. 

In view of (i), we may assume that $D^\mu \not \in \T_{n'}$. So $\dim D^\mu \geq 2$. 
If $D^\mu_\xi= 0$ for all $\xi\in X_{m-1}^\times$, then $D^\mu_0=D^\mu$, and we can take $\xi_a=0$ for all $a$ to satisfy (a) and pick vectors $d_a\in D^\mu_0$, not all proportional to each other, to satisfy (b). Thus we may assume that $D^\mu_\xi\neq 0$ for some $\xi\in X_{m-1}^\times$, in which case we have $D^\mu_\xi\neq 0$ for all $\xi\in X_{m-1}^\times$. Now, we use Lemma~\ref{LPChars} to find characters $\xi_a\in X_{m-1}^\times$ satisfying (a), and by taking any non-zero $d_a\in D^\mu_{\xi_a}$ we also satisfy (b). Lemma~\ref{LPChars} is applicable unless $r=2$, or $r=3$ and $m=2$, so these cases need to be considered separately. In the case $r=3$ and $m=2$ there is actually nothing to check in view of the exception (ii) since for $\SSS_3$ all irreducible modules are in $\NT_3$. 

Let $r=2$. If there is $\xi\in X_{m-1}$ with $\dim D^\mu_\xi\geq 2$, we set $\xi_a:=\xi$ for all $a$ to satisfy (a), cf. (\ref{EqKills}). 
Then pick linearly independent vectors $x_1,x_2\in D^\mu_\xi$ and set $d_a:=x_1$ for $a\in A_1$ and $d_a=x_2$ for $a\in A_2$, where $A=A_1\sqcup A_2$ for some non-empty sets $A_1,A_2$. The vectors $d_a$ satisfy (b). Thus we may assume that $\dim D^\mu_{\xi}\leq 1$ for all $\chi\in X_{m-1}$, i.e. $\dim D^\mu\leq n'$. Using \cite[Theorem 6(ii)]{JamesDim}, we deduce that $D^\mu\in\N_{n'}$, which is exception (ii). 
\end{proof}

\begin{Lemma} \label{LII} %{\rm \cite{}}%{\bf ()}
Let $M$ be an irreducible $\F W$-module of the form $M=M(\bmu)$ for some non-constant $p$-regular $A$-multipartition $\bmu$. Then $M^{V_m}=0$ if and only one of the following two conditions holds: 
\begin{enumerate}
\item[{\rm (i)}] there exists $b\in A$ such that $(D^{\mu^b})^{V_{m-1}}=0$ and $D^{\mu^a}\in \T_{n'}$ for all $a\neq b$. 
\item[{\rm (ii)}] $r=2$, $m=3$, $p> 3$, and there exists $b\in A$ such that $D^{\mu^b}\in \N_4$
%$(D^{\mu^b})^{V_{2}}=0$ 
and $\mu^a=(2,2)$ for $a\neq b$. 
\end{enumerate} 
\end{Lemma}
\begin{proof}
We denote by $N$ the $\F V_{m-1}$-module $\otimes_{a\in A} D^{\mu^a}\da_{V_{m-1}}$. 
By (\ref{E290317}), we have 
$$M^{V_m}\cong N^{V_{m-1}}\boxtimes (\F A)^A\cong N^{V_{m-1}}.$$
This gives the `if'-part, for if there exists $b\in A$ such that $(D^{\mu^b})^{V_{m-1}}=0$ and $D^{\mu^a}\in \T_{n'}$ for all $a\neq b$, then $N\cong D^{\mu^b}\da_{V_{m-1}}$, and so $N^{V_{m-1}}=0$. 

For the `only-if-part', assume first that $r\neq 2$. If there is at most one $b\in A$ with $D^{\mu_b}\not\in \T_n$ the result easily follows. Suppose there are $b\neq c$ in $A$ such that $D^{\mu^b},D^{\mu^c}\not\in \T_{n'}$. Then $D^{\mu^b}_\xi$ and $D^{\mu^c}_\xi$ are non-zero for all $\xi\in X_{m-1}^\times$. For each $a\in A\setminus\{ b,c\}$, take $d_a\in D^{\mu^a}_{\xi_a}$ for some $\xi_a\in X_{m-1}$. Now, there exist $\xi_b,\xi_c\in X_{m-1}^\times$ such that $\sum_{a\in A}\xi_a=0$. Pick non-zero $d_b\in D^{\mu^b}_{\xi_b}$ and $d_c\in D^{\mu^c}_{\xi_c}$. Then $\otimes_{a\in A} d_a$ is a non-zero $V_{m-1}$-invariant vector of $N$. 

Now, let $r=2$. Note that $N^{V_{m-1}}\neq 0$ if and only if there is a character $\xi\in X_{m-1}$ such that $D^{\mu^0}_\xi$ and $D^{\mu^1}_\xi$ are non-zero. This is not the case exactly when $(D^{\mu^0})^{V_{m-1}}=D^{\mu^0}$ and $(D^{\mu^1})^{V_{m-1}}=0$, or $(D^{\mu^0})^{V_{m-1}}=0$ and $(D^{\mu^1})^{V_{m-1}}=D^{\mu^1}$. But the equality $(D^{\mu})^{V_{m-1}}=D^\mu$ holds if and only if $D^\mu\in \T_{n'}$, or $m=3$, $p> 3$, and $D^{\mu}=D^{(2,2)}$, cf. for example \cite[Lemma 5.5]{BK}. 
\end{proof}

\subsection{Invariants in modules over symmetric groups}
Recall that we are 
%assuming that $D^\la\not\in \NT_{n}$ and 
considering the embeddings $V_m\leq G_m\leq \SSS_n$ for $n=r^m$ and assuming that $p \neq r$. 

\begin{Lemma} \label{LNat} %{\rm \cite{}}%{\bf ()}
If $D\in \N_n$, then $D^{V_m}=0$. 
\end{Lemma}
\begin{proof}
Since the action of $\SSS_n$ on $D$ is faithful, $D$ affords a non-trivial character of $V_m$, hence $D$ affords all $n-1$ non-trivial characters of $V_m$, hence the trivial character does not appear by dimensions. 
\end{proof}

\begin{Lemma} \label{LGAP} %{\rm \cite{}}%{\bf ()}
Let $p=3$ and $r=2$. 
\begin{enumerate}
\item[{\rm (i)}] Let $m=3$, i.e. $n=8$. Then $(D^\la)^{V_3}=0$ if and only if $D^\la\in \N_8\cup \llbracket D^{(6,1,1)}, D^{(5,3)}\rrbracket$. 

\item[{\rm (ii)}] Let $m=4$, i.e. $n=16$. Then $(D^\la)^{V_4}=0$ if and only if $D^\la\in\N_{16}\cup \llbracket D^{(14,1,1)}\rrbracket$. 
\end{enumerate}
\end{Lemma}
\begin{proof}
(i) In \cite{GAP}, 
%\color{red}
%(how to reference this correctly?) 
\color{black}
any nontrivial element $t$  
in $V_3$ is of class 2A. If $\varphi$ is the Brauer character 
of an irreducible 3-modular module $D$ of $\SSS_8$, then $D^{V_m}=0$ if and only if 
$\varphi(t) = -\varphi(1)/7$. Now inspecting the $3$-Brauer character table in \cite{GAP} of $\SSS_8$, 
we see that there are exactly six possibilities for such $\varphi$, two for each dimension 
$7$, $21$, and $28$. These correspond, respectively, to modules in $\N_8$, $\llbracket D^{(6,1,1)}\rrbracket$, and $\llbracket D^{(5,3)}\rrbracket$. 

(ii) We apply the same argument as in the case $m=3$. Now $t \in V_4 \smallsetminus \{1\}$ 
is of class 2C, and the condition is $\varphi(t) = -\varphi(1)/15$. It follows by checking 
the $3$-Brauer character table of $\SSS_{16}$ in \cite{GAP} that there are exactly four possibilities for such $\varphi$, two of dimension $15$  and two of dimension $105$. These correspond, respectively, to modules in $\N_{16}$ and $\llbracket D^{(14,1,1)}\rrbracket$. \end{proof}

\begin{Lemma} \label{LGAP1} %{\rm \cite{}}%{\bf ()}
Let $p=2$, $r=3$, and $m=2$, i.e. $n=9$. Then $(D^\la)^{V_2}=0$ if and only if $D^\la\cong \N_9\cup D^{(5,4)}$. 
\end{Lemma}
\begin{proof}
(i) In \cite{GAP}, the elements in $V_2\smallsetminus \{1\}$ are of class 3B. So, arguing as in the proof of Lemma~\ref{LGAP}, we get exactly two $2$-modular 
modules $W$ of $\SSS_9$ with $W^{V_2}=0$, of dimensions $8$ and $16$. These correspond, respectively, to modules in $N_9$ and $D^{(5,4)}$.
\end{proof}

Now we can prove our key technical result which develops \cite[Proposition~4.6]{KS}:

\begin{Proposition} \label{P110417_5} %{\rm \cite{}}%{\bf ()}
Let $p=2$ or $3$ and $m\geq 2$. Then $(D^\la)^{V_m}= 0$ if and only if one of the following happens: 
\begin{enumerate}
\item[{\rm (i)}] $D^\la\in\N_n$;
\item[{\rm (ii)}] $r=2$, $p=3$, $D^\la\in\llbracket D^{(n-2,1,1)}, D^{(5,3)}\rrbracket$;
\item[{\rm (iii)}] $r=3$, $p=2$, $D^\la\cong D^{(5,4)}$.
\end{enumerate}
\end{Proposition}
\begin{proof}
It follows from \cite[Lemma 5.6]{BK} that in the case $r=2$ and $p=0$, we have $(S^{(n-2,1,1)}_\C)^{V_m}=0$. Reducing modulo $3$, we deduce $(D^{(n-2,1,1)})^{V_m}= 0$. The rest of the ``if" part follows from Lemmas~\ref{LNat}, \ref{LGAP}, \ref{LGAP1}. 

For the ``only-if" part, recall that $V_m\leq W=P\rtimes A\leq \SSS_n$, cf. (\ref{EW}). By Lemmas~\ref{LI} and \ref{LII}, we have $(D^\la)^{V_m}= 0$ only if all composition factors  of $D^\la_P$ are of the form $\boxtimes_{a\in A} D^{\mu^a}$ satisfying one of the following conditions:
\begin{enumerate}
\item[${\tt (C_1)}$] there is $b\in A$ such that $D^{\mu^a}\in\T_{n'}$ for all $a\neq b$, and either $(D^{\mu^b})^{V_{m-1}}=0$ or $D^{\mu^b}\in\T_{n'}$;
\item[${\tt (C_2)}$] $\mu^a=\mu^b$ for all $a,b\in A$, $D^{\mu^a}\in\N_{n'}$ for all $a\in A$, and either $r=2$, or $r=3$ and $m=2$;
\end{enumerate}

Assume first that $p=2$. Then $r\neq 2$. If $(r,m)\neq (3,2)$, the restriction $D^\la_P$ only has composition factors $\boxtimes_{a\in A} D^{\mu^a}$ satisfying ${\tt (C_1)}$. By Proposition~\ref{C080417}, $D^\la\in\NT_n$. 
In the exceptional case $(r,m)= (3,2)$ use Lemma~\ref{LGAP1}.

Now, let $p=3$. Then $r\neq 3$. If $r\neq 2$, the restriction $D^\la_P$ only has composition factors $\boxtimes_{a\in A} D^{\mu^a}$ satisfying ${\tt (C_1)}$. By Proposition~\ref{C080417}, $D^\la\in\NT_n$. Let $r=2$. By Lemma~\ref{LGAP}, we may assume that $m\geq 5$. Then, by induction on $m$, we may assume that all composition factors $D^\mu\boxtimes D^\nu$ of $D^\la_{\SSS_{n/2}\times \SSS_{n/2}}$ satisfy one of the following three conditions: 
\begin{enumerate}
\item[{\rm (1)}] $D^\mu\cong D^\nu\in \N_{n/2}$, 
\item[{\rm (2)}] $D^\mu\in\T_{n/2}, D^\nu\in \NT_{n/2}\cup \llbracket D^{(n/2-2,1,1)}\rrbracket$,
\item[{\rm (3)}] $D^\nu\in\NT_{n/2}, D^\mu\in \N_{n/2}\cup \llbracket D^{(n/2-2,1,1)}\rrbracket$.
\end{enumerate}
%(1) $D^\mu\cong D^\nu\in \N_{n/2}$, (2)  $D^\mu\in\T_{n/2}, D^\nu\in \NT_{n/2}\cup \llbracket D^{(n/2-2,1,1)}\rrbracket$, (3) $D^\nu\in\NT_{n/2}, D^\mu\in \N_{n/2}\cup \llbracket D^{(n/2-2,1,1)}\rrbracket$. 
By Corollary~\ref{P110417}, $D^\la\in\NT_n\cup\llbracket D^{(n-2,1,1)}\rrbracket$.
\end{proof}

\begin{Remark}
One can ask what could be an analogue of Proposition~\ref{P110417_5} in the case $m=1$, that is, which irreducible 
modules $D^\la$ of $\SSS_r$ in characteristic $p \neq r$ have no invariants on the cyclic subgroup $\CCC_r < \SSS_r$. Until now, this question
has been resolved only in the case $p =0$ (see \cite{Z}), and in the case $r/2 < p < r$ (see \cite[Lemma 3.2]{Th}). 
\end{Remark}

\subsection{Irreducible restrictions to affine permutation groups}
Let $G \leq \SSS_n$ be a primitive subgroup with a regular normal abelian subgroup;
or more generally, let $G \leq \SSS_n$ be any subgroup with a regular normal elementary abelian subgroup. Then, up to conjugacy, 
$G$ is a subgroup of the group $AGL_m(r)$ of all affine transformations of the affine  space $V_m=\F_r^m$ for a prime $r$. The group $AGL_m(r)$ acts naturally on $n=r^m$ points of $V_m$, which yields an embedding $G\leq \SSS_n$. 
%If $q=r^f$ for a prime $r$, then 
Moreover, 
$AGL_m(r)\cong V_m \rtimes GL_m(r)$, and $G=V_m\rtimes H$ for $H\leq GL_m(r)$. 
% which acts naturally on $n=q^m$ points of $V$, yielding an embedding $G\leq \SSS_n$. 
%The group $\AGaL(m,q)$ acts naturally on $n=q^m$ points of $V$, which yields an embedding $G\leq \SSS_n$. 
%Let $H:=GL(m,q) \rtimes \CCC_f$. We have $G=V_m\rtimes H$. The group $G$ acts naturally on $n=q^m$ points of $V$, which yields an embedding $G\leq \SSS_n$. 

%The case $D^\la\in \T_n$ is of course trivial, while the case $D^\la\in\N_n$ has been handled in \cite{Mortimer}. This explains the exclusion of these cases in the theorem. %\item[$\bullet$] $p=2$ or $3$ (the case $p>3$ has been handled in \cite{BK} and \cite{KS}).

The following theorem is the main result of the section. It develops \cite[Corollary 4.7]{KS}. 
%The case $D^\la\in \T_n$ is of course trivial, while the case $D^\la\in\N_n$ can be handled using \cite{Mortimer} since $p\nmid n$ and so $D^{(n-1,1)}\da_G$ is irreducible only if $G$ is $2$-transitive. This explains the exclusion of these cases in the theorem. 

\begin{Theorem} \label{TAGL} %{\rm \cite{}}%{\bf ()}
Let $p=2$ or $3$, $n \geq 5$, $H=\SSS_n$ or $\AAA_n$, and let $M$ be an irreducible $\F H$-module of dimension greater than $1$. Let 
$G < H$ be a subgroup that contains a regular normal, elementary abelian subgroup. Then $M\da_G$ is irreducible if and only if one of the following happens:
\begin{enumerate}
\item[{\rm (i)}] $M\da_{\AAA_n}\cong E^{(n-1,1)}$ and $G$ is $2$-transitive;
\item[\rm (ii)] $M\da_{\AAA_n}\cong E^{(n-2,1^2)}$, and of the following holds:
\begin{enumerate}
\item[{\rm (a)}] $p=3$, $G = AGL_m(2)$ with $n=2^m$;
\item[{\rm (b)}] $p=3$, $G = \CCC_2^4 \rtimes \AAA_7$ with $n=16$;
\end{enumerate}
%$m\geq 3$, 
\item[\rm (iii)]  $p=2$, $H = \AAA_9$, $G = ASL_2(3)$
or $\CCC_3^2 \rtimes {\mathsf Q}_8$, and $M\cong E^{(5,4)}_\pm$. 
%has dimension  $\dim M = 8$.
\item[\rm (iv)] $p=2$, $H=\AAA_5$, $G = \CCC_5\rtimes\CCC_2$, and $M\cong E^{(3,2)}_\pm$. 
\end{enumerate}
\end{Theorem}

\begin{proof}
Let $\la\in\Par_p(n)$ be such that $M=D^\la$ if $H=\SSS_n$, or $M=E^\la_{(\pm)}$ if $H=\AAA_n$. 

Recall that $n=r^m$ and 
$%\begin{equation}\label{EEmbed}
G=V_m\rtimes G_0 \leq AGL_m(r) 
$. %\end{equation}
%So $G\leq AGL(mf,r)$, and we may assume that $f=1$. 
Assume $M\da_G$ is irreducible. By Clifford theory,  $M^{V_m} = 0$,  
%then $V_m$ acts trivially on $M$, which is impossible since the module is faithful. So from now on we assume that $M^{V_m}=0$, 
and so $p \neq r$. 
In particular, $p\nmid n$, and so $D^{(n-1,1)}$ is 
%the natural $n-1$-dimensional representation which is 
a reduction modulo $p$ of the natural $(n-1)$-dimensional representation in characteristic $0$, hence $D^{(n-1,1)}\da_G$ is irreducible only if $G$ is $2$-transitive, in which case it is indeed always irreducible by \cite{Mortimer}. This gives case (i), and from now on we assume that $E\da_{\AAA_n}\not\cong E^{(n-1,1)}$. 

We now exclude the case $m=1$. In this case $|G|\leq |AGL(1,r)|= r(r-1)$. In particular if $M\da_G$ is irreducible then $\dim M<r=n$. If $H=\SSS_n$, or $H=\AAA_n$ and $M$ lifts to $\SSS_n$, then by \cite[Theorem 6(i)]{JamesDim} for $r \geq 7$ we have $D^\la \in \NT_n$, which was excluded in the previous paragraph. The special case $r=5$ is checked using  \cite{ModularAtlas}. On the other hand, if $H=\AAA_n$ and $M\cong E^\la_\pm$ note first that $G\leq AGL(1,r)\cap \AAA_r\cong \CCC_r\rtimes \CCC_{(r-1)/2}$,
%(since $GL(1,r)\not\subseteq A_r$ as $r\geq 5$). 
and the dimension of an irreducible $\F(\CCC_r\rtimes \CCC_{(r-1)/2})$-module is 
%Since $V_1\cong \F_r$ is a normal abelian subgroup of $G$ it follows that any irreducible representation of $\F G$ has dimension 
at most $(r-1)/2$. On the other hand, from \cite[Proposition 4.1]{KST}, we have that $\dim E^\la_\pm\geq 2^{(r-6)/2}$. If $r\geq 11$ we have that $2^{(r-6)/2}>(r-1)/2$. So we only need to consider the cases $r=5$ and $7$. In these cases using modular character tables it can be checked that if $\dim E^\la_\pm\leq (r-1)/2$ then $r=5$, $p=2$ and $\la=(3,2)$, in which case the restriction is indeed irreducible. This corresponds to the special case (iv).

%\smallskip {\sf Step 1.}  

As $M\da_{G}$ is irreducible, so is $M\da_{AGL_m(r)\cap H}$. It easily follows that $M\da_{V_{m}}$ is a direct summand of $N\da_{V_{m}}$ for some irreducible  $\F AGL_m(r)$-module $N$. 
%Since $mf\geq 2$, we always have $V_{mf} \leq \AAA_n$. 
% and that one can identify the regular normal subgroup $U$ of $G$ with $V_{mf}$. 
%Hence if $M^{V_{m}}\neq 0$, then $M^{V_{m}}=M$ by irreducibility. But this contradicts the faithfulness of $M$, so 
As mentioned above, we have $M^{V_{m}}=0$. If $H=\SSS_n$, or $H=\AAA_n$ and $M$ lifts to $\SSS_n$, then $(D^\la)^{V_{m}}=M^{V_{m}}= 0$. On the other hand, if $H=\AAA_n$ and $M=E^\la_\pm$, note first that any non-trivial element $g\in V_{m}$ has cycle type $(r^{n/r})$ and $n/r>1$ since we have already excluded the case $m=1$. So  $g^\si$ is in the same $\AAA_n$-conjugacy class as $g$ for any $\si\in\SSS_n$. It follows that the Brauer characters of $E^\la_+\da_{V_{m}}$ and of $E^\la_-\da_{V_{m}}$ coincide and hence  $E^\la_+\da_{V_{m}}\cong E^\la_-\da_{V_{m}}$ as $p\nmid r$. So in this case we can still conclude that $(D^\la)^{V_{m}}= 0$.
 By Proposition~\ref{P110417_5} we may now assume that one of the following happens: 
\begin{enumerate}
\item[{\rm (1)}] $r=2$, $p=3$, $D^\la\in\llbracket D^{(n-2,1,1)}, D^{(5,3)}\rrbracket$;
\item[{\rm (2)}] $r=3$, $p=2$, $D^\la\cong D^{(5,4)}$. 
\end{enumerate}

{\sf Case 1.1:} $r=2$, $p=3$, and $D^\la\in\llbracket D^{(n-2,1,1)}\rrbracket$. 
By \cite[Theorem 24.1]{JamesBook} that if $2< p \nmid n$ then 
$D^{(n-2,1,1)}$ is reduction modulo $p$ of the Specht module $S^{(n-2,1,1)}_\C$ in characteristic $0$. Here $p = 3 \nmid n = 2^{m}$, so 
(by tensoring with $\sgn$ if necessary) we may assume 
that $D^\la$ is reduction modulo $3$ of $S^{(n-2,1,1)}_\C$. 
Moreover, %$m \geq 2$ implies $4|n$, whence $\dim S^{(n-2,1^2)}$ is odd and 
it is easy to see that $D^\la\da_{\AAA_n}$ is irreducible. 
Hence $S^{(n-2,1,1)}_\C\da_{G}$ is irreducible. 
%If furthermore
%$q > 2$, then, as shown on \cite[p. 174]{BK}, $G$ is not $3$-homogeneous, and so 
%$S^{(n-2,1,1)}_G$ must be reducible by \cite[Proposition 5.1(i)]{BK}. Thus 
%$q=2$ in this case, and 
Note that $m\geq 3$ as $n = 2^{m} \geq 5$. By \cite{S},
the only proper subgroup of $AGL_m(2)$ that contains $V_m$ and is irreducible on $S^{(n-2,1^2)}$ is $K := V_4 \rtimes \AAA_7 < \AAA_{16}$.
Furthermore, the only complex irreducible character of degree $7$ of $GL_3(2)$ remains irreducible modulo $3$, cf. \cite{ModularAtlas}, so the arguments on pp. 179--180 of \cite{BK} show that $D^\la\da_K$ is indeed irreducible.
We have shown that either $G = AGL_m(2)$ with $m \geq 3$, or
$G  = K$ and $m=4$, as stated in (ii).

\smallskip
{\sf Case 1.2.} $r=2$, $p=3$, $D^\la\in\llbracket  D^{(5,3)}\rrbracket$. By \cite[Tables]{JamesBook}, we have $\dim D^{(5,3)}=28$. 
Furthermore, $D^{(5,3)}$ is irreducible over $\AAA_8$, so it suffices to show that 
$D^\la\da_{AGL_3(2)}$ is reducible. Since $D^\la$ affords all $7$ non-trivial linear characters of $V_3$,
it follows that $D^\la\da_G = \ind^G_{G_1}(U)$, where $U$ is a $4$-dimensional module of $G_1 = V_3 \rtimes \SSS_4$. Now $V_3$ acts via scalars on 
$U$, and the degree of any irreducible $\F\SSS_4$-representation is at most $3$, whence $U$, and so $D^\la\da_G$, is reducible.  
%Moreover, $|V_3^\times|=7$ does not divide $8$, so by Lemma~\ref{L110417_4}(ii), the restriction $D^{(5,4)}_G$ must be reducible. 

\smallskip
{\sf Case 2:} $r=3$, $p=2$, $D^\la\cong D^{(5,4)}$. By \cite[Tables]{JamesBook}, we have $\dim D^{(5,4)}=16$. 
First we consider the case $H = \SSS_9$. Then it suffices to show that 
$D^\la\da_G$ is reducible for $G = AGL_2(3)$. This group $G$ is the $7^{\mathrm {th}}$ maximal subgroup of 
$\SSS_9$ as listed in \cite{GAP}. We can pick two elements $x \in V_2 \smallsetminus \{1\}$ and $y \in G \smallsetminus V_2$, which belong to
classes $3A$ and $3C$ in $G$ (in the notation of \cite{GAP}) and which both induce fixed-point-free permutations in $\SSS_9$. Thus both $x$ and 
$y$ belong to class $3B$ of $\SSS_9$. % (in the notation of \cite{GAP}). 
The only irreducible $2$-Brauer character of $G$ of degree $16$ takes 
value $1$ at $y$, whereas the character of $D^\la$ takes value $-2$ at $y$, cf. \cite{GAP}. 
Hence we conclude that $D^\la\da_G$ is reducible. (An alternate way is to 
note that $D^\la$ is reduction modulo $2$ of the basic spin  module of a double cover $\hat\SSS_9$, and the latter is 
reducible over the inverse image of $G$ in $\hat\SSS_9$ by \cite[Theorem 1.1]{KlW}.) 

Now let $H = \AAA_9$. Then $D^{(5,4)}\da_{\AAA_9}$ splits. Each of $E^{(5,4)}_\pm$ affords $8$ non-trivial linear characters of $V_2$, which are permuted transitively by 
$AGL_2(3) \cap \AAA_9 = ASL_2(3)$. Moreover, the only proper subgroup of $SL_2(3)$ that 
acts transitively on these $8$ characters is ${\mathsf Q}_8$. It follows that 
$G  = ASL_2(3)$ or $V_2 \rtimes {\mathsf Q_8}$, in which case the restriction is indeed irreducible, giving the case (iii).
\end{proof}

\section{Doubly transitive groups with socle $PSL_m(q)$}\label{SLG1}
Throughout the section: $q=r^f$ is a power of a prime $r$, $m \geq 2$, $W:= \F_q^m$ with standard basis $e_1,\dots,e_m$, and $\PP(W)$ is the set of $1$-dimensional subspaces of $W$. 
Also, unless otherwise stated, 
$$n := |\PP(W)|=(q^m-1)/(q-1),$$ and  $G < \SSS_n$ with $S:=\soc(G) = PSL_m(q)$ acting naturally on  $\PP(W)$. 
%, and $n = (q^m-1)/(q-1)$ is the number of $1$-dimensional subspaces of $W$. 

\subsection{Bounding the partition $\la$ for groups with socle $PSL_m(q)$}%\label{SSBound1}
With the notation as above, we have:

\begin{Lemma} \label{L270219} %{\rm \cite{}}%{\bf ()}
We have $S \unlhd G\leq PGL_m(q)\rtimes \CCC_f\leq \Aut(S)$. 
\end{Lemma}
\begin{proof}
Note that 
$N := N_{\SSS_n}(S)$ is doubly transitive with non-abelian simple normal subgroup $S$. By 
\cite[Proposition 5.2]{Cameron}, $\soc(N) = S$. Now $C_{\SSS_n}(S) \cap S = 1$, hence $\soc(N) = S$ implies that $C_{\SSS_n}(S)=1$. So  
$S \unlhd G \leq N \leq \Aut(S)$. The group $\Aut(S)$ is described in \cite[Theorem 2.5.12]{GLS}. 
If $m=2$, we have $\Aut(S) = PGL_m(q) \rtimes \CCC_f$, and we are done. If $m \geq 3$, the inverse-transpose automorphism of $S$ does not stabilize its action on $\PP(W)$, so we have $G \leq PGL_m(q) \rtimes \CCC_f$. 
\end{proof}

By Lemma~\ref{L270219}, we have $G\leq PGL_m(q)\rtimes \CCC_f$ where $PGL_m(q)\rtimes \CCC_f$ acts naturally on $\PP(W)$. For $1\leq k\leq m-1$,
let $W_k:= \langle e_1, e_2, \ldots ,e_k \rangle_{\F_q}\subseteq W$, and denote by $\tilde P_k$ the subgroup of $PGL_m(q)\rtimes \CCC_f$ 
consisting of all elements that fix every point of $\PP(W_k)$. (If $k > 1$, then $\tilde P_k$ is the image in $PGL_m(q) \rtimes \CCC_f$ 
of the subgroup of $GL_m(q) \rtimes C_f$ that acts via scalars on $W_k$.) Also, let $P_k := \tilde P_k \cap G$. By construction,
$P_k$ fixes all 
$$L_k:=(q^k-1)/(q-1)$$ 
$1$-dimensional subspaces of $\langle e_1, e_2, \ldots,e_k \rangle_{\F_q}$. Thus: 

\begin{Lemma} \label{L270219_3} %{\rm \cite{}}%{\bf ()}
The subgroup $P_k$ is contained in a natural subgroup $\SSS_{n-L_k}$ of $\SSS_n$. 
\end{Lemma}

\begin{Lemma} \label{L270219_4} %{\rm \cite{}}%{\bf ()}
Let $\la=(n-\ell,\dots)\in\Par_p(n)$. For an integer $1 \leq k \leq m-1$ such that $(q^k-1)/(q-1) \geq 2\ell$, we have $(D^\la)^{P_k}\neq 0$. In particular, if $D^\la\da_G$ is irreducible then $\dim D^\la\leq [G:P_k]$. 
\end{Lemma}
\begin{proof}
The first statement follows from Lemma~\ref{L270219_3} and Theorem \ref{TSub}. The second one then follows from the Frobenius Reciprocity. 
\end{proof}

Setting
$$P_k \unlhd R_k:= \operatorname{Stab}_G(\langle e_1, \ldots,e_k \rangle_{\F_q}),$$ 
we have that
\begin{equation}\label{E280219_3}
PGL_k(q) \cong (R_k \cap S)/(P_k \cap S) \unlhd R_k/P_k \leq PGL_k(q)\rtimes\CCC_f.
\end{equation} 
(Indeed, one can find an element of $SL_n(q)$ that fixes $W_k$ and has any prescribed determinant in its action on $W_k$.)
Since both $G$ and $S$ act transitively on the set $\PP_k(W)$ of $k$-subspaces of $W$, we have
\begin{equation}\label{E280219}
[G:R_k]=[S:S\cap R_k]=|\PP_k(W)|
=\prod_{i=1}^k\frac{q^{n-i+1}-1}{q^i-1}.
%=\frac{(q^n-1)\dots(q^{n-k+1}-1)}{(q^k-1)\dots(q-1)}.
\end{equation}

\begin{Lemma} \label{L270219_5} %{\rm \cite{}}%{\bf ()}
Let $\la=(n-\ell,\dots)\in\Par_p(n)$ and $D^\la\da_G$ be irreducible. For an integer $1 \leq k \leq m-1$ such that $(D^\la)^{P_k}\neq 0$, we have 
$$
\dim D^\la \leq [G:R_k]\bmax_p(R_k/P_k)= \frac{[G:P_k]}{[R_k:P_k]}\bmax_p(R_k/P_k)=\bmax_p(R_k/P_k)\prod_{i=1}^k\frac{q^{n-i+1}-1}{q^i-1}.
$$
The assumption $(D^\la)^{P_k}\neq 0$ is guaranteed if 
$(q^k-1)/(q-1) \geq 2\ell$.  
\end{Lemma}
\begin{proof}
%By Lemma~\ref{L270219_4}, $V^{P_k}\neq 0$. 
Note that $R_k$ acts on $(D^\la)^{P_k}$ and the $R_k$-module $(D^\la)^{P_k}$ contains a simple submodule 
$X$ of dimension at most $\bmax_p(R_k/P_k)$. By the Frobenius reciprocity, we have
$$\dim D^\la \leq [G:R_k]\dim X \leq [G:R_k]\bmax_p(R_k/P_k),$$
and it remains to use (\ref{E280219}) and Lemma~\ref{L270219_4}. 
\end{proof}

\begin{Proposition}\label{PRedSL}
Let $V$ an irreducible $\F G$-module. Then:
\begin{enumerate}[\rm(i)]
\item $\dim V \leq |G|^{1/2} \leq |\Aut(S)|^{1/2}$.
\item 
%If $q = r^f$ for a prime $r$, then 
$\dim V \leq n^{(m+1)/2} < n^{\frac{1}{2}\log_2 n +1}$.
\item Suppose that $V = D^\la\da_G$ for $\la = (n-\ell, \ldots)\in\Par_p(n)$, and that  
there exists an integer $1 \leq k \leq m-1$ such that $(q^k-1)/(q-1) \geq 2\ell$. Then $\dim V < q^{mk}$. 
\end{enumerate}
\end{Proposition}

%We will often choose $k := \lceil \log_q(2\ell-1) \rceil +1$, and 

\begin{proof}
(i) Follows from Lemma~\ref{L270219}. 

%The first inequality is well known. Note that $N := N_{\SSS_n}(S)$ is doubly transitive with non-abelian simple normal subgroup $S$. By \cite[Proposition 5.2]{Cameron}, $\soc(N) = S$. Now $C_{\SSS_n}(S) \cap S = 1$, so $\soc(N) = S$ implies that $C_{\SSS_n}(S)=1$. Hence $S \unlhd G \leq N \leq \Aut(S)$, which implies the second inequality. 

\smallskip
(ii) Note that $n  > 2^{m-1}$, so $m \leq 1+\log_2 n$, which implies the second inequality. 
%We noted in (i) that $G \leq \Aut(S)$, and the latter is described in \cite[Theorem 2.5.12]{GLS}. 
Let $H := PGL_m(q)$. By Lemma~\ref{L270219}, $G\leq H \rtimes \CCC_f$. 
If $m=2$, then $\bmax(H) = q+1$. Since $f^2 \leq 2^f+1 \leq q+1$, we deduce that
$\dim V \leq f\bmax(H) \leq (q+1)^{3/2}$, as stated.

Let $m \geq 3$. 
%As the inverse-transpose automorphism of $S$ does not stabilize its action on $1$-dimensional subspaces of $W$, we have $G \leq H \rtimes \CCC_f$. 
Using Lemma \ref{TBound3}, and the estimate $(q^i-1)(q^{m-i}-1) < q^m-1$ for $1 \leq i \leq m-1$, we get
$$\begin{aligned}
    \bmax(H) & \leq \bmax(S) \cdot [H:S] \leq \bmax(SL_m(q)) \cdot [H:S]\\
    & \leq \frac{(q-1)(q^2-1)(q^3-1) \ldots (q^m-1)}{(q-1)^m} \cdot \gcd(m,q-1)\\
    & \leq \frac{(q^m-1)^{(m+1)/2}}{(q-1)^m} \cdot \gcd(m,q-1).\end{aligned}$$
Also,
$f \leq 2^f-1 \leq q-1$. So for $m \geq 5$ we have 
$$\dim V \leq \bmax(G) \leq f\bmax(H) \leq \frac{(q^m-1)^{(m+1)/2}}{(q-1)^m} \cdot f \cdot \gcd(m,q-1) \leq 
    \biggl( \frac{q^m-1}{q-1} \biggr)^{(m+1)/2},$$
as stated. For $m = 3,4$, using \cite{Lu1} one can drop the factor of $f$ in the above estimates for 
$\bmax(H)$ and $\bmax(G)$, whence the statement follows again.

\smallskip
(iii) First we consider the case $k=1$. Then note that $R_1=P_1$ and 
$S\cap R_1 = S\cap P_1$. It follows from (\ref{E280219}) and Lemma~\ref{L270219_4} that
$$\dim V \leq [G:R_1] = [S:S \cap R_1] = [S:S \cap P_1].$$  
Now let $k \geq 2$. As recorded in \eqref{E280219}, $Y := (S \cap R_k)/(S \cap P_k) \cong PGL_k(q)$. Clearly, $|Y| > q^2 > f^2$, whence 
$\b_p(Y) \leq |Y|^{1/2} < |Y|/f$. Again using \eqref{E280219}, we obtain
$$\b_p(R_k/P_k) \leq f\cdot\b_p(Y) < |Y|.$$
Combining with Lemma~\ref{L270219_4} and \eqref{E280219}, we get 
$$\dim V < [G:R_k] \cdot |Y| = [S:S \cap R_k] \cdot |(S \cap R_k)/(S \cap P_k)| = [S:S \cap P_k].$$
Thus in both cases we have
$$\dim V \leq [S:S \cap P_k] \leq q^{k(k-1)/2}\prod^m_{i=m-k+1}(q^i-1) < q^{mk},$$
which completes the proof.
\end{proof}

\begin{Proposition}\label{PBoundSL}
Let $n \geq 324$, $m\geq 4$, $p = 2$ or $3$, $\la\in\Par_p(n)$ such that $D^\la\da_G$ is irreducible, and define $\ell$ from $n-\ell=\max(\la_1,\la^\Mull_1)$. Then 
$\ell \leq 4$ if $2 \leq q \leq 5$ and $\ell \leq 3$ if $q \geq 7$. 
\end{Proposition}

\begin{proof}
We may assume that $\ell\geq 4$, for otherwise there is nothing to prove. Replacing $\la$ by $\la^\Mull$ if necessary, we may assume that $\la_1 = n-\ell$. By Propositions \ref{PRedSL}(ii) and~\ref{PDim1}, we have  
\begin{equation}\label{eq-PB11}
  \ell \leq L(n) := 0.7\log_2 n + 1.4. % < 2(n-1)/7.
\end{equation}  
So $n\geq p(\de_p+\ell-2)$, and by Theorem \ref{TBound1},  we have
\begin{equation}\label{E270219}
\dim D^\la\geq C^p_\ell(n)>\frac{(n+3-3\ell)^\ell}{\ell !}.
\end{equation} 

\smallskip 
\noindent
 {\sf Claim 1:} If $1 \leq k \leq m-1$ and  $(q^k-1)/(q-1) \geq 2\ell$ then $q^{mk}>\frac{(n+3-3\ell)^\ell}{\ell !}$. 
 
 \noindent
 Indeed, by Proposition \ref{PRedSL}(iii), $\dim D^\la< q^{mk}$, and the claim follows from (\ref{E270219}). 
 
\smallskip 
\noindent
{\sf Claim 2:} If $k:= \lceil \log_q(2\ell-1) \rceil +1$, then $q^{mk}>\frac{(n+3-3\ell)^\ell}{\ell !}$. 

\noindent To prove Claim 2, it suffices to verify that the given $k$ satisfies the assumptions of Claim 1. Clearly $k\geq 1$, and $(q^k-1)/(q-1) \geq 2\ell$ is easy. 
Note that $(2L(n)-1)^2 < n$ by our assumption on $n$, so from (\ref{eq-PB11}), we have $2\ell-1<n^{1/2}$, and hence, using also $m \geq 4$, we get $q^2(2\ell-1) < q^{m/2}n^{1/2} < q^m$.
Now $k \leq m-1$ follows from 
$$q^k < q^{\log_q(2\ell-1) +2} = q^2(2\ell-1) < q^m.$$

\medskip
Suppose $\ell \geq 12$.  Then for $k$ as in Claim 2, we have
\begin{equation}\label{eq-PB13}
  \ell/k > \frac{\ell}{\log_2 (2\ell-1)+2} > 1.83.
\end{equation}  
On the other hand, $n = (q^m-1)/(q-1)$ implies that $m < \log_q n +1 < \frac{4}{3}\log_q n$, i.e.
\begin{equation}\label{eq-PB14}
  n > q^{3m/4}.
\end{equation}   
Also, for $n \geq 324$ we have $(L(n)/1.87)^{4.27} < n$, and so from (\ref{eq-PB11}) we get 
\begin{equation}\label{eq-PB15}
  \frac{1.87n}{\ell} > n^{0.765}.
\end{equation}  
We also have 
\begin{equation}\label{E270219_2}
n+3-3\ell >n+3-3L(n) > 14.8n/15.8
\end{equation} 
for $n \geq 324$. Using Claim 2 and (\ref{eq-PB13})--\eqref{E270219_2}, and $\ell! < (\ell/2)^\ell$ (which certainly holds for $\ell\geq 12$), we arrive at a contradiction:
$$q^{mk} > \frac{(n+3-3\ell)^\ell}{\ell !} > \biggl( \frac{14.8n/15.8}{\ell/2}\biggr)^\ell > \biggl(\frac{1.87n}{\ell}\biggr)^\ell > n^{0.765\ell} > n^{1.39k} > q^{mk}.$$
%a contradiction. 

Suppose $8 \leq \ell \leq 11$. If $q \geq 3$, we take $k$ as in Claim 2. If $q=2$ then $k=5$ satisfies the assumptions of Claim 1---indeed, $n=(2^m-1) \geq 324$ implies $m \geq 9$, and $2^5-1 > 2\ell$. As we have $k \leq 5$ for all $q$, using Claims 1 and 2 for $q=2$ and $q\geq 3$, respectively, we get
\begin{equation}\label{E270218_3}
q^{5m} \geq q^{mk} > \frac{(n+3-3\ell)^\ell}{\ell !}.
\end{equation}
If $\ell =10$ or $11$, then %the right hand side of (\ref{E270218_3}) is greater than 
$$
\frac{(n+3-3\ell)^\ell}{\ell !}\geq\frac{(n-27)^{10}}{10!} > n^7 > q^{7(m-1)},
$$
hence $5m>7(m-1)$, a contradiction. 
If $\ell = 9$, then %the right hand side of (\ref{E270218_3}) is 
$$\frac{(n+3-3\ell)^\ell}{\ell !}=\frac{(n-24)^9}{9!} > n^{20/3} > q^{20(m-1)/3},$$
hence $5m>20(m-1)/3$, a contradiction since $m\geq 4$. 
Let $\ell = 8$. Then for $q\geq 3$ we have $k=
\lceil \log_q(15) \rceil +1\leq 4$. So by Claim 2, we have 
$$q^{4m} \geq q^{mk} >  \frac{(n-21)^8}{8!} > n^{6} > q^{6(m-1)},$$
a contradiction. For $q=2$, we have $m \geq 9$, $k=5$, and we again get a contradiction: 
$$q^{5m} = q^{mk} >  \frac{(n-21)^8}{8!} > n^{6} > q^{6(m-1)}.$$

Suppose $\ell = 7$. If $q \geq 3$, choose $k$ as in Claim 2. If $q=2$, then choose $k = 4$. In both cases we have $k \leq 4$. Now we get a contradiction using Claims 1 and 2: 
$$ q^{4m} \geq q^{mk}  >  \frac{(n-18)^7}{7!} > n^{16/3} > q^{16(m-1)/3}.$$

Suppose $5 \leq \ell \leq 6$. If $q \geq 3$ take $k=3$ and apply  Claim 1 to get a contradiction:
%As $m \geq 4$ and $(3^3-1)/2 > 2\ell$, we can always take $k=3$. Then applying Theorem \ref{TBound1} and Proposition \ref{PRedSL}(iii) we have 
$$q^{3m} \geq q^{mk}  >  \frac{(n-12)^5}{5!} > n^{4} > q^{4(m-1)}.
$$
%a contradiction. Thus $\ell \leq 4$ if $q \geq 3$.
If $q =2$ take $k = 4$ and apply Claim 1 to get a contradiction:  
$$2^{4m} = q^{mk}  >  \frac{(n-12)^5}{5!} > (n+1)^{4} > 2^{4m}.$$

If $q \geq 7$ and $\ell = 4$ take $k = 2$ and apply Claim~1 to get a contradiction: 
$$q^{2m} = q^{mk}  >  \frac{(n-9)^4}{4!} > n^{8/3} > q^{8(m-1)/3}.$$
\end{proof}

\subsection{Ruling out the remaining $D^\la$ for groups with socle $PSL_m(q)$} Proposition~\ref{PBoundSL} rules out irreducible restrictions $D^\la\da_G$ in the generic case where $n \geq 324$, $m\geq 4$, and $\ell$ not too small. In this subsection we deal with the remaining cases.  

\begin{Lemma}\label{PRestSL2}
Let $p = 2$ or $3$, $\la\in\Par_p(n)$, $m \geq 5$, and $q=2$. If $D^\la\da_G$ is irreducible then $\la \in \L^{(1)}(n)$.
% or $p=3$, $2|m$, and either $\la$ or $\la^\Mull$ equals $(n-2,1^2)$. 
\end{Lemma}

\begin{proof}
By Lemma~\ref{L270219}, we have $G = S = SL_m(2)$. 
Set $V:=D^\la$ and suppose $V\da_G$ is irreducible. Write $n-\ell=\max(\la_1,\la^\Mull_1)$. Replacing $\la$ by $\la^\Mull$ if necessary, we may assume that $\la_1 = n-\ell$. We need to prove that $\ell\leq 1$.

\smallskip
\noindent
{\sf Claim 1:} $\ell \leq 5$.

\noindent
If $m \geq 9$, then $n =2^m-1\geq 511$ and so $\ell \leq 4$ by Proposition \ref{PBoundSL}. 
%We use the obvious bound 
%\begin{equation}\label{eq-sl21}
%  \dim V \leq D:=\bmax(G).
%\end{equation}
Let $m=8$. By \cite{GAP}, $\bmax(G) = 361416600 < 2^{28.5}$. As $\dim V \geq 2^{\ell/2}$ by  
Theorem~\ref{TBound2}, we have that $\ell \leq 56 < n/4$. By Theorem \ref{TBound1}, 
$$\dim V> \biggl( \frac{2(n+3)}{\ell}-6 \biggr)^\ell =  \biggl( \frac{516}{\ell}-6 \biggr)^\ell > \bmax(G),$$
for $6 \leq \ell \leq 56$, which contradicts the irreducibility of $V\da_G$. 
%For $\ell=5$, by Theorem \ref{TBound1}, $\dim V>C^p_5(n)>\bmax(G)$ by a direct computation. 
Let $m =7$. By \cite{GAP}, $\bmax(G) = 2731008 < 2^{21.5}$, whence $\ell \leq 42 < n/3$ by  
Theorem~\ref{TBound2}. By Theorem \ref{TBound1},
$$\dim V> \biggl( \frac{2(n+3)}{\ell}-6 \biggr)^\ell =  \biggl( \frac{262}{\ell}-6 \biggr)^\ell > \bmax(G),$$
for $6 \leq \ell \leq 42$, which contradicts the irreducibility of $V\da_G$. 
%For $\ell=5$, by Theorem \ref{TBound1}, $\dim V>C^p_5(n)>\bmax(G)$ again by a direct computation. 
Let $m = 6$. By \cite{GAP} for $p=3$ and \cite{Lu2} for $p=2$, 
$$\max(\bmax_2(G),\bmax_3(G)) =\max(32768,29295) =32768 < (n^3-9n^2+14n)/6,$$ 
whence $\ell \leq 2$ by Proposition \ref{PBound12}(i).  The case $m=5$ is treated similarly, using the bound 
$\dim V \leq 1024$ coming from \cite{GAP}.

\smallskip
\noindent
{\sf Claim 1:} $\ell \leq 2$.

\noindent By Claim 1, we may assume that $\ell \leq 5$, so we can take $k=4$ in Lemma \ref{L270219_5} to get 
%conclude, using $R_4/P_4 \cong SL_4(2) \cong \AAA_8$, that 
$$\dim V \leq \frac{[G:P_4]}{[R_4:P_4]}\bmax_p(R_4/P_4)=\frac{[G:P_4]}{[R_4:P_4]}\bmax_p(SL_4(2)) < \frac{(2^m-1)^4}{20160} \cdot 64 = \frac{n^4}{315}.$$
If $\ell = 4$ or $5$, then $\dim V \geq \min\{ (n-9)^4/24, (n-12)^5/120 \}$ by Theorem \ref{TBound1}, a contradiction since $n \geq 31$. So %we may assume that 
$\ell \leq 3$. Taking $k=3$ in Lemma~\ref{L270219_5}, we get %conclude that 
$$\dim V \leq  \frac{[G:P_3]}{[R_3:P_3]}\bmax_p(SL_3(2)) < \frac{(2^m-1)^3}{21} = \frac{n^3}{21} < \frac{n^3-9n^2+14n}{6},$$
since $n \geq 31$. By Proposition \ref{PBound12}(i), this implies   that $\ell \leq 2$.

\smallskip
Now we consider the case $\ell = 2$. If $\la = (n-2,1^2)$ then  $p=3$. Using Lemma \ref{LL2}(iii) one can show that 
$D^\la = \wedge^2(D^{(n-1,1)})$. Thus $SL_m(2)$ admits a non-trivial (irreducible) module $V\da_G$ whose exterior square is irreducible.
This is impossible by \cite[Proposition 3.4]{MM}. (An alternative  argument is to note that $G$ is not $3$-homogeneous and then apply 
\cite[Theorem A]{KMT1}.)

Finally, let $\la = (n-2,2)$.
% or $p=3$, $2\nmid m$ and $\la = (n-2,1^2)$. 
%In the latter case, since $p \nmid n$, by Peel's theorem $D^\la$ is obtained by reducing $\wedge^2(S^{(n-1,1)})$ modulo $3$.
%In this case, as $S^{(n-1,1)}\da_{\SSS_{n-3}}$ contains a $2$-dimensional trivial submodule, we see that $V\da_{\SSS_{n-3}}$ 
%also contains a trivial submodule, i.e. the conclusion of Lemma \ref{LSub22} holds as well.  
By Lemma~\ref{L270219_3}, $P_2\leq S_{n-3}$, and by Lemma \ref{LSub22} we have $V^{P_2} \neq 0$, so by Lemma~\ref{L270219_5}, we obtain
%with $P_2 \unlhd R_2:= \operatorname{Stab}_G(\langle e_1,e_2 \rangle_{\F_2})$ and $R_2/P_2 \cong SL_2(2) \cong \SSS_3$. As  $R_2$ acts on $V^{P_2}$ we have by the Frobenius reciprocity that
$$\dim V \leq \frac{[G:P_2]}{[R_2:P_2]}\bmax(SL_2(2)) < \frac{2(2^m-1)^2}{6} = \frac{n^2}{3} < \frac{(n^2-5n+2)}{2},$$
since $n \geq 31$. This contradicts Lemma \ref{Lr1}.
\end{proof}

\begin{Lemma}\label{PRestSL3}
Let $p = 2$ or $3$, $\la\in\Par_p(n)$, $m \geq 4$, and $q=3$. If $ D^\la\da_G$ is irreducible then $\la \in \L^{(1)}(n)$. 
\end{Lemma}

\begin{proof}
By Lemma~\ref{L270219}, we have 
$PSL_m(3) \unlhd G \leq PGL_m(3)$. 
Set $V:=D^\la$ and suppose $V\da_G$ is irreducible. 
Write $n-\ell=\max(\la_1,\la^\Mull_1)$. Replacing $\la$ by $\la^\Mull$ if necessary, we may assume that $\la_1 = n-\ell$. We need to prove that $\ell\leq 1$. 

If $m \geq 6$, then $n \geq 364$ and so $\ell \leq 4$ by Proposition \ref{PBoundSL}. If $m = 5$, by \cite{GAP}, we have  
$\bmax(G) \leq 98010 < (n^3-9n^2+14n)/6,$ 
hence $\ell \leq 2$ by Proposition \ref{PBound12}(i).  The same argument applies in the case $m=4$ where $\bmax(G) \leq 2080$.
Thus, we have $\ell \leq 4$ in all cases. 

By \cite{GAP}, $\bmax_p(SL_3(3))\leq 27$. So by Lemma~\ref{L270219_5}, we obtain 
$$\dim V  \leq \frac{27(3^m-1)(3^{m-1}-1)(3^{m-2}-1)}{(3^3-1)(3^2-1)(3-1)} 
  <  \frac{(3^m-1)^3}{416} = \frac{n^3}{52} < \frac{n^3-9n^2+14n}{6},$$
since $n \geq 40$. This in turn implies by Proposition \ref{PBound12}(i) that $\ell \leq 2$. We can take $k=2$ in 
Lemma~\ref{L270219_5}, and, using $\bmax_p(R_2/P_2) = \bmax_p(PGL_2(3))=\bmax_p(\SSS_4) \leq 3$, 
 to get 
%Proposition \ref{PRedSL}(iii). Again we have that $V^P \neq 0$, with $P \unlhd R:= \operatorname{Stab}_G(\langle e_1,e_2 \rangle_{\F_3})$ and $R/P \leq PGL_2(3) \cong \SSS_4$, whence $\bmax(R/P) \leq \bmax(\SSS_4) = 3$. As before, we see that $R$ acts on $V^P$ and then obtain by Frobenius' reciprocity that
$$\dim V \leq \frac{3(3^m-1)(3^{m-1}-1)}{(3^2-1)(3-1)} < \frac{n^2}{4} < \frac{(n^2-5n+2)}{2},$$
and so $\ell = 1$ by Lemma \ref{Lr1}.
\end{proof}

\begin{Lemma}\label{PRestSL4}
Let $p = 2$ or $3$, $\la\in\Par_p(n)$, $m \geq 4$, and $q=4$. If $D^\la\da_G$ is irreducible then $\la \in \L^{(1)}(n)$.
\end{Lemma}

\begin{proof}
By Lemma~\ref{L270219}, we have 
$PSL_m(4) \unlhd G \leq PGL_m(4) \rtimes \CCC_2$. 
Set $V:=D^\la$ and suppose $V\da_G$ is irreducible. 
Write $n-\ell=\max(\la_1,\la^\Mull_1)$. Replacing $\la$ by $\la^\Mull$ if necessary, we may assume that $\la_1 = n-\ell$. We need to prove that $\ell\leq 1$. 

If $m \geq 5$, then $n \geq 341$ and so $\ell \leq 4$ by Proposition \ref{PBoundSL}. If $m = 4$ then by \cite{GAP},  we have $\dim V \leq 2 \cdot 7140 < (n^3-9n^2+14n)/6,$ 
whence $\ell \leq 2$ by Proposition \ref{PBound12}(i).  
Thus we always have $\ell \leq 4$, and we can take $k=3$ in Lemma~\ref{L270219_5}. Note using (\ref{E280219_3}) 
that 
$R_3/P_3 \leq PGL_3(4) \rtimes \CCC_2$, so $\bmax_p(R_3/P_3)\leq2\bmax_p(PGL_3(4))=128$ by \cite{GAP}, and by Lemma~\ref{L270219_5},
$$
\dim V  \leq \frac{128(4^m-1)(4^{m-1}-1)(4^{m-2}-1)}{(4^3-1)(4^2-1)(4-1)} 
   <  \frac{2 (4^m-1)^3}{2835} = \frac{2n^3}{105} < \frac{n^3-9n^2+14n}{6}$$
since $n \geq 85$. By Proposition \ref{PBound12}(i), we have $\ell \leq 2$. So we can take $k=2$ in Lemma~\ref{L270219_5}. Note that $R_3/P_3 \leq \SSS_5$ as $\FF_4^2$ contains $5$ lines, whence $\bmax_p(R_3/P_3) \leq \bmax_p(\SSS_5) \leq 6$, and by Lemma~\ref{L270219_5}, 
$$\dim V \leq \frac{6(4^m-1)(4^{m-1}-1)}{(4^2-1)(4-1)} < \frac{3n^2}{10} < \frac{(n^2-5n+2)}{2}.$$
Now $\ell = 1$ by Lemma \ref{Lr1}.
\end{proof}

\begin{Lemma}\label{PRestSL5}
Let $p = 2$ or $3$, $\la\in\Par_p(n)$, $m \geq 4$, and $q\geq 5$. If $D^\la\da_G$ is irreducible then $\la \in \L^{(1)}(n)$.
\end{Lemma}

\begin{proof}
By Lemma~\ref{L270219}, we have 
$PSL_m(q) \unlhd G \leq PGL_m(q) \rtimes \CCC_f$. Note that $f < q/2.6$ as $q \geq 5$. 
Set $V:=D^\la$ and suppose $V\da_G$ is irreducible. 
Write $n-\ell=\max(\la_1,\la^\Mull_1)$. Replacing $\la$ by $\la^\Mull$ if necessary, we may assume that $\la_1 = n-\ell$. We need to prove that $\ell\leq 1$. 

We claim that $\ell\leq 3$. If $m = 4$ then 
$|G| \leq f \cdot |PGL_4(q)| < q^{15}f  < q^{16}/2.6,$
and
$$\dim V \leq |G|^{1/2} < \frac{q^8}{\sqrt{2.6}} < \frac{q^9}{8} < \frac{n^3}{8} <  \frac{n^3-9n^2+14n}{6},$$
hence $\ell \leq 2$ by Proposition \ref{PBound12}(i). Let  
$m \geq 5$. Then $n \geq 781$, and we have  $\ell \leq 4$ by Proposition \ref{PBoundSL}, so we may assume that $\ell=4$. 
We show that $\dim V  <  \frac{(n-9)^4}{24},$
which contradicts Theorem \ref{TBound1}.
If $m = 5$ then 
%$$|G| \leq f \cdot |PGL_5(q)| < q^{24}f  < q^{25}/2.6,$$ and so
$$\dim V \leq \sqrt{|G|} \leq \sqrt{f \cdot |PGL_5(q)|}<
\sqrt{q^{24}f}< \sqrt{q^{25}/2.6}
< \frac{q^{12.5}}{\sqrt{2.6}} < \frac{n^{25/8}}{\sqrt{2.6}} <  \frac{(n-9)^4}{24}. $$
If $m \geq 6$, Proposition \ref{PRedSL}(iii) with $k = 3$ yields
$$\dim V < q^{3m} \leq q^{3.6(m-1)} < n^{3.6} <  \frac{(n-9)^4}{24}.$$

As $\ell\leq 3$, we can take $k=2$ in Lemma~\ref{L270219}. 
Note using (\ref{E280219_3}) 
that 
$R_2/P_2 \leq PGL_2(q) \rtimes \CCC_f$, so 
$$\bmax_p(R_2/P_2)\leq f\bmax_p(PGL_2(q))\leq f(q+1),$$ 
and by Lemma~\ref{L270219_5},
\begin{align*}
\dim V &\leq f(q+1)\frac{(q^m-1)(q^{m-1}-1)}{(q^2-1)(q-1)}  = \frac{f(q^m-1)(q^{m-1}-1)}{(q-1)^2} 
\\
&<  \frac{n^2f}{q} < \frac{n^2}{2.6} < \frac{n^2-5n+2}{2}
\end{align*}
since $n \geq 156$. We conclude that $\ell=1$ by by Lemma \ref{Lr1}.
\end{proof}

\begin{Lemma}\label{LSL23}
Let $p = 2$ or $3$, $\la\in\Par_p(n)$, and $(m,q) = (3,q \geq 5)$ or $(2, q \geq 11)$. If $D^\la\da_G$ is irreducible then $\la \in \L^{(1)}(n)$.
\end{Lemma}

\begin{proof}
By Lemma~\ref{L270219}, we have 
$PSL_m(q) \unlhd G \leq PGL_m(q) \rtimes \CCC_f$. Note that $f \leq 3q/8$ as $q \geq 5$. 
Set $V:=D^\la$ and suppose $V\da_G$ is irreducible. 
Write $n-\ell=\max(\la_1,\la^\Mull_1)$. Replacing $\la$ by $\la^\Mull$ if necessary, we may assume that $\la_1 = n-\ell$. We need to prove that $\ell\leq 1$.

If $m=3$ then $n=q^2+q+1$, and 
$$\dim V \leq \bmax(G) \leq f\bmax(PGL_3(q)) \leq f(q+1)(q^2+q+1) < 3n^2/8 < (n^2-5n+2)/2 ,$$
hence $\ell = 1$ by Lemma \ref{Lr1}.

Let $m=2$, so $n=q+1$. If $q \geq 16$, we have that
$$\dim V \leq \bmax(G) \leq f\bmax(PGL_2(q)) \leq f(q+1) < 3n(n-1)/8 < (n^2-5n+2)/2 ,$$
hence $\ell = 1$ by Lemma \ref{Lr1}. If $q=13$ or $11$, then $\dim V \leq \bmax(G) = q+1=n$, and we 
conclude that $\ell = 1$ using \cite{GAP}.
\end{proof}

Now we can prove the main result of this section:

\begin{Theorem}\label{TMainSL}
Let $p=2$ or $3$, and $\la\in \Par_p(n)$ such that $\dim D^\la>1$. %$H=\AAA_n$ or $\SSS_n$, and $V$ be an irreducible constituent of $D^\la\da_H$, with $\dim V>1$. 
Suppose that $G < \SSS_n$ is a doubly transitive subgroup with $S = \soc(G) = PSL_m(q)$ acting on $n=(q^m-1)/(q-1)$ $1$-subspaces of 
$\F_q^m$,  and either $m \geq 3$, or $m = 2$ and $q \geq 4$. Then $D^\la\da_G$ is irreducible if and only if one of the following holds: 

\begin{enumerate}[\rm(i)]
\item $\la \in \L^{(1)}(n)$. Furthermore, $p \nmid q$ if $m \geq 3$, and $G \not\leq P\Sigma L_2(q)$ if $m=p=2 \nmid q$.
\item $m=2$, and one of {\rm (ii), (iii), (v)} of Proposition \ref{PSmall} occurs.
\end{enumerate}
\end{Theorem}

\begin{proof}
Define $\ell$ from $n-\ell=\max(\la_1,\la^\Mull_1)$. Then $\ell\geq 1$. Recall the notation (\ref{En-lmu}). 
Replacing $\la$ by $\la^\Mull$ if necessary,
we may assume that $\la = (n-\ell,\mu)$ for a partition $\mu$ of $\ell$.
If $(m,q)$ are as listed in Proposition \ref{PSmall}, then we are done. Otherwise we apply Lemmas \ref{PRestSL2}--\ref{PRestSL5}
when $m \geq 4$ and Lemma \ref{LSL23} when $2 \leq m \leq 3$ to conclude that $\ell = 1$, in which case  
the theorem follows from the main result of \cite{Mortimer}.
\end{proof}

\section{Doubly transitive groups $Sp_{2m}(2)$}\label{SLG3}

Throughout the section: $\de\in\{0,1\}$ , $m \geq 3$, $W$ is a $2m$-dimensional vector space over $\F_2$ with symplectic form $(\cdot,\cdot)$ and symplectic basis $(e_1,\dots,e_m,f_1,\dots,f_m)$, 
$\Omega^\de$ is the set of the quadratic forms of Witt defect $\de$ on $W$ associated with $(\cdot,\cdot)$, 
$$n =n(\de):=|\Om^\de|= 2^{m-1}(2^m + (-1)^\de),$$ 
and  $G=Sp(W)\cong Sp_{2m}(2)$ is embedded into $\SSS_n$ via its doubly transitive action on $\Om^\de$. 
For $1\leq k\leq m$ we put $W_k:=\langle e_1, \ldots ,e_k \rangle_{\F_2}$.

\subsection{Bounding $D^\la$ for $Sp_{2m}(2)$}
We follow \cite[\S7.7]{DM} and \cite[\S5]{BK}. Let $\Om$ be the set of all quadratic forms on $W$ which satisfy $Q(v+w)-Q(v)-Q(w)=(v,w)$ for all $v,w\in W$. The group $G$ acts on $\Om$ via $g\cdot Q(w)=Q(g^{-1}w)$ for $g\in G, Q\in \Om, w\in W$. Let 
$Q_0\in\Om$ be the quadratic form defined by 
$Q_0\big(\sum_{i=1}^{m}(a_ie_i+b_if_i)\big)=\sum_{i=1}^m a_ib_i. %,\quad Q_1\left(\sum_{i=1}^m(a_ie_i+b_if_i)\right)=a_m+b_m+\sum_{i=1}^{m-1}a_ib_i. 
$ 
%We fix $Q_\eps\in\Om^\eps$ such that $Q_+$ is totally singular on $W_m$ and $W_-$ is totally singular on $W_{m-1}$, cf. \cite[p.183]{BK}. 
Then $\Om=\{Q^v\mid v\in W\}$, where $Q^v(-):=Q_0(-)+(v,-)$. 
For $\de\in\{0,1\}$, we set 
\begin{equation}\label{E020319}
\Om^\de=\{Q^v\mid Q_0(v)=\de\}.
\end{equation}
By \cite[Theorem~7.7A]{DM}, $\Om^0$ and $\Om^1$ are the $G$-orbits on $\Om$, and the $G$-action on both of them is doubly transitive. Note that $Q_0=Q^0\in \Om^0$, also fix $Q_1:=Q^{e_m+f_m}\in\Om^1$. 
%Note that $\Om^0$ consists of all quadratic forms in $\Om$ which vanish on a maximal totally isotropic subspace of $W$. 

Let  $1\leq k\leq m-1$. We define certain subgroups $P^\de_k \unlhd R_k^\de\leq G$. First, let
\begin{equation}\label{eq-PSp2}
P^\de_k := \operatorname{Stab}_G(Q_\de,e_1, \ldots,e_k) = \operatorname{Stab}_{O(Q_\de)}(e_1, \ldots,e_k)
\end{equation}  
be the subgroup of $G$ that fixes $Q_\de$ and each of $k$ vectors $e_1, \ldots, e_k$. Also set 
\begin{equation}\label{eq-PSp4}
R_k^\de := \operatorname{Stab}_{O(Q_\de)}(W_k). 
\end{equation}
Note that $P_k^\de \unlhd R_k^\de$ and 
\begin{equation}\label{eq-PSp4N}
%P_k^\de \unlhd R_k^\de \quad \text{and}\quad 
R_k^\de/P_k^\de \cong SL_k(2). 
\end{equation}

\begin{Lemma} \label{L020319} %{\rm \cite{}}%{\bf ()}
Let  $1\leq k\leq m-1$. Then $P^\de_k$ fixes $2^k$ quadratic forms in $\Om^\de$, so $P_k^\de$ is contained in a natural subgroup $\SSS_{n-2^k}$ in $\SSS_n$. 
\end{Lemma}
\begin{proof}
Note that $P^\de_k$ fixes each of the $2^k$ forms $\{Q^{v+\de(e_m+f_m)}\mid v \in W_k\}$ in $\Om^\de$.
\end{proof}

\begin{Lemma} \label{L020319_2} %{\rm \cite{}}%{\bf ()}
Let  $\la = (n-\ell, \mu) \in\Par_p(n)$ for a partition $\mu$ of $\ell$. For an integer $1 \leq k \leq m-1$ such that $2^{k-1} \geq \ell$, we have $(D^\la)^{P_k^\de}\neq 0$. In particular, if $D^\la\da_G$ is irreducible then $\dim D^\la\leq [G:P_k^\de]$. 
\end{Lemma}
\begin{proof}
The first statement follows from Lemma~\ref{L020319} and Theorem \ref{TSub}. The second one then follows from the Frobenius Reciprocity. 
\end{proof}

\begin{Lemma} \label{L050319} %{\rm \cite{}}%{\bf ()}
Let $\la=(n-\ell,\mu)\in\Par_p(n)$ with $\mu \vdash \ell$ and $D^\la\da_G$ be irreducible. For an integer $1 \leq k \leq m-1$ such that $(D^\la)^{P_k^\de}\neq 0$, we have 
$$
\dim D^\la \leq [G:R^\de_k]\bmax_p(SL_k(2))= 
\frac{[G:P^\de_k]}{|SL_k(2)|}\bmax_p(SL_k(2)). 
%=\bmax_p(R_k/P_k)\prod_{i=1}^k\frac{q^{n-i+1}-1}{q^i-1}.
$$
The assumption $(D^\la)^{P^\de_k}\neq 0$ is guaranteed if 
$2^{k-1} \geq \ell$.  
\end{Lemma}
\begin{proof}
%By Lemma~\ref{L270219_4}, $V^{P_k}\neq 0$. 
Note that $R_k^\de$ acts on $(D^\la)^{P^\de_k}$ and the $R^\de_k$-module $(D^\la)^{P^\de_k}$ contains a simple submodule 
$X$ of dimension at most $\bmax_p(R_k/P_k)$. By the Frobenius reciprocity, we have
$$\dim D^\la \leq [G:R^\de_k]\dim X \leq [G:R^\de_k]\bmax_p(R_k/P_k)=\frac{[G:P^\de_k]}{[R^\de_k:P^\de_k]}\bmax_p(R^\de_k/P^\de_k),$$
and it remains to use (\ref{eq-PSp4N}) and Lemma~\ref{L020319_2}. 
\end{proof}

\begin{Lemma} \label{L070519} %{\rm \cite{}}%{\bf ()}
For $1\leq k\leq m-1$, we have 
$$
[G:P_k^\de] \,=\, 2^{m-1+k(k-1)/2}  (2^{m-k}+(-1)^\de) \prod^m_{i=m-k+1}(2^{2i}-1)\, <\, 2^{(4m-k)(k+1)/2}.
$$ 
\end{Lemma}
\begin{proof}
Let $\eps:=+$ if $\de=0$ and $\eps:=-$ if $\de=1$. We have 
$$\begin{aligned}\dim V & \leq [G:P_k^\de] = [G:O(Q_\de)] \,  [O(Q_\de):P_k^\de]\\
   & = 2^{m-1}(2^m+(-1)^\de) \, [O^\eps_{2m}(2):(GL_k(2) \times O^\eps_{2m-2k}(2))]_{2'} \, |GL_k(2)|\\ 
   & = 2^{m-1}(2^m+(-1)^\de) (2^m-(-1)^\de)(2^{m-k}+(-1)^\de)  \hspace{-1mm}\prod^{m-1}_{i=m-k+1}(2^{2i}-1) \cdot 2^{k(k-1)/2}\\
   & = 2^{m-1+k(k-1)/2}  (2^{m-k}+(-1)^\de) \prod^m_{i=m-k+1}(2^{2i}-1) < 2^{(4m-k)(k+1)/2},\end{aligned}$$  
 as required. 
\end{proof}

\begin{Proposition}\label{PRedSp}
Let $V$ be an irreducible $\F G$-module.  %the following statements hold.
\begin{enumerate}[\rm(i)]
\item If $m \geq 4$ then $\dim V < (\sqrt{3}/2) \cdot 2^{m(m+1/2)} < n^{\frac{1}{4}\log_2 n +1.52}$.
\item Suppose that $V = D^\la\da_G$ for $\la = (n-\ell, \ldots) \in\Par_p(n)$, and that  
there exists an integer $1 \leq k \leq m-1$ such that $2^{k-1} \geq \ell$. Then 
$\dim V < 2^{(4m-k)(k+1)/2}.$
\end{enumerate}

\end{Proposition}

\begin{proof}
(i) 
%The first inequality comes from $\dim V\leq |G|^{1/2}<(3 \cdot 2^{m(2m+1)-2})^{1/2}.$ 
Note that $n > 2^{2m-2}$, so 
$m < \frac{1}{2}\log_2 n+1$. As $m\geq 4$, we have $n \geq 28$,  and
$$\begin{aligned}
\dim V\leq |G|^{1/2}&<(3 \cdot 2^{m(2m+1)-2})^{1/2}
=
(\sqrt{3}/2) \cdot 2^{m(m+1/2)} 
\\& < (\sqrt{3}/2) \cdot 2^{(1+\frac{1}{2}\log_2 n)(\frac{3}{2}+\frac{1}{2}\log_2 n)}
 = \sqrt{6} \cdot 2^{(\log_2 n) \cdot (\frac{1}{4}\log_2 n+\frac{5}{4})}\\
    & =  \sqrt{6} \cdot n^{\frac{1}{4}\log_2 n+\frac{5}{4}} < n^{\frac{1}{4}\log_2 n +1.52}.\end{aligned}$$
%completing the proof of (i).

\smallskip
(ii) 
Follows from Lemmas~\ref{L020319_2} and \ref{L070519}.
\iffalse{
, we have 
$$\begin{aligned}\dim V & \leq [G:P_k^\de] = [G:O(Q_\de)] \,  [O(Q_\de):P_k^\de]\\
   & = 2^{m-1}(2^m+(-1)^\de) \, [O^\eps_{2m}(2):(GL_k(2) \times O^\eps_{2m-2k}(2))]_{2'} \, |GL_k(2)|\\ 
   & = 2^{m-1}(2^m+(-1)^\de) (2^m-(-1)^\de)(2^{m-k}+(-1)^\de)  \hspace{-1mm}\prod^{m-1}_{i=m-k+1}(2^{2i}-1) \cdot 2^{k(k-1)/2}\\
   & = 2^{m-1+k(k-1)/2}  (2^{m-k}+(-1)^\de) \prod^m_{i=m-k+1}(2^{2i}-1) < 2^{(4m-k)(k+1)/2},\end{aligned}$$  
 as required. 
 }\fi
\end{proof}

\begin{Proposition}\label{PBoundSp}
Let $m \geq 3$, $p = 2$ or $3$, $\la\in\Par_p(n)$ such that $D^\la\da_G$ is irreducible. Determine $\ell$ from $n-\ell=\max(\la_1,\la^\Mull_1)$. Then $\ell \leq 2$ if $m=3,4$, and $\ell \leq 3$ if $m \geq 5$. 
\end{Proposition}

\begin{proof}
Replacing $\la$ by $\la^\Mull$ if necessary, we may assume that $\la = (n-\ell, \mu)$ for a partition $\mu$ of $\ell$. 
Let $m = 4$. By \cite{GAP}, $\bmax_p(Sp_8(2))\leq 2^{16}$. So 
$$\dim D^\la \leq 2^{16} < (n^3-9n^2+14)/6,$$ 
hence $\ell \leq 2$ by Proposition \ref{PBound12}(i). The same argument applies to the case $m=3$. So we may assume that $m\geq 5$, hence $n \geq 496$, and $\ell\geq 4$. 

By Proposition \ref{PRedSp}(i), we have
that $\dim V < n^{L(n)}$ with 
$$L(n) := \frac{1}{4}\log_2 n + 1.52 < \frac{1}{2}\log_2 n+1.$$
So by Proposition \ref{PDim1}, we get 
\begin{equation}\label{eq-PB21}
  \ell \leq L'(n):=0.7\log_2 n +1.4 < n/p+2-\de_p.
\end{equation}  
Now Theorem \ref{TBound1} applies to give $\dim D^\la\geq C^p_\ell(n)$, hence 
\begin{equation}\label{E050319}
n^{L(n)}>\dim D^\la\geq C^p_\ell(n)=\frac{1}{\ell!}\prod_{i=0}^{\ell-1}(n-(\de_p+i)p).
\end{equation}

Assume that $\ell = 4$. By Lemmas~\ref{L050319} and \ref{L070519} with $k = 3$, 
% and apply Proposition \ref{PRedSp}(ii). As we noted in the proof of Proposition \ref{PRedSp}(ii), $V^P \neq 0$ and $[G:P] < 2^{8m-6}$. Now, as $P \unlhd R$ and $R/P \cong SL_3(2)$, we see that $R$ acts on $V^P$ and the $R$-module $V^P$ contains a simple submodule $X$ of dimension $\dim X \leq \bmax(SL_3(2)) =8$. By Frobenius' reciprocity, 
we have
$$\dim V \leq \frac{[G:P_3^\de]}{|SL_3(2)|}\bmax_p(SL_3(2)) < \frac{2^{8m-6}}{168} \cdot 8 = \frac{2^{8m-6}}{21}.$$
On the other hand, (\ref{E050319}) implies 
$$\dim V \geq \frac{(n-9)^4}{24} \geq \frac{(2^{m-1}(2^m-1)-9)^4}{24} > \frac{2^{8m-6}}{21}$$
as $m \geq 5$, a contradiction. So we may assume that $\ell\geq 5$.

Assume that $\ell = 5$. By Lemma~\ref{L050319} with $k = 4$, %noting that $[G:P_4^\de] < 2^{10m-10} < n^5$, 
%We noted in  the proof of Proposition \ref{PRedSp}(ii) that $V^P \neq 0$ and $[G:P] < 2^{10m-10} < n^5$. Now, as $P \unlhd R$ and $R/P \cong SL_4(2) \cong \AAA_8$, we see that $R$ acts on $V^P$ and the $R$-module $V^P$ contains a simple submodule $X$ of dimension $\dim X \leq \bmax(\AAA_8) =64$. By Frobenius' reciprocity, 
we have 
$$\dim V \leq \frac{[G:P_4^\de]}{|SL_4(2)|}\bmax_p(SL_4(2))  < \frac{n^5}{20160} \cdot 64 = \frac{n^5}{315},$$
where we have used Lemma~\ref{L070519} to get $[G:P_4^\de] < 2^{10m-10} < n^5$. 
On the other hand, (\ref{E050319}) implies 
$$\dim V \geq \frac{(n-12)^5}{120} > \frac{n^5}{315}$$
as $n \geq 496$, a contradiction.

%\smallskip
%\noindent
%{\sf Claim:} $\ell\leq 5$. 

%\noindent
%Otherwise 
Now we may assume that $\ell\geq 6$. In particular, $\ell!<(\ell/2)^\ell$, and by (\ref{E050319}), we get 
\begin{equation}\label{E020319_6}
n^{L(n)}>\dim D^\la >\frac{(n+3-3\ell)^\ell}{\ell!}>\left(\frac{2(n+3)}{\ell}-6\right)^\ell.
\end{equation}
If $\ell \geq 1.3L(n)$, then 
$$
n^{L(n)}>\biggl( \frac{2(n+3)}{\ell}-6 \biggr)^{1.3L(n)}
$$ and so, since $n \geq 496$,  
$$\ell \geq \frac{2(n+3)}{n^{1/1.3}+6} > L'(n),$$
contradicting \eqref{eq-PB21}. So  
\begin{equation}\label{eq-PB22}
  \ell < 1.3L(n) < 0.33\log_2 n + 2.
\end{equation}  
Now $\ell \geq 6$ implies $m \geq 7$ and $n \geq 8128$. 
In this case, \eqref{eq-PB22} implies that 
\begin{equation}\label{eq-PB23}
  \ell < \sqrt{n}/16.
\end{equation}
As $n < 2^{2m}$, we get $\ell < 2^{m-4}$ and so for 
\begin{equation}\label{eq-PB24}
  k:= \lceil \log_2 \ell \rceil +1 %< \log_2 \ell +2,
\end{equation}  
we have  $1 < k < m-2$ and $2^{k-1} \geq \ell$. By Proposition \ref{PRedSp}(ii), we now conclude that 
\begin{equation}\label{E020319_5}
\dim D^\la < 2^{(4m-k)(k+1)/2}.
\end{equation}

If $\ell \geq 14$ then \eqref{eq-PB24} implies that $k+1 < \log_2 \ell+3 < \ell/2$.
As $n > 2^{2m-2}$ and $k \geq 4$, we then have from (\ref{E020319_5}) that 
$$\dim D^\la <  2^{(2m-2)(k+1)} < n^{k+1} < n^{\ell/2}.$$
On the other hand, using \eqref{eq-PB23} and (\ref{E020319_6}),  we obtain
$$\dim D^\la > \biggl( \frac{2(n+3)}{\ell}-6 \biggr)^\ell > (32\sqrt{n}-6)^\ell > n^{\ell/2},$$
a contradiction.

If $9 \leq \ell \leq 13$ then $k=5$ by \eqref{eq-PB24}. Using  
(\ref{E020319_6}) and (\ref{E020319_5}) we get
$$\frac{(n-36)^9}{13!} < \dim D^\la < 2^{12m-15} < n^6,$$
which is a contradiction since $n \geq 8128$.

If $6 \leq \ell \leq 8$ then $k=4$ by \eqref{eq-PB24}. Again using  
(\ref{E020319_6}) and (\ref{E020319_5}), we obtain
$$\min\left\{\frac{(n-15)^6}{6!},\frac{(n-21)^7}{8!} \right\} < \dim V < 2^{10m-10} < n^5,$$
a contradiction. The proof of the claim is complete. 
\end{proof}

\subsection{Ruling out the remaining $D^\la$ for $Sp_{2m}(2)$}

%This idea is used in the following lemma. 

\begin{Proposition}\label{PSp233}
Let $p=3$, $n\geq 28$, and $\la=(n-a,a)$ with $a=2$ or $3$. 
%$\la=(n-a,a)$ with $2\leq a\leq 3$. 
Then $D^\la\da_G$ is reducible.
\end{Proposition}

\begin{proof}
Note that $n=2^{m-1}(2^m + (-1)^\de)\not\equiv{2}\pmod{3}$. 
In particular, by Lemma~\ref{l4}, if $\la=(n-2,2)$ we may assume that $n\equiv 1\pmod{3}$. We will use the following notation from \cite{KMT1}: 
\begin{align*}
&S_k:=S^{(n-k,k)},\ M_k:=M^{(n-k,k)},\ i_k(G):=\dim M_k^G,
\\
&\EE(\la):=\End_\F(D^\la),\ \I(G):=\ind_G^{\SSS_n}\F_G.
\end{align*}
%$S_k:=S^{(n-k,k)}$, $M_k:=M^{(n-k,k)}$, $\EE(\la):=\End_\F(D^\la)$, $\I(G):=\ind_G^{\SSS_n}\F_G$, and $i_k(G):=\dim M_k^G$.

If $\la=(n-2,2)$ and $n\equiv 1\pmod{3}$, or $\la=(n-3,3)$ and $n\equiv 0\pmod{3}$, then by \cite[Lemma 6.8]{M}, there exists a homomorphism $\zeta:M_3\to\EE(\la)$ with $[\im\zeta:D_3]\neq 0$. If $\la=(n-3,3)$ and $n\equiv 1\pmod{3}$, then  by \cite[Lemma 6.12]{M}, there exists a homomorphism $\zeta:M_3\to\EE(\la)$ with $[\im\zeta:D_3]\neq 0$ or there exists a homomorphism $\zeta':M_4\to\EE(\la)$ with $[\im\zeta':D_4]\neq 0$.

From \cite[Corollary 2.31]{KMT1}, $(S_1^*)^G=0$. Further if $n\equiv 1\pmod{3}$ then by \cite[Corollary 7.5]{KMT1}, we have $(S_2^*)^G=0$. By \cite[Lemmas 5.11, 5.12]{BK}, we have  $i_2(G)=1$, $i_3(G)=2$ and $i_4(G)>2$. It then follows by \cite[Lemmas 3.3, 3.4]{KMT1} that there exists a homomorphism $\psi:\I(G)\to M_3$ with $[\im\psi:D_3]\neq 0$. If $n\equiv 1\pmod{3}$ then by \cite[Lemma 3.5]{M} there exists a homomorphism $\psi':\I(G)\to M_4$ with $[\im\psi':D_4]\neq 0$.

Therefore, $[\im(\zeta\circ\psi):D_3]\neq 0$ or $[\im(\zeta'\circ\psi'):D_4]\neq 0$. The proposition then follows from \cite[Lemma 2.18]{KMT1}.
\end{proof}

\begin{Lemma}\label{LSp232}
Let $p=2$, $n\geq 28$, and $\la=(n-a,a)$ with $a=2$ or $3$. Then $D^\la\da_G$ is reducible.
\end{Lemma}

\begin{proof}
Assume the contrary. Note that $D^\la$ is a subquotient of the $\F\SSS_n$-module 
$\Sym^a(U)$, where $U$ denotes the $\F\SSS_n$-permutation module on the set $\Om^\de$ of cardinality $n=2^{m-1}(2^m + (-1)^\de$.
It is shown on \cite[p. 10]{Mortimer} that the $\F G$-module $U$ contains a subquotient $B$ of dimension $2m+1 \geq 7$. Thus, in the 
Grothendieck group of $\F G$-modules we can write $U = A +B$ for a $\F G$-module $A$ of dimension $n-(2m+1)$. Note that 
$\dim A \geq 4(\dim B)$ 
since $m \geq 3$. This implies that, in the following decomposition in the Grothendieck group
\begin{equation}\label{eq-PR31}
  \Sym^a(U) = \sum^a_{i=0}\Sym^{a-i}(A) \otimes \Sym^i(B),
\end{equation}  
the summand $\Sym^a(A)$ has the largest dimension. Now, by Proposition \ref{PBound12}(i) we have
$$\dim \Sym^3(A) \leq (n-5)(n-6)(n-7)/6 < (n^3-9n^2+14n)/6 \leq \dim D^{(n-3,3)}$$
as $n \geq 28$. Thus, when $a=3$ every summand in \eqref{eq-PR31} has dimension less than $\dim V$, and so $V\da_G$ cannot be irreducible.
Likewise, 
$$\dim \Sym^2(A) \leq (n-6)(n-7)/2 < (n^2-5n+2)/2 \leq \dim D^{(n-2,2)}$$
by Lemma \ref{Lr1}. Thus, when $a=2$ every summand in \eqref{eq-PR31} has dimension less than $\dim V$, and so $V\da_G$ cannot be irreducible.
\end{proof}

%{\bf Sasha believes} that when $p=2$ and $2|n \geq 6$, $D^{(n-3,2,1)}$ is obtained by reducing the Specht module
%$S^{(n-3,2,1)}$, hence we can apply the result of \cite{S} and \cite{BK} in the complex case to rule $\la = (n-3,2,1)$ out.
%{\bf What to do when $p = 3$? }

%We also need to {\bf ***check that $h(\la^\Mull) \geq 3$***}. With this done, 
%So we have completed the proof of 

We now prove the main result of the section:

\begin{Theorem}\label{TMainSp}
Let $m \geq 3$, $\de=0$ or $1$, and let $G = Sp_{2m}(2) < \SSS_n$ with 
$G=Sp(W)$ acting on the $n = 2^{m-1}(2^m+(-1)^\de)$ quadratic forms of Witt defect $\de$ on the symplectic space $W:= \F_2^{2m}$. 
Let $p = 2$ or $3$, and let $\la\in\Par_p(n)$ be such that $\dim D^\la >1$.  Then $D^\la\da_G$ is irreducible if and only 
if $p=3$ and $\la=(n-1,1)$ or $(n-1,1)^\Mull$. 
\end{Theorem} 

\begin{proof}
Assume that $D^\la\da_G$ is irreducible. By Proposition \ref{PBoundSp}, we may assume that $\la = (n-\ell,\mu)$ with $\ell \leq 3$ and 
$\mu \vdash \ell$.
Next, by Proposition \ref{PSp233} and Lemma \ref{LSp232}, $\la \neq (n-2,2)$, $(n-3,3)$.  

The cases where $\la = (n-3,2,1)$, or $p = 3$ and $\la = (n-2,1^2)$, when $G = Sp_{2m}(2) < \SSS_n$ with $m \geq 3$, are ruled out by 
\cite[Theorem A]{KMT1}. Indeed, it was shown in \cite[Lemma 5.11]{BK} that $G$ has (exactly) two orbits on the set of $3$-element subsets 
of $\Om^\eps$, and so $G$ is not $3$-homogeneous. Also by \cite[Lemma 2.2]{bkz} we have that if $p=3$ and $n\geq 10$ then $(n-3,2,1)^\Mull=((n-3)^\Mull,3)$ and $(n-2,1^2)^\Mull=((n-2)^\Mull,2)$, so that in either case $h(\la^\Mull)=3$. 

This leaves only one possibility $\la = (n-1,1)$.
Now we apply the main result of \cite{Mortimer} to see that $D^{(n-1,1)}$ is irreducible over $G$ if and only if $p=3$.
\end{proof}

\section{Proofs of Main Theorems}
\label{SMain}

%\begin{proof}[Proof of Theorem A]
\subsection{Proof of Theorem A}
For $p\geq 5$ this is \cite[Main Theorem]{BK} (and Remark \ref{R180319_2}). 
So we may assume that $p=2$ or $3$. 
%Note that $D^{(10,2)}\da_{M_{11}}\cong D^{(9,2)}\da_{M_{11}}$ with $M_{11}\leq \SSS_{11}$ is not irreducible for $p=5$.
Since the case $(p,\la)=(2,\be_n)$ is excluded, by  \cite[Theorems A, B]{KMT1}, we may assume that one of the following happens:

\begin{enumerate}
\item[{\sf (1)}] $p=2$, $n\equiv 2\pmod{4}$, $\la=(n-1,1)$, and $G\leq\SSS_{n/2}\wr\SSS_2$ is as in \cite[Theorem~B]{KMT1};
\item[{\sf (2)}] $G$ is $2$-transitive on $\{1,\dots,n\}$;
\item[{\sf (3)}] $G\leq \SSS_{n-1}$ and $\la$ is JS.
\end{enumerate}

Since {\sf (1)} is  Theorem~A(iii), we assume from now on that this case does not occur.

Suppose we are in the case {\sf (2)}. 
If $\la\in\L^{(1)}(n)$ then by \cite{Mortimer} and the remarks preceding Table II, we arrive at Theorem~A(ii). If $G=\AAA_n$, then, by definition of $\Parinv_p(n)$, we arrive at Theorem~A(i). So we may assume that $G\neq \AAA_n$ and $\la\not\in\L^{(1)}(n)$. 
By the classification of $2$-transitive groups  \cite{Cameron}, we are in one of the following situations:
\begin{enumerate}
\item[\rm (A)] $\soc(G)$ is an elementary abelian subgroup;

\item [\rm (B)] $\soc(G) \cong PSL_m(q)$ (is non-abelian simple)  acting on $n = (q^m-1)/(q-1)$ $1$-dimensional subspaces of $ \F_q^m$;

\item [\rm (C)] $G\cong Sp_{2m}(2)$, $m \geq 3$, acting on $n = 2^{m-1}(2^m + (-1)^\de)$ quadratic forms on $\F_2^{2m}$ of the given Witt defect $\de \in \{0,1\}$; % (we interpret $2^m+\eps$ for $\eps = \pm$ as $2^m+\eps 1$);

\item [\rm (D)] $G$ is any of the other doubly transitive subgroups (which we call small).
\end{enumerate}
We now apply Theorems~\ref{TAGL}, \ref{TMainSL}, \ref{TMainSp}, and \ref{T2small} for the cases (A), (B), (C) and (D), respectively. 

Suppose we are in the case {\sf (3)}. If $n=5$ then $\la$ is JS only if $\la=(5)$ and $p=2$, or $\la\in\{(5),(3,2)\}$ and $p=3$, and in either case we have $\dim D^\la=1$. So we may assume that $n\geq 6$. By \cite{k2}, we have $D^\la\da_{\SSS_{n-1}}\cong D^\mu$, where $\mu$ is obtained from $\la$ by removing the top removable node of $\la$. 
If $G=\SSS_{n-1}$ we arrive at Theorem~A(v). Now we may assume that $G<\SSS_{n-1}$. 

We now apply \cite[Theorems A, B]{KMT1} again with $n-1$ in place of $n$ and $\mu$ in place of $\la$ to arrive to the cases {\sf (1'),(2'),(3')} parallel to the cases {\sf (1),(2),(3)} above. 
%The case {\sf (1)} is excluded since $\dim D^\mu=\dim D^\la>1$. 
For example, by \cite[Theorems 3.3, 3.6]{KS2Tran}, $\mu$ is not JS, so {\sf (3')} is excluded. The case {\sf (1')} is also excluded, since $\mu=(n-2,1)$ implies $\la=(n-1,1)$, but $n-1\equiv 2\pmod{4}$ implies that $n$ is odd, and so $\la$ is not JS. Thus, we are in the case {\sf (2')}, i.e. $G$ is $2$-transitive on $\{1,2,\dots,n-1\}$.  

Suppose $G=\AAA_{n-1}$. Then $D^\la\da_{\AAA_{n-1}}\cong D^\mu\da_{\AAA_{n-1}}$ is irreducible if and only if $\mu\not\in\Parinv_p(n-1)$. If $p=2$, since $\la\neq \be_n$ is JS, it can be easily seen from Lemma~\ref{LBenson} that $\mu\not\in\Parinv_2(n-1)$ if and only if $\la\not\in\Parinv_2(n)$. If $p\not=2$, since $\la$ is JS, we have from \cite[Theorem 5.10]{BO2} that $\mu\not\in\Parinv_p(n-1)$ if and only if $\la\not\in\Parinv_p(n)$. So $D^\la\da_{\AAA_{n-1}}$ is irreducible if and only if $\la$ is JS and $\la\not\in\Parinv_p(n)$. We have arrived at Theorem~A(vi).

Assume finally that $\AAA_{n-1}\neq G<\SSS_{n-1}$. As $n\geq 6$, passing from $\la$ to $\la^\Mull$ if necessary, we may assume  by Theorems~\ref{TAGL}, \ref{TMainSL}, \ref{TMainSp}, and \ref{T2small} that $\mu=(n-2,1)$, $(n-3,2)$ or $(n-3,1^2)$ (the last partition only for $p=3$). Since $\mu$ is obtained from $\la$ by removing the top removable node it follows that $\la=(n-1,1)$, $(n-2,2)$ or $(n-2,1^2)$ respectively. Note that $(n-1,1)$ and $(n-2,1^2)$ are JS if and only if $n\equiv 0\pmod{p}$, while $(n-2,2)$ is JS if and only if $n\equiv 2\pmod{p}$. The result then easily follows in this case by checking when 
the required congruences modulo $p$ hold
and when $D^\mu\da_G$ is irreducible using Theorems~\ref{TAGL}, \ref{TMainSL}, \ref{TMainSp}, and \ref{T2small}.

%\end{proof}

\subsection{Proof of Theorem A$'$}
%\begin{proof}[Proof of Theorem A$'$]
For $p\geq 5$ this is \cite[Main Theorem]{KS}. So we may assume that $p=2$ or $3$. If $V$ lifts to $\SSS_n$, we arrive at Theorem~A$'$(i). Otherwise $\la\in\Parinv_p(n)$. From Lemma~\ref{L4.3} it then follows that $\la_1\leq (n+4)/2$. By  \cite[Theorem A]{KMT3}, we are in one of the following situations:

\begin{enumerate}
\item[{\sf (1)}] $G$ is primitive on $\{1,2,\dots,n\}$;

\item [{\sf (2)}] $G\leq\AAA_{n-1}$ and either $\la$ is JS or $\la$ has exactly two normal nodes, both of residue different from $0$. 

\item [{\sf (3)}] $G\leq\AAA_{n-2,2}\cong\SSS_{n-2}$ and $\la$ is JS.
\end{enumerate}

%The case {\sf (1)} is covered in Theorems  \ref{ext-spin}, \ref{T2small} and \ref{TAGL}, giving Theorem~A$'$(ii)(c),(d),(e). 
Suppose we are in the case {\sf (1)}. By Theorem \ref{ext-spin}, we see that either $G$ is an affine group, which is subsequently ruled out by Theorem \ref{TAGL}, or 
$G$ is a Mathieu group, in which case one can apply Theorem \ref{T2small} to arrive at the case (A1) from Table IV, or else the case (A3) from Table IV occurs.

Consider the case {\sf (2)}. Suppose first that $\la$ is JS. Then  $E^\la_\pm\da_{\AAA_{n-1}}\cong E^\pi_\pm$ for some $\pi\in\Parinv_p(n-1)$. As $\pi$ can not  be JS, applying  \cite[Theorem A]{KMT3} to $n-1$ instead of $n$, we deduce that either $G$ is a subgroup of $\AAA_{n-1}$ primitive on $\{1,2,\dots,n-1\}$, or $G\leq\AAA_{n-2}$. The former case is considered as in the case {\sf (1)} using Theorems  \ref{ext-spin}, \ref{T2small} and \ref{TAGL}.  The case $G\leq\AAA_{n-2}$ is subsumed by the case  {\sf (3)} to be considered below. 

Suppose now that $\la$ has exactly two normal nodes both of residue different from 0. From \cite[Proposition 3.8]{KS} or the proofs of \cite[Theorems B, 5.3]{KMT3} we have that $D^\la\da_{\AAA_{n-1}}\cong E^\nu$ with $\nu\in\Par_p(n-1)\setminus\Parinv_p(n-1)$ obtained by removing a good node from $\la$. If $p=3$ we also have that $\nu^\Mull$ is obtained from $\la$ by removing a good node. In particular $\nu_1,\nu_1^\Mull\leq (n+4)/2<(n-1)-2$ if $n\geq 11$, so by Theorem A we have that $G\in\{\AAA_{n-2},\AAA_{n-1}\}$. Using \cite[Theorem A]{KMT3}, we arrive at Theorem~A$'$(ii)(a). If $n\leq 10$ and $p=2$, $\la=(4,3,1)$ and $\nu=(4,2,1)$, in which case we can conclude as above. If $n\leq 10$ and $p=3$, $\la=(3,1^2)$ and $\nu^{(\Mull)}=(3,1)$. In this case $E^\la_\pm\cong E^{(4,1^2)}_\pm\da_{\AAA_5}$, which will be considered below when covering case {\sf (3)}.

Consider the case {\sf (3)}. Using the isomorphism $\AAA_{n-2,2}\cong\SSS_{n-2}$, by \cite[Theorem 5.4]{KMT3} and \cite[Theorem 3.6]{KS}, we can write  $E^\la_\pm\da_{\AAA_{n-2,2}}\cong D^\mu$ where $\mu\in\Par_p(n-2)\setminus\Parinv_p(n-2)$ is obtained from $\la$ by removing two good nodes. If $p=3$ we also have that $\mu^\Mull$ is obtained from $\la$ by removing two good nodes. In particular $\mu_1,\mu_1^\Mull\leq (n+4)/2<(n-2)-2$ if $n\geq 13$. In this case it follows from Theorem A that $G\in\{\AAA_{n-2},\AAA_{n-3}\}$. The case $G=\AAA_{n-3}$ can be excluded, since $E^\la_\pm\da_{\AAA_{n-3,3}}$ is not irreducible by \cite[Theorem A]{KMT3}. So $G\in\{\AAA_{n-2},\AAA_{n-2,2}\}$, in which case $E^\la_\pm\da_G$ is irreducible by \cite[Theorem C]{KMT3}, and we arrive at Theorem~A$'$(ii)(b). If $n\leq 12$ then $p=2$, $\la=(5,3,1)$ and $\mu=(4,2,1)$ or $p=3$ and $(\la,\mu^{(\Mull)})\in\{((4,1^2),(3,1)),((7,3,2),(5,3,2))\}$. If $p=2$ and $\la=(5,3,1)$ or $p=3$ and $\la=(7,3,
 2)$ we can conclude as above. If $p=3$ and $\la=(4,1^2)$ then $E^\la_\pm\da_{\AAA_{4,2}}\cong D^{(3,1)^{(\Mull)}}$. Since $\dim E^\la_\pm=3$, we have that $E^\la_\pm\da_G$ is reducible if $G$ is abelian or a 2-group. Further $E^\la_\pm\da_{\AAA_{3,2}}\cong D^{(3,1)^{(\Mull)}}\da_{\SSS_3}$ is reducible, under the identification of $\AAA_{3,2}\cong \SSS_3$. Considering the submodule structure of $\SSS_4$ it then follows that $G\in\{\AAA_{4,2},\AAA_4\}$ and so we can again conclude by \cite[Theorem C]{KMT3}.

%\end{proof}

%\begin{proof}[Proof of Theorem B$'$]

\subsection{Proof of Theorem B}
%\begin{proof}[Proof of Theorem C]
\medskip
For the `if' direction, by Theorem~A, the cases listed in Theorem B do give rise to irreducible restrictions $D^\la\da_G$. 

For the `only-if' direction, assume that 
$D^\la\da_G$ is irreducible. By Schur's Lemma, $\ZB(G)$ acts 
on $D^\la$ via scalars, and so $\ZB(G) \leq \ZB(\SSS_n) = 1$ as $\SSS_n$ acts faithfully on $D^\la$. Thus $G$ is in fact almost simple, i.e. 
$S \unlhd G \leq \Aut(S)$ for a non-abelian simple group $S$. 
Inspecting the list of exceptions in Theorem A for almost simple groups, we conclude that it is enough to show that such a group cannot occur in the case (iii) of Theorem~A. 
%Assuming this has been done, we arrive at one of the possibilities (i)--(xii) listed in Theorem B.

So assume for a contradiction that $G$ is almost simple with socle $S$ and satisfies the conditions described in Remark~\ref{R180319}. Recall that $B=\SSS_{n/2,n/2}$ is the base subgroup of $\SSS_{n/2} \wr \SSS_2$. 
As $G$ is almost simple, we have $S \unlhd G \cap B$, and 
\begin{equation}\label{eq-out}
  2 \mbox{ divides }|\Out(S)|.
\end{equation} 
Let $\pi_1$ (resp. $\pi_2$) denote the permutation representations of {\it odd} degree $n/2$ of $G \cap B$, induced by the projection of $B$
onto the first (resp. second) factor $\SSS_{n/2}$ of $B$. By assumption, $\pi_i(G \cap B)$ is $2$-transitive,
but the homomorphisms 
$$G \cap B \stackrel{\pi_i}{\longrightarrow} \SSS_{n/2} \to GL(D^{(n/2-1,1)})$$ 
for $i=1,2$ give rise to non-isomorphic irreducible representations (of degree $n/2-1$). This implies that 
\begin{equation}\label{eq-tworep}
  \pi_1\mbox{ and }\pi_2\mbox{ induce two distinct 2-transitive permutation characters of }G \cap B.
\end{equation}

We also note that both $\pi_1$ and $\pi_2$ are faithful. Indeed, if $\Ker(\pi_i) \neq 1$ for some $i$, then $\Ker(\pi_i) \geq \soc(G) = S$. 
Since $G$ interchanges $\pi_1$ and $\pi_2$, it follows that $S \leq \Ker(\pi_{3-i})$, whence $S$ acts trivially on $\{1,2, \ldots ,n\}$, a contradiction.

Now we can go over the list of $2$-transitive permutation groups of odd degree $n/2$ with socle $S$, 
e.g. in \cite[Table I]{Mortimer}. Then \eqref{eq-out} rules out the cases $S = M_{11}$, $M_{23}$, and $^2 B_2(q)$. If
$(S,n/2) = (\AAA_m,m \geq 5)$, then, since $|\Out(S)|=2$, we must have that $G \cong \SSS_m$ and $G \cap B = \AAA_m$, which has 
a unique $2$-transitive permutation character of degree $m$, violating \eqref{eq-tworep}. Likewise, if $(S,n/2) = (PSL_2(11),11)$ or 
$(\AAA_7,15)$, then again $|\Out(S)| = 2$, and $G \cap B = S$ has a unique $2$-transitive permutation character of degree $n/2$, a contradiction.

Consider the cases $(S,n/2) = (PSL_2(q),q+1)$ or $(PSU_3(q),q^3+1)$. In these cases, $2|q$ as $n/2$ is odd. If $S_1$ and $G_1$ denote the stabilizer of 
$1$ in $S$, respectively in $G \cap B$, then it is easy to see that $S_1 = \NB_S(Q)$ and $Q = \OB_2(S_1) \unlhd G_1$. As 
$$[G \cap B:G_1] = n/2=[S:S_1],$$ 
by Frattini argument we have $G_1 = \NB_{G_1}(Q)$. The same argument also applies to the stabilizer of $n/2+1$ in $G \cap B$. Thus the 
$2$-transitive representations of $G \cap B$ induced by $\pi_1$ and $\pi_2$ are in fact $G \cap B$-conjugate and so have the same character,
contradicting \eqref{eq-tworep}. 

Finally, consider the case $(S,n/2) = (PSL_d(q),(q^d-1)/(q-1))$ with $d \geq 3$. In this case, the $2$-transitive permutation action $\pi_i(S)$ extends
to $P\Gamma L_d(q)$, but not to the entire $\Aut(S) \cong P\Gamma L_d(q) \rtimes \CCC_2$ (where $\CCC_2$ is generated by the 
inverse-transpose automorphism $\tau$). 
As $\pi_i\da_S$ extends to $G \cap B$, $G \cap B \leq P\Gamma L_d(q)$. We may assume that the stabilizer $S_1$ of 
$1$ in $S$ is the stabilizer of a fixed one-dimensional subspace in the natural module $\FF_q^d$ for $SL_d(q)$. Then 
$Q:= \OB_r(S_1)$ is an elementary abelian $r$-subgroup of order $q^{d-1}$, if $r$ is the prime dividing $q$. Note that $P\Gamma L_d(q)$ preserves the 
$S$-conjugacy classes of $Q$, and so 
$$[G \cap B:\NB_{G \cap B}(Q)] = [S:\NB_S(Q)].$$
Arguing as in the previous case, we obtain that the stabilizer $G_1$ of $1$ in $G \cap B$ is precisely $\NB_{G \cap B}(Q)$. Thus the representation
$\pi_1$ of $G \cap B$ is uniquely determined once we fix (the $S$-conjugacy class of) $Q$, whence it must be the restriction to $G \cap B$ of the usual 
action of $P\Gamma L_d(q)$ on $1$-spaces of $\FF_q^d$, with character say $\psi$. Clearly, 
$\psi(g)$ is the number of $g$-invariant $1$-spaces on $\FF_q^d$ for all $g \in P\Gamma L_d(q)$. Note that $S$ has only one more $2$-transitive
representation that is not equivalent to $\pi_1\da_S$, namely the one on hyperplanes of $\FF_q^d$, which extends to the usual action of $P\Gamma L_d(q)$ on hyperplanes of $\FF_q^d$, with character say $\psi'$. Again,
$\psi'(g)$ is the number of $g$-invariant hyperplanes on $\FF_q^d$ for all $g \in P\Gamma L_d(q)$. Now, $\psi' = \psi^\tau$, and $\psi$ is $\tau$-invariant
by the proof of \cite[Lemma 6.2]{T}. It follows that $\psi' = \psi$. As the $2$-transitive permutation character of $G \cap B$ induced by $\pi_2$ is either 
$\psi\da_{G \cap B}$ or $\psi'_{G \cap B}$, we see that $\pi_1$ and $\pi_2$ induce the same permutation character, again violating \eqref{eq-tworep}.

\subsection{Proof of Theorem B$'$}
%\begin{proof}[Proof of Theorem C$'$]
Inspect the list of exceptions in Theorem~A$'$ for almost simple groups.

\subsection{Proof of Theorem C}
The first statement of the theorem and the `if' part of the second statement is \cite[Theorem C]{KMT1}, but see Remark~\ref{R180319_2}. For the `only-if' part of the second statement, in view of part (iii) of the first statement, we may assume that $G$ is not primitive. By \cite[Table III]{Wales}, $D^{\be_n}$ is obtained by reducing modulo $2$ a basic spin representation $B_\C$ of $\hat\SSS_n$. So $B_\C\da_G$ is irreducible. By \cite[Theorem C]{KT} we have that $$G\in\{\SSS_{n-1},\AAA_{n-1},\SSS_{n-2},\AAA_{n-2,2}\}.$$ 
If $D^{\be_n}\da_{\AAA_{n-1}}$ is irreducible then $D^{\be_n}\da_{\AAA_n}$ 
and $D^{\be_{n-1}}\da_{\AAA_{n-1}}$ must be irreducible, which is impossible. The cases $G=\SSS_{n-2}$ and $\AAA_{n-2,2}$ can be also ruled out, since by part (i) of the first statement of the theorem, we have that $D^\la\da_{\SSS_{n-2,2}}$ is reducible. 
%If $D^\la\da_G$ is irreducible and $G=\AAA_n$ we are in case (xiii)(a), while i
Finally, if $G=\SSS_{n-1}$ we apply part (i) of the first statement of the theorem  to arrive to part (1) of the second statement.

\subsection{Proof of Theorem C$'$}
For the first statement of the theorem, taking into account  \cite[Propositions 6.3, 6.6, 6.7]{KMT3}, which deal with irreducible restrictions of basic spin modules to the subgroups of the from $\AAA_n\cap(\SSS_{n-k}\times\SSS_k)$ and $\AAA_n\cap(\SSS_{a}\wr\SSS_b)$,  we may assume that $G$ is primitive. If $n\equiv 2\pmod{4}$ then $\be_n\not\in\Parinv_2(n)$, so in this case the first statement follows from Theorem C. So we may also assume that $n\not\equiv 2\pmod{4}$, in which case $\be_n\in\Parinv_2(n)$. By Theorems \ref{ext-spin}, \ref{T2small} and \ref{TAGL}, if $E^{\be_n}_{\pm}\da_G$ is irreducible then we are in one of the exceptional cases (A7)-(A12) listed in Theorem C$'$(iii).
Conversely, the cases (A11),(A12) give rise to examples by Theorem \ref{T2small}; the cases (A7),(A8) occur by Theorem  \ref{TAGL}; the case (A9) occurs by Theorem \ref{ext-spin} (and the case (A10) is covered by Theorem C since in this case $n\equiv 2\pmod{4}$).
%\end{proof}

For the second statement, the `if' part follows from the first statement, and \cite[Proposition 6.3]{KMT3}. 

We finally prove the `only-if' part of the second statement. 
In view of Theorem C, we may assume that $\be_n\in\Parinv_2(n)$, i.e. $n\not\equiv 2\pmod{4}$. As in the proof of 
Theorem B, we have that $S \unlhd G \leq \Aut(S)$ for a non-abelian simple group $S$. By the case (iii) of the first statement, we may also assume that $G$ is not primitive. 

Since $\dim V=2^{\lfloor(n-1)/2\rfloor-1} \geq 2^{(n-4)/2}$, we have that $|\Aut(S)| \geq |G| \geq 2^{n-4}$. Now we apply \cite[Proposition 6.1]{KlW} and consider the possible cases for $G$ listed there. If we are in one of the cases listed in
Proposition \ref{PSmall2}, then we arrive at the exceptional cases covered in part (1). So we may 
assume that $S = \AAA_m$, with $m \geq 7$ and each orbit of $S$ on $\Omega:=\{1,2, \ldots,n\}$ having length $1$ or $m$. 

Let $\Om_1,\dots,\Om_a$  be the $S$-orbits of length $m$ so that $S$ fixes $b:=n-am$ points in 
$$\Omega':=\Om\setminus(\Om_1\cup\dots\cup \Om_a).$$
Let $\pi_i$ denote
the permutation action of $S$ on $\Omega_i$, and also of $G$ on $\Omega_i$ in the case $G$ stabilizes $\Omega_i$. Let 
$\SSS(\Omega_i) \cong \SSS_m$ and $\AAA(\Omega_i) \cong \AAA_m$ denote the natural subgroups of $\SSS_n$ that act only
on $\Omega_i$.

%Indeed, assume the contrary and consider a $3$-cycle $t \in S$. For each $i$, $\pi_i(t)$ is a $3$-cycle on $\Omega_i$ and so it acts quadratically on $V\da_{\AAA(\Omega_i)}$, see \cite[Theorem 8.1]{Wales}. 
Restricting $V$ to $\prod^a_{i=1}\AAA(\Omega_i)$, 
% and using Lemma \ref{Lcycle}, 
we see that $V\da_S$ contains 
a submodule 
$$U:=V_1 \otimes V_2 \otimes \ldots \otimes V_a \otimes X,$$ 
where $V_1,\dots,V_a$ are basic spin modules of $S$. If $a \geq 3$, then, as $S$ has 
at most two non-isomorphic basic spin modules, we may assume $V_1 \cong V_2$ and note that $\dim V_i \geq 2$.
The same holds if $a=2$ and $S$ has a unique basic spin module, i.e. $m\equiv 2\pmod{4}$. Thus in either case $U$ has 
a proper submodule 
$\Sym^2(V_1) \otimes V_3 \otimes \ldots \otimes V_a \otimes X$ 
of dimension greater than $(\dim U)/2$. Hence $V\da_S$ has a nonzero subquotient of
dimension less than $(\dim V)/2$, contradicting to the irreducibility of $G$ on $V$. We deduce that $a \leq 2$, and if $a=2$ then 
$S$ has two basic spin modules, i.e. $a=2$ implies 
$m\not\equiv 2\pmod{4}$. 

Let $c$ denote the number of $G$-orbits on $\Omega$. Since $|G/S| \leq 2$ we have that 
$$c \geq \lfloor (a+1)/2 \rfloor + \lfloor (b+1)/2 \rfloor.$$
By \cite[Proposition 6.3]{KMT3}, $c \leq 3$, which implies that $b \leq 4$. If $b=4$, then $G$ must have two orbits of length 
$2$ on the set $\Omega'$ of $S$-fixed points, contradicting \cite[Proposition 6.3]{KMT3}. Suppose $b=3$. Then $G$ must have two orbits of length
$2$ and $1$ on $\Omega'$, so \cite[Proposition 6.3(1)]{KMT3} implies that $4\mid n$, $c=3$, $m=n-3$, and also 
$G  = \langle \AAA_m, h \rangle \cong \SSS_m$. Since $h$ does not centralize $S$, $h$ must act non-trivially, in fact as an odd permutation on $\Omega_1$.
As $G$ has three orbits of length $n-3$, $2$, and $1$ on $\Omega$, we see that $G = \AAA_{n-3,2,1}$. 
%Conversely, if $4|n$ and $G = \AAA_{n-3,2,1}$, then by Lemma \ref{Lcycle}, all composition factors of $V\da_G$ are basic spin, of dimension $2^{(n-4)/2} = \dim V$, whence $V\da_G$ is irreducible, as stated in (h1). 

Assume now that $b=2$ and $G$ fixes the two points of $\Omega'$. Then $c=3$, and so $4|n$ by 
\cite[Proposition 6.3(1)]{KMT3}. If $m=n-2$, we have arrived at the second case of Theorem C$'$(1)(a). 
%then, as $\dim E^{\be_{n-2}} = 2^{(n-4)/2} = \dim V$, we see that  $V$ is irreducible over $\AAA_{n-2}$, and arrive at (h2). 
As $4|n$, the restriction of $V$ to $\AAA(\Omega_1 \cup \Omega_2) \cong \AAA_{n-2}$ is $E^{\be_{n-2}}$, which extends to 
$D^{\be_{n-2}}$. Hence $V\da_S$ contains a subquotient 
$$D^{\be_{n-2}}\da_{\pi_1(S) \times \pi_2(S)} \cong D^{\be_m}\da_{\pi_1(S)} \otimes D^{\be_m}\da_{\pi_2(S)}.$$
Note that all embedding $\AAA_m \to \SSS_m$ are $\SSS_m$-conjugate. It follows that 
$$X := D^{\be_m}\da_{\pi_1(S)} \cong D^{\be_m}\da_{\pi_2(S)}$$
as $\F S$-modules, of dimension $e \geq 4$. Hence $V\da_S$ contains subquotients $\Sym^2(X)$ and $\wedge^2(X)$ of distinct dimensions,
contradicting the irreducibility of $G$ on $V$.

Consider the case $b=2$ and $\Omega'$ forms a $G$-orbit of length $2$. Recall that $a \leq 2$. Now if $c=3$, then
$4|n$ by \cite[Proposition 6.3(1)]{KMT3}, $n=2m+2$, $G \cong \SSS_m$, and we can repeat the above argument with $\Sym^2/\wedge^2(X)$ 
to reach a contradiction. Suppose $c=a=2$. Then $G$ has orbits of length $2$ and $n-2=2m$ on $\Omega$, whence $4|n$ as $n\not\equiv 2\pmod{4}$ by assumption. Then we can again repeat the above argument with 
$\Sym^2/\wedge^2(X)$. So we must have $a=1$, $n = m+2$, $G = \langle S,h \rangle \SSS_m$. Again, since $[h,S] \neq 1$, we must have 
that $h$ acts non-trivially on $\Omega_1$ (and on $\Omega'$), and so $G = \AAA_{n-2,2}$, and we have arrived at the case (1)(b) of Theorem C$'$.
%leading to (h3) by \cite[Proposition 6.3(3)]{KMT3}.

Now assume that $b=1$. As $G \leq \AAA_{n-1}$ is irreducible on $V$, by \cite[Proposition 6.3]{KMT3} we have 
$n \equiv 0,3 \pmod{4}$. If $c=3$, then, since $a \leq 2$, we have that $G$ has three orbits of length $m$, $m$, and $1$ on $\Omega$, but this 
contradicts \cite[Proposition 6.3(1)]{KMT3}. Suppose $c=a=2$, so that $n=2m+1 \equiv 3 \pmod{4}$. In this case, the restriction of 
$V$ to $\AAA(\Omega_1 \cup \Omega_2) \cong \AAA_{n-1}$ is $E^{\be_{n-1}}$, which extends to 
$D^{\be_{n-1}}$. Now we can repeat the argument with $\Sym^2/\wedge^2(X)$ to reach a contradiction. Thus $a=1$, $n = m+1$, 
$G = \AAA_m$, and we and we have arrived at the case (1)(c) of Theorem C$'$.

Finally, we consider the case $b=0$. As $n > m$ and $a \leq 2$, we must have that $a=2$, $n = 2m \equiv 0 \pmod{4}$. By \cite[Proposition 6.3]{KMT3}
for $c=2$ (where $G \leq \AAA_{m,m}$) and \cite[Proposition 6.6]{KMT3} for $c=1$ (where $G \leq \SSS_{m}\wr\SSS_2$), we have $m \equiv 2 \pmod 4$, which contradicts what was proved above.

\end{document}